\documentclass[11pt]{amsart}
\usepackage{amssymb}
\usepackage{hyperref}
\usepackage{amsmath,amsthm,amsxtra,mathtools}
\usepackage{eucal,mathrsfs}
\usepackage{enumitem}
\usepackage{a4wide}
\usepackage{xcolor}
\usepackage{soul}
\usepackage{comment}
\usepackage{bbm,stix}
\usepackage{subcaption}
\usepackage{wrapfig}
\usepackage{multicol}

\usepackage{tikz}
\usetikzlibrary{cd}

\newcommand{\ssubset}{\subset\joinrel\subset}

\newcommand{\N}{\mathbb{N}}
\newcommand{\R}{\mathbb{R}}
\newcommand{\Z}{\mathbb{Z}}
\newcommand{\Ind}{\mathbbm{1}}
\newcommand{\Id}{\text{I}_d}
\newcommand{\T}{\mathbb{T}}
\newcommand{\bbL}{\mathbb{L}_h}
\newcommand{\bbP}{\mathbb{P}_h}

\newcommand{\calB}{\mathcal{B}}
\newcommand{\calC}{\mathcal{C}}
\newcommand{\calD}{\mathcal{D}}
\newcommand{\calE}{\mathcal{E}}
\newcommand{\calF}{\mathcal{F}}

\newcommand{\calI}{\mathcal{I}}

\newcommand{\calL}{\mathscr{L}}
\newcommand{\calM}{\mathcal{M}}
\newcommand{\calN}{\mathcal{N}}
\newcommand{\calO}{\mathcal{O}}
\newcommand{\calP}{\mathcal{P}}

\newcommand{\calR}{\mathcal{R}}
\newcommand{\calT}{\mathcal{T}}

\newcommand{\supp}{\text{supp}}
\newcommand{\Lip}{\text{Lip}}

\newcommand{\ddiv}{\overline{\text{div}}\,}
\newcommand{\dnabla}{\overline{\nabla}}

\newcommand{\dd}[1]{\mathop{}\!\mathrm{d} #1}

\DeclareMathOperator{\Ent}{\text{Ent}}

\newtheorem{mainthm}{Theorem}

\newtheorem{theorem}{Theorem}[section]

\newtheorem{cor}[theorem]{Corollary}
\newtheorem{lemma}[theorem]{Lemma}
\newtheorem{proposition}[theorem]{Proposition}

\theoremstyle{definition}
\newtheorem{definition}[theorem]{Definition}
\newtheorem{remark}[theorem]{Remark}
\newtheorem{example}[theorem]{Example}

\numberwithin{equation}{section}

\title{Diffusive limit of random walks on tessellations via generalized gradient flows}
\author{Anastasiia Hraivoronska, Oliver Tse}
\date{\today}

\begin{document}
\maketitle

\begin{abstract}
    We study asymptotic limits of reversible random walks on tessellations via a variational approach, which relies on a specific generalized-gradient-flow formulation of the corresponding forward Kolmogorov equation.
    We establish sufficient conditions on sequences of tessellations and jump intensities under which a sequence of random walks converges to a diffusion process with a possibly spatially-dependent diffusion tensor.
\end{abstract}


\section{Introduction}\label{sec_intro}

In this paper, we are interested in the limiting behavior of random walks on graphs corresponding to tessellations 
in the diffusive regime, known as the {\em diffusive limit}. A well-known example of such convergence is that of random walks on lattices to the Brownian motion (for instance, as a consequence of Donsker's theorem \cite[Theorem~14.1]{billingsley1999convergence}). Many generalizations of Donsker's theorem
have appeared in the literature, including scaling limits of the random conductance model \cite{andres2021quenched, biskup2011recent}, limit theorems for percolation clusters \cite{hambly2009parabolic, kumagai2014random}, diffusion limits for continuous-time random walks \cite{meerschaert2004limit, sandev2018continuous}, Brownian motion as a limit of deterministic dynamics \cite{kotelenez2005discrete}, and others \cite{caputo13, croydon2008local, varadhan1997diffusive}. 
The techniques used in these references are mainly probabilistic, and the underlying state space is usually the lattice or $\R^d$. On the other hand, not much is known about diffusive limits of random walks on general geometric graphs and tessellations. This paper aims to contribute to filling this gap by exploiting modern variational techniques.

Recently, there has been renewed interest in studying such limits for families of tessellations from the viewpoint of numerical schemes,
for instance, finite-volume methods \cite{barth2018finite,droniou2018,eymard2000finite} and flux discretization schemes \cite{eymard2006finite, heida2018convergences, kantner2020generalized} for parabolic equations such as the Fokker--Planck equation (see below \eqref{eq_diffusion}). These methods are known to converge for a restrictive class of tessellations. From the variational perspective, an approach similar to ours was used in \cite{disser2015gradient} (one-dimension) and \cite{forkert2020evolutionary} (multi-dimension) to prove convergence of the finite-volume method for the Fokker--Planck equation. 


The goal of this paper is twofold: (a) to provide sufficient conditions 
on the family of tessellations and transitions intensities of the random walk such that diffusive limits exist, and (b) to study the impact of these assumptions on the limit process. We believe that the outcome and the methodology used in this work can help with making advances in e.g.\ proving the convergence of more general numerical schemes, and in studying evolution equations in random environment.

\begin{figure}[h]
\begin{subfigure}{0.24\textwidth}
\includegraphics[width=0.95\linewidth, height=3.5cm]{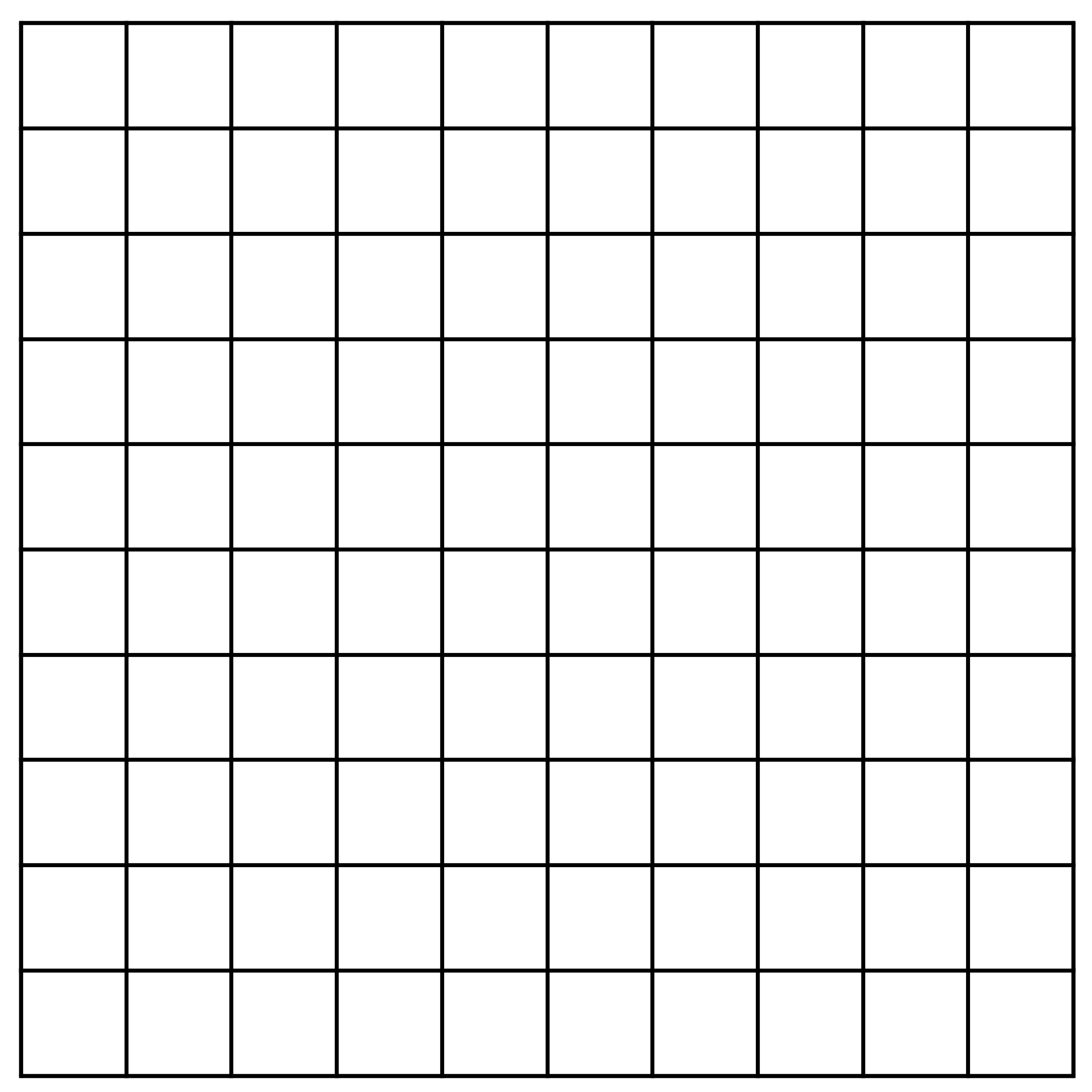} 
\label{fig_lattice}
\end{subfigure}
\begin{subfigure}{0.24\textwidth}
\includegraphics[width=0.95\linewidth, height=3.5cm]{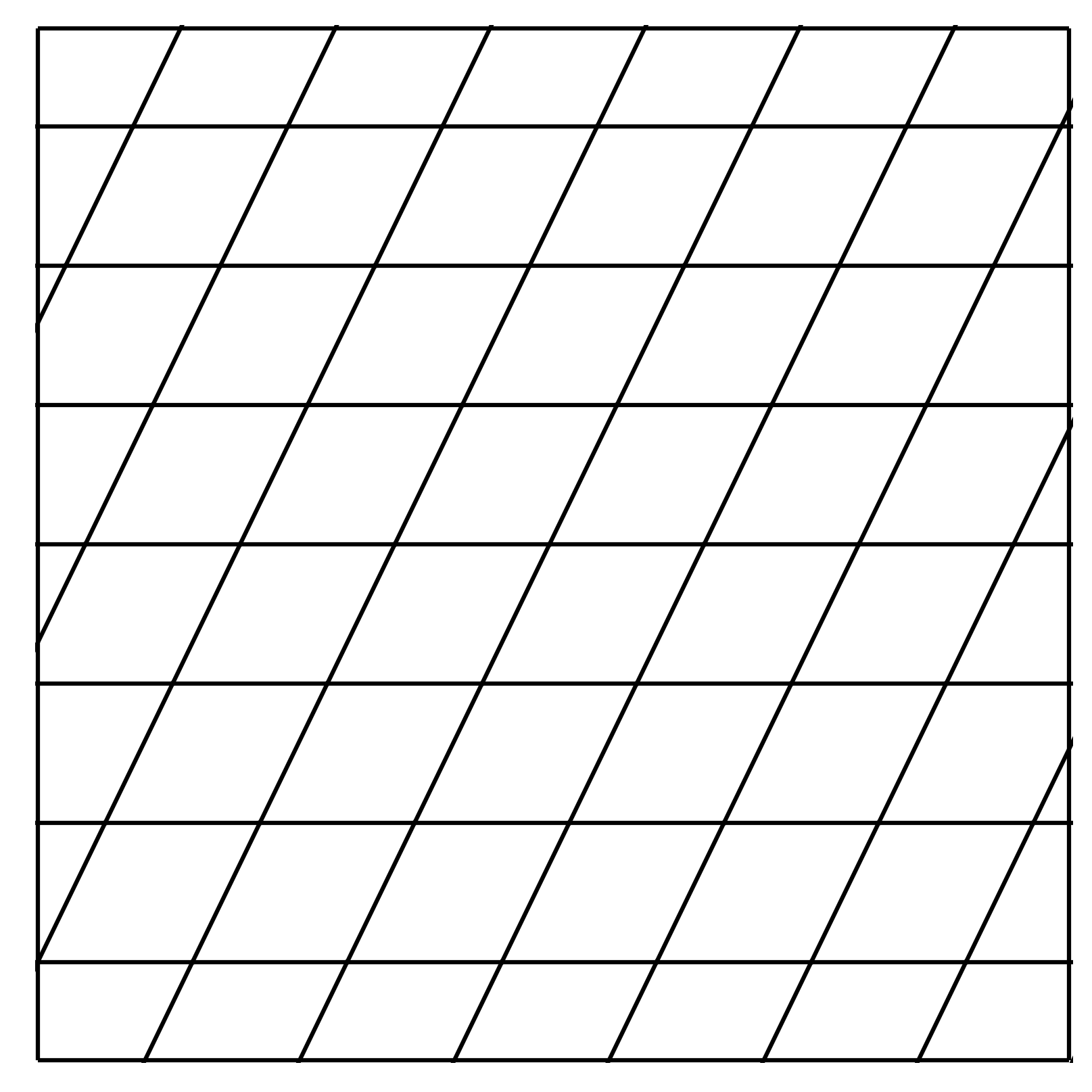}
\label{fig_tilted}
\end{subfigure}
\begin{subfigure}{0.24\textwidth}
\includegraphics[width=0.95\linewidth, height=3.5cm]{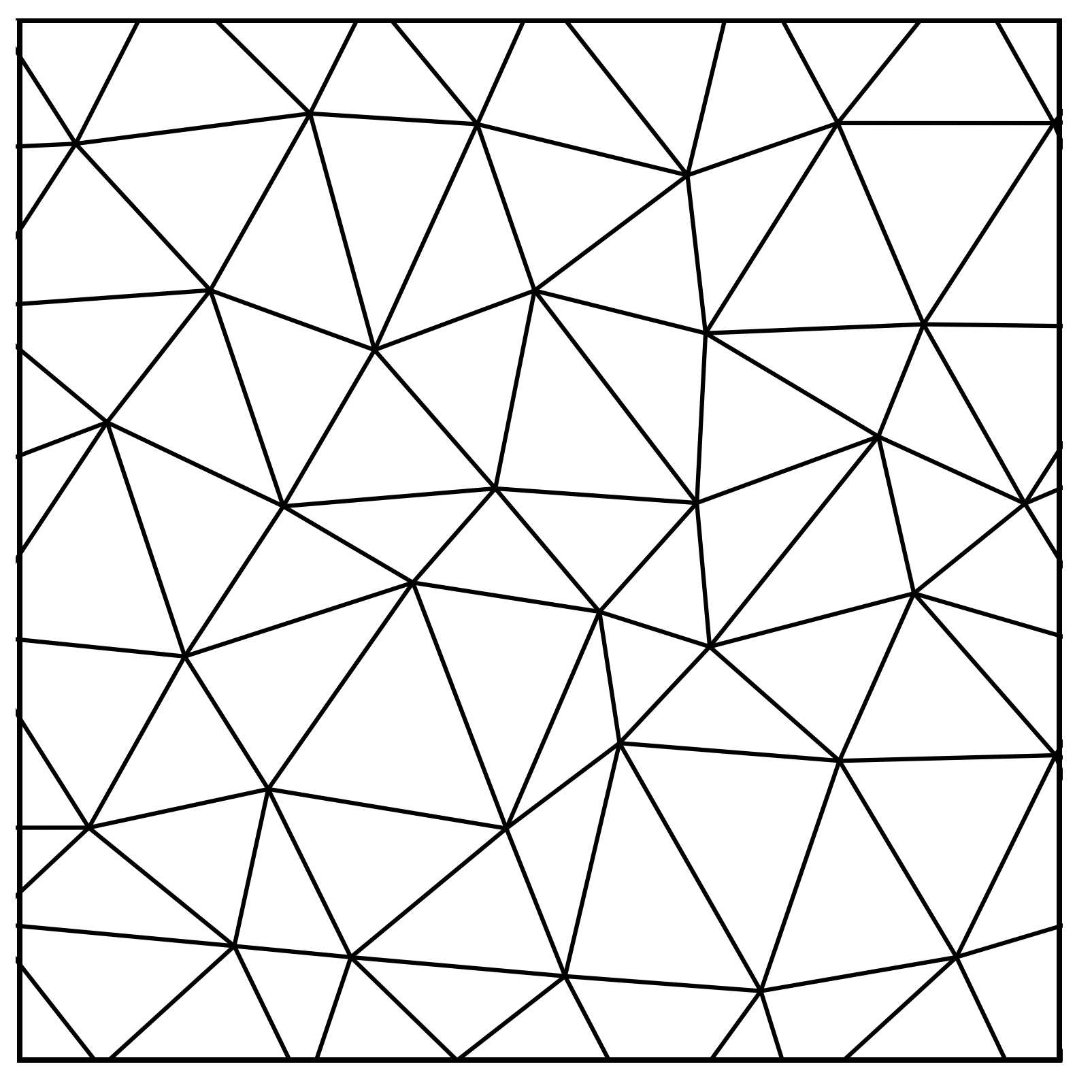} 
\label{fig_triang}
\end{subfigure}
\begin{subfigure}{0.24\textwidth}
\includegraphics[width=0.95\linewidth, height=3.5cm]{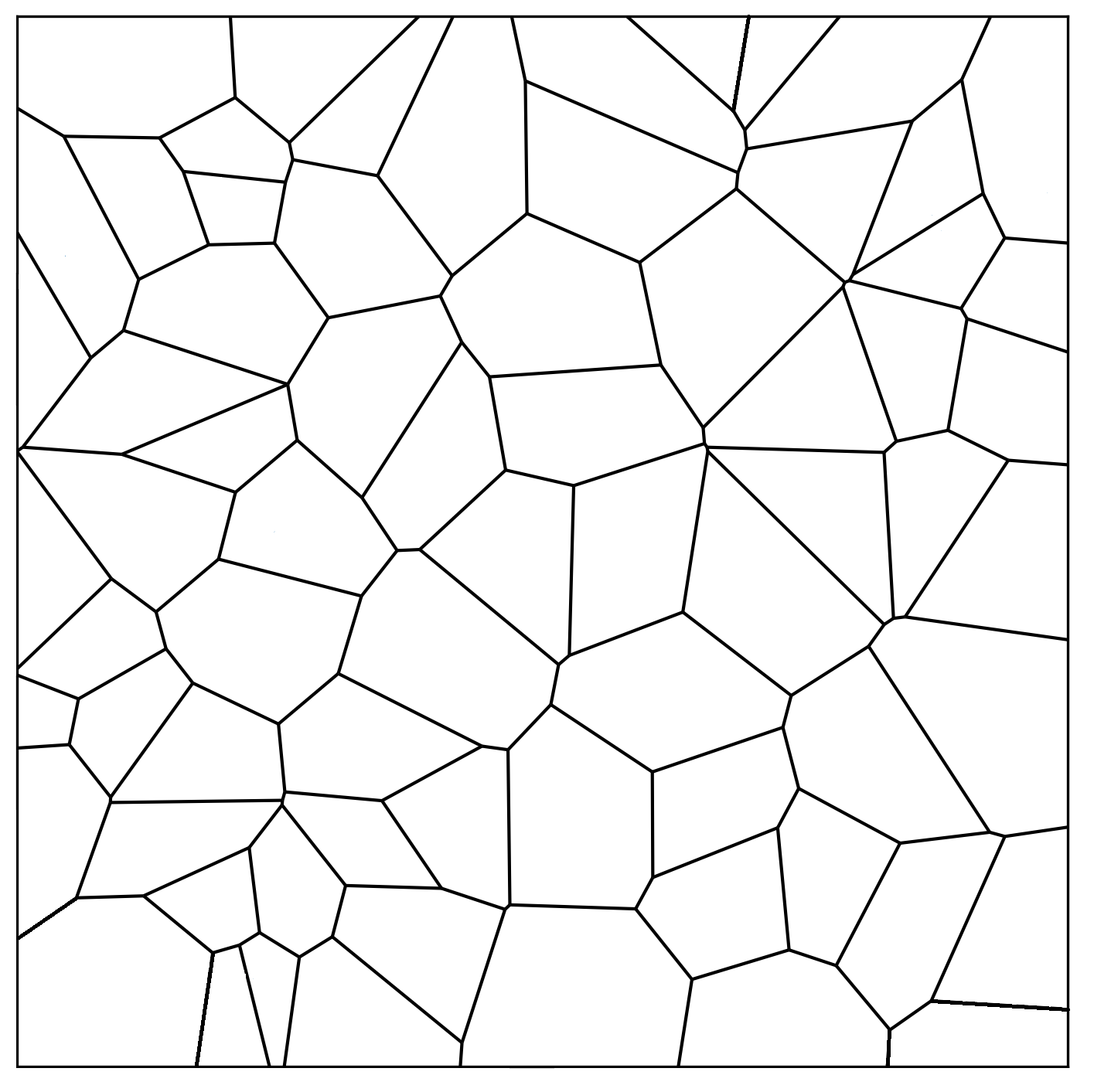} 
\label{fig_voronoi}
\end{subfigure}
\caption{Examples of tessellations.}
\label{fig_tessel}
\end{figure}

\medskip

To clarify the goal, we now introduce our setup: Let $\{(\calT^h,\Sigma^h)\}_{h>0}$ be a family of finite tessellations of a bounded and convex set $\Omega\subset\R^d$, where $\calT^h$ is the family of {\em cells} and $\Sigma^h\subset\calT^h\times\calT^h$ is identified with the set of {\em faces} (cf.\ Section~\ref{sec_tessellations} for the precise definition of cells and faces), and $\kappa^h:\Sigma^h\to[0,\infty)$ be transition kernels for the random walk. Examples of tessellations are shown in Figure~\ref{fig_tessel}. 
The small parameter $h>0$ stands for the characteristic size of the tessellation, i.e. the maximal diameter of the cells. The (time) marginal law of the random walk with initial law $\bar\rho^h$ is known 
to satisfy the {\em forward Kolmogorov equation} (see for example \cite[Section~6.3]{schuette2009markov})
\begin{equation}\label{eq_Kolmogorov}\tag{fK$_h$}
	\partial_t \rho_t^h = Q^*_h \rho_t^h,\qquad \rho_0^h = \bar\rho^h,
\end{equation}
with $Q^*_h$ being the dual of the generator $Q_h$ given, for any bounded function $\varphi \in B(\calT^h)$, by
\[
    (Q_h \varphi) (K) = \sum_{L \in \calT_K^h} [\varphi(L)-\varphi(K)] \kappa^h(K, L), \qquad  K\in\calT^h,
\]
where the sum is taken over all elements in the set of adjacent cells $\calT^h_K = \{ L\in\calT^h: (K, L) \in \Sigma^h \}$. Here, we restrict ourselves to random walks satisfying {\em detailed balance}, i.e.\ random walks admitting a stationary measure $\pi^h\in\calP(\calT^h)$ such that
\[
    \pi^h(K) \kappa^h(K, L) = \pi^h(L) \kappa^h(L,K) \qquad \forall (K, L) \in \Sigma^h.
\]

The questions we set out to answer are then: 
\begin{enumerate}[label=(\alph*)]
    \item Under which sufficient conditions on $\{(\calT^h,\Sigma^h)\}_{h>0}$ and $\{\kappa^h\}_{h>0}$ does the family of solutions $\{t\mapsto\rho_t^h\}_{h>0}$ to \eqref{eq_Kolmogorov} converge 
    to a non-degenerate diffusion process?
    \item What equation does the limiting object satisfy?
\end{enumerate}
For this purpose, we employ a variational formulation based on a generalized gradient structure for the forward Kolmogorov equation $\eqref{eq_Kolmogorov}$, which we describe briefly in the following (see Section~\ref{sec_generalized_gf} for more details).

\medskip

A gradient structure is completely defined by the \emph{driving energy} $\calE_h: \calP(\calT^h) \to [0, +\infty]$ and the \emph{dual dissipation potential} $\calR^*_h: \calP(\calT^h) \times \calB(\Sigma^h) \to [0, +\infty]$. Here, $\calP(X)$ and $\calB(X)$ denote the spaces of probability measures and bounded measurable functions on $X$ respectively. 
In the case of the random walk, the evolution is driven by the relative entropy with respect to the stationary measure $\pi^h$:
\begin{equation*}
    \calE_h(\rho^h) := \Ent (\rho^h | \pi^h) = 
    \begin{cases}
        \displaystyle \sum_{K\in\calT^h} \phi\bigl( u^h(K) \bigr) \pi^h(K) & \text{if } \rho^h \ll \pi^h \text{ with } u^h := \dfrac{\dd \rho^h}{\dd \pi^h}, \\
       +\infty & \text{otherwise,}
    \end{cases}
\end{equation*}
with the energy density $\phi(s) = s \log s - s + 1$. As for the dual dissipation potential $\calR_h^*$, a range of choices can give rise to a gradient structure for \eqref{eq_Kolmogorov}. The general form of $\calR^*_h$ studied in \cite{peletier2022jump} is
\begin{equation*}
    \calR^*_h(\rho^h, \xi^h) = \frac{1}{2} \sum_{(K, L)\in\Sigma^h} \Psi^* \bigl( \xi^h(K, L) \bigr) \alpha(u^h(K), u^h(L)) \kappa^h(K, L) \pi^h(K).
\end{equation*}
In this work, we make use of the so-called `cosh' gradient structure, for which 
\[
    \Psi^*(\xi) = 4 (\cosh(\xi/2) - 1)\quad\text{and}\quad \alpha(u, v) = \sqrt{u v}.
\]
This choice first appeared in \cite{grmela1993weakly}, was later derived from the large-deviation characterization in \cite{mielke2014relation}, and received significant attention in literature thereafter.

Another well-studied choice is the quadratic gradient structure. In particular, the quadratic structure was used to prove the convergence for the finite-volume discretization of the Fokker--Planck equation in \cite{forkert2020evolutionary}. In comparison, the adoption of the cosh-type gradient structure allows us to consider a more general class of tessellations, including a tilted $\Z^d$ tessellation (Example~\ref{example_tilted}), and we can dispense with the orthogonality assumption used in \cite{forkert2020evolutionary} and the finite-volume methods mentioned above.


\medskip

With the introduced $\calE_h$ and $\calR^*_h$, one can express \eqref{eq_Kolmogorov} in the form of a {\em continuity equation} \eqref{eq_CE_discrete} and the {\em force-flux relation} \eqref{eq_FF_discrete} for the density-flux pair $(\rho^h, j^h)$:
\begin{align*}
    \partial_t \rho^h_t + \ddiv j^h_t &= 0\qquad \text{on } (0,T)\times\calT^h. \tag{CE$_h$}\label{eq_CE_discrete} \\
    j^h_t &= D_2 \calR^*_h \bigl(\rho^h_t, -\dnabla \calE_h'(\rho^h_t)\bigr), \tag{FF$_h$}\label{eq_FF_discrete}
\end{align*}
where the flux $j^h$ is written in terms of $\calR^*_h$ and $\calE_h$ acting on $\rho^h$. Here, $D_2$ denotes the derivative in the second variable, $\dnabla$ is the graph gradient and $\ddiv$ is the graph divergence (defined in Section~\ref{sec_generalized_gf}). Using Legendre--Fenchel duality, one obtains a variational characterization of the solutions to \eqref{eq_FF_discrete} given by \emph{the energy-dissipation balance}, i.e.\ for any $T>0$, the pair $(\rho^h,j^h)$ satisfies
\begin{equation}\label{eq_edb_discrete}\tag{EDB$_h$}
    \calI_h(\rho^h, j^h) := \int_0^T \calR_h(\rho^h_t, j^h_t) + \calR_h^*\bigl(\rho^h_t, -\dnabla\calE_h'(\rho^h_t)\bigr) \,\dd t + \calE_h(\rho^h_T) - \calE_h(\rho^h_0) = 0.
\end{equation} 
with $\calR_h$ and $\calR_h^*$ being Legendre--Fenchel duality pairs w.r.t.\ the second variable. 

When the following chain rule applies for all pairs $(\nu^h,\eta^h)$ satisfying \eqref{eq_CE_discrete}
\begin{align}\label{eq_CR_discrete}\tag{CR$_h$}
    \frac{\dd}{\dd t}\calE_h(\nu_t^h) =  \langle \eta_t^h,\dnabla \calE_h'(\nu_t^h)\rangle\qquad\text{for almost every $t\in(0,T)$},
\end{align}
then one also has that $\calI_h(\nu^h, \eta^h) \geq 0$. In particular, a pair $(\rho^h,j^h)$ satisfying \eqref{eq_CE_discrete} and \eqref{eq_edb_discrete} is a minimizer of $\calI_h$,
which we use to define generalized gradient flow (GGF) solutions to \eqref{eq_Kolmogorov}:
\begin{align*}
    \left.\begin{gathered}
        (\rho^h, j^h)\; \text{satisfies}\\
        \text{\eqref{eq_edb_discrete}} 
    \end{gathered}\;\right\}
     \iff (\rho^h, j^h) = \arg\min \calI_h  \iff: (\rho^h,j^h) \text{ is GGF-solution of \eqref{eq_Kolmogorov}}.
\end{align*}
 
\subsubsection*{Outline of strategy} 

The variational framework described above allows us to prove the  discrete-to-continuuum convergence result in question by employing tools from the Calculus of Variations. This form of convergence is known as \emph{evolutionary $\varGamma$-convergence}. It was introduced by Sandier and Serfaty in \cite{sandier2004gamma} and led to numerous subsequent studies surveyed in \cite{mielke2016evolutionary,serfaty2011gamma}, see also \cite{mielke2021exploring} for recent developments and various forms of EDP convergence.

Our strategy comprises the following main steps:
\begin{enumerate}
    \item Prove \emph{compactness} for the family $(\rho^h, j^h)$ satisfying \eqref{eq_CE_discrete} and \eqref{eq_edb_discrete}. This allows us to extract a subsequence converging to a limiting  pair $(\rho, j)$.
    \item Prove \emph{lim\! inf inequalities} for all the functionals in the energy-dissipation functional $\calI_h$ to recover a limiting energy-dissipation functional $\calI$:
    $$
        \calI(\rho, j)\le \liminf_{h\to 0} \calI_h(\rho^h, j^h).
    $$
    \item Recover \emph{a limiting diffusion equation} of the type
    \begin{equation}\label{eq_limit_Kolmogorov}\tag{fK}
        \partial_t \rho_t = Q^* \rho_t.
    \end{equation}
    with some generator $Q$ from the limiting energy-dissipation functional $\calI$.
\end{enumerate}

Let us outline the main ideas of the steps above: {\em Step 1} raises the question of how to approach compactness in our discrete-to-continuous settings when density-flux pairs $(\rho^h, j^h)$ belong to different spaces for each $h>0$. Moreover, in the diffusive limit, we expect to obtain a curve $(\rho, j)$ satisfying the continuity equation with the usual divergence operator:
\begin{equation}\label{eq_CE}\tag{CE}
    \partial_t \rho_t + \text{div}\, j_t = 0 \qquad \text{on } (0,T)\times\R^d.
\end{equation}
For this purpose, we introduce a continuous reconstruction procedure in Section~\ref{section_reconstr_compact} such that any pair $(\rho^h, j^h)$ satisfying \eqref{eq_CE_discrete} induces a pair $(\hat{\rho}^h, \hat{\jmath}^h)$ that satisfies 
\eqref{eq_CE} \emph{exactly}. This allows us to `embed' the discrete objects into a common space of continuous objects, and serves as a link between the discrete and continuous problems (see Figure~\ref{fig_commutative_diagram}).
Another important aspect of the reconstruction procedure is that it allows us to prove a compactness result for GGF-solutions of \eqref{eq_Kolmogorov}, thereby allowing us to extract a subsequence that converges to a limiting curve $(\rho, j)$ satisfying \eqref{eq_CE}.

\begin{figure}[ht]
\begin{tikzcd}
\begin{array}{c}
\textbf{Discrete problem} \\[0.2em]
    (\rho^h,j^h) \text{ is GGF-solution of } \eqref{eq_Kolmogorov}
\end{array}\qquad
\arrow[r, dashed, start anchor={[xshift=-4ex]}, end anchor={[xshift=4ex]}]
& 
\qquad \begin{array}{c}
\textbf{Continuous limit} \\[0.2em]
     (\rho,j) \text{ is GF-solution of } \eqref{eq_limit_Kolmogorov}
\end{array} 
\arrow[dd, Leftrightarrow, "\text{\;\;\bf {\em Step 3}}"] \\
\begin{array}{c}
     (\rho^h,j^h)\in \eqref{eq_CE_discrete}: 
     \calI_h(\rho^h, j^h) = 0
\end{array} \qquad 
\arrow[u, Leftrightarrow] \arrow[d, Rightarrow, "\text{\;\;\bf {\em Step 1}}"] & \\
\begin{array}{c}
     \textbf{Continuous reconstruction} \\[0.2em]
     \text{Define } (\hat\rho^h, \hat{\jmath}^h)\in \eqref{eq_CE} :
     \hat{\calI}_h (\hat{\rho}^h, \hat{\jmath}^h) = 0
\end{array}\qquad \arrow[r, "\text{\bf {\em Step 2}}", start anchor={[xshift=-4ex]}, end anchor={[xshift=4ex]}] 
&
\qquad\begin{array}{c}
    \textbf{Continuous limit} \\[0.2em]
     (\rho,j)\in\eqref{eq_CE}:
     \calI(\rho, j) = 0
\end{array}
\end{tikzcd}
\caption{To pass to the discrete-to-continuum limits (dashed arrow in the first line), we employ the definition of the GGF-solution of \eqref{eq_Kolmogorov} and introduce continuous reconstructions of the discrete objects with $\hat{\calI}_h(\hat{\rho}^h, \hat{\jmath}^h) := \calI_h(\rho^h, j^h)$. Passing $h\to 0$ from the continuous reconstruction to the continuum limit comprises two ingredients: the compactness result and the liminf inequality for $\calI_h$. The energy-dissipation functional $\calI$ obtained in the limit gives rise to the limit equation \eqref{eq_limit_Kolmogorov}.}\label{fig_commutative_diagram}
\end{figure}
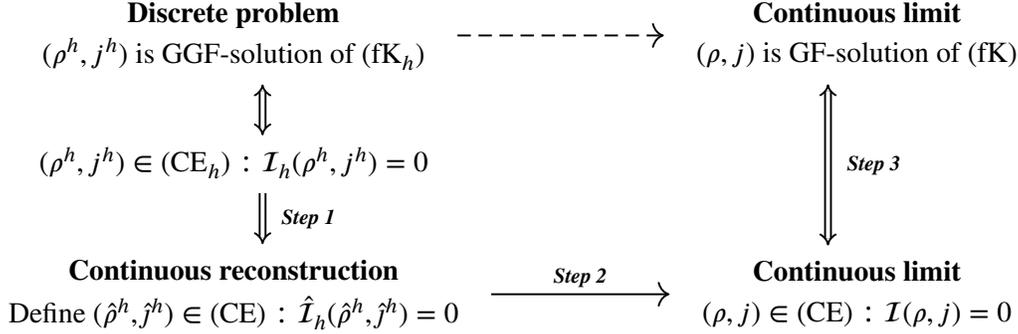

To prove the liminf inequality in {\em Step 2}, we study all the components of the energy-dissipation functional $\calI_h$ separately. The lower semicontinuity of the driving energy $\calE_h$ follows from standard result, whereas the challenging part lies in proving the result for the {\em dissipation potential} $\calR_h$ and the {\em Fisher information} $\calD_h:=\calR_h^*(\cdot, -D\calE_h(\cdot))$. For this purpose, we apply $\Gamma$-convergence techniques to obtain and characterize the variational limits $\calR$ and $\calD$ of $\{\calR_h\}_{h>0}$ and $\{\calD_h\}_{h>0}$ respectively. Here the assumptions placed on the family of tessellations $\{(\calT^h,\Sigma^h)\}_{h>0}$ and transition kernels $\{\kappa^h\}_{h>0}$ that we ask about in question (a) come into play, as the form of $\calR$ and $\calD$ depends strongly on the relationship between $\{(\calT^h,\Sigma^h)\}_{h>0}$ and $\{\kappa^h\}_{h>0}$. 

To illustrate the idea, the discrete Fisher information $\calD_h$ for the `cosh' gradient structure takes the form
\begin{equation*}
    \calD_h(\rho^h) = \sum_{(K, L)\in\Sigma^h} \left|(\dnabla \sqrt{u^h}) (K, L) \right|^2 \kappa^h(K, L) \pi^h(K) \quad \text{ with } u^h = \frac{\dd \rho^h}{\dd \pi^h}.
\end{equation*}
We prove that, under suitable assumptions on the families $\{(\calT^h,\Sigma^h)\}_{h>0}$, $\{\kappa^h\}_{h>0}$ and $\{\pi^h\}_{h>0}$ 
, the family 
$\{\calD_h\}_{h>0}$ $\Gamma$-converges to a limit functional of the form

$$
    \calD(\rho) = \int_\Omega \bigl\langle \nabla\sqrt{u},\T\nabla\sqrt{u}\bigr\rangle \dd \pi \qquad \text{with } u = \frac{\dd \rho}{\dd \pi},
$$
where $\T:\R^d\to \R^{d\times d}$ is a symmetric and positive definite tensor, and $\pi$ is the limit of the sequence $\{\pi^h\}_{h>0}$. All assumptions will be made precise in Section~\ref{sec_random_walks}, where we also formulate and state our main result. The next step in our strategy provides the interpretation of $\T$ and $\pi$ as objects related to a diffusion process on $\R^d$.


Morally, {\em Step 3} is the "reverse" procedure of formulating the forward Kolmogorov equation \eqref{eq_Kolmogorov} as the generalized gradient flow characterised by the energy-dissipation balance \eqref{eq_edb_discrete}. Once we have identified the limit energy-dissipation functional $\calI$, we can make use of classical gradient flow theory to deduce the form of the limit forward Kolmogorov equation \eqref{eq_limit_Kolmogorov}. In particular, we formally obtain the diffusion equation
\begin{equation}\label{eq_diffusion}
    \partial_t \rho_t = \text{div} \bigl(\T (\nabla\rho_t + \rho_t \nabla V) \bigr),
\end{equation}
with $V = -\log(\dd\pi / \dd\calL^d)$, thereby answering  question (b). If $\T$ arises from a homogeneous random walk on a uniform lattice (see Example~\ref{ex_lattice}), then we arrive at \eqref{eq_diffusion} with $\T=\text{Id}$, the identity tensor.


\medskip

The techniques we use to prove the $\liminf$ inequalities in \textit{Step 2} are similar to those used in \cite{forkert2020evolutionary}; however, the philosophy and results have considerable differences. The authors of \cite{forkert2020evolutionary} prove the convergence of the finite-volume discretization of the equation \eqref{eq_diffusion} with $\T=\Id$ to the original equation. We, on the other hand, start with a more general discrete evolution equation \eqref{eq_Kolmogorov} and, consequently, recover the diffusion equation \eqref{eq_diffusion} with variable coefficients $\T$.

\subsubsection*{Outline of the paper}
The paper is organized as follows. In Section~\ref{sec_random_walks}, we introduce assumptions on the sequence of tessellations and jump intensities that allow us to realize the described strategy. After that, we formulate the main result in Section~\ref{sec_main_results}. Moreover, in Section~\ref{sec_examples}, we discuss several examples that illustrate the applicability of our main result to specific families of tessellations. Section~\eqref{sec_gf} summarizes the definitions of the continuity equations and (generalized) GF-solutions of \eqref{eq_Kolmogorov} and \eqref{eq_limit_Kolmogorov}. In Section~\ref{section_reconstr_compact}, we specify the continuous reconstruction procedure and prove compactness result for the GGF-solutions of \eqref{eq_Kolmogorov}. Section~\ref{sec_gamma_convergence} is devoted to the $\Gamma$-convergence results for the Fisher information and the dual dissipation potential. Finally, we conclude with the proof of the main result in Section~\ref{sec_result}.

\subsection*{Acknowledgments} The authors acknowledge support from NWO Vidi grant 016.Vidi.189.102 $\\$ "Dynamical-Variational Transport Costs and Application to Variational Evolution".

\section{Assumptions and Main Results}\label{sec_random_walks}

In this section, we specify our 
assumptions on the families of tessellations, transition kernels, and stationary measures. After that we formulate our main result in Theorem~\ref{th_main_result}.

\subsection{Tessellations}\label{sec_tessellations} Let $\Omega\subset\R^d$ be an open bounded convex set. A tessellation $(\calT^h,\Sigma^h)$ of $\Omega$ consists of a family $\calT^h$ of mutually disjoint {\em cells} (usually denoted by $K$ or $L$) that are open and convex sets in $\Omega$, and a family $\Sigma^h$ of pairs of neighboring cells $\{ (K, L)\in\calT^h\times\calT^h : \mathscr{H}^{d-1} (\overline{K}\cap\overline{L}) > 0  \}$, where $\mathscr{H}^{d-1}$ is the $(d-1)$-dimensional Hausdorff measure. Examples of suitable tessellations include Voronoi tessellations, and meshes commonly used in finite-volume methods. The common face of $(K,L)\in \Sigma^h$ is denoted by $(K|L)$. The characterizing size of a tessellation is its maximum diameter:
$$
    h := \max_{K\in \calT^h} \left( {\text{diam}(K)} \right).
$$
The maximum diameter $h>0$ gives an upper bound on the volumes of the cells $|K|\le C_d h^d$ and faces $|(K|L)| \le C_{d-1} h^{d-1}$, where $C_d$, $C_{d-1}>0$ are universal constants depending only on the spatial dimension $d\ge 1$. In our work, it is also necessary to assume lower bounds on the volumes of the cells to prevent degeneration of cells. We make the following non-degeneracy assumption. 

\vspace{0.3cm}
\framebox{ \centering \begin{minipage}{0.92\linewidth}
\textbf{Non-degeneracy.}\label{def_zeta_regularity}
    There exist $\zeta \in (0, 1)$ such that
    \begin{enumerate}[label=(\roman*)]
        \item  For each $K\in\calT^h$ there is an inner ball $B(x_K, \zeta h) \subset K$ with $x_K = \intbar_K x \dd x$;
        \item For every $(K,L)\in\Sigma^h$ it holds that $|(K|L)| \geq \zeta h^{d-1}$.
    \end{enumerate}
\end{minipage}}
\vspace{0.3cm}

\begin{remark}\label{rem_cardinality}
The non-degeneracy assumption implies particularly that $|K| \geq C_d (\zeta h)^d$ for all $K\in\calT^h$, and also provides a uniform bound on the cardinality of neighboring cells (cf.\ \cite[Lemma~2.12]{gladbach2020scaling}):
\begin{equation*}
    C_\calN := \sup_{h>0} \sup_{K\in\calT^h} \text{card\,}\calT^h_K < \infty,
\end{equation*}
which follows from the following calculations:
    $$
        \sum_{L\in\calT^h_K} C_d (\zeta h)^d \leq \sum_{L\in\calT^h_K} |L| \leq \left| B(x_K, 2h) \right| \leq C_d (2h)^d \quad \Rightarrow \quad \text{card\,}\calT^h_K \leq \frac{2^d}{\zeta^d}.
    $$
Here, $\text{card\,}A$ is the cardinality of the set $A$.
\end{remark}

\begin{remark}
    While we closely follow the finite-volume setup when defining our tessellations, 
    we remark that in contrast to \cite{eymard2000finite, gladbach2020scaling}, we do not make the orthogonality assumption, i.e.\ requiring $x_K - x_L$ to be orthogonal to $(K|L)$ for $(K, L)\in\Sigma^h$.
\end{remark}

We now summarize the assumptions on the tessellations that is used within this paper.

\vspace{0.3cm}
\framebox{ \centering \begin{minipage}{0.92\linewidth}
\textbf{Assumptions on $\calT^h$}. We assume the family of tessellations $\{(\calT^h, \Sigma^h)\}_{h>0}$ to be such that 
\begin{align}\label{ass:tessellation}\tag{Ass$\calT$}\hspace{-1em}\left\{\quad
\begin{aligned}
    &\text{for any $h>0$ all cells $K\in\calT^h$ are open, convex, and mutually disjoint;}\\
    &\text{$\{(\calT^h, \Sigma^h)\}_{h>0}$ is non-degenerate with some $\zeta\in(0, 1)$ independent of $h$.}
    \end{aligned}\right.
\end{align}
\end{minipage}}
\vspace{0.3cm}

\subsection{Relations between jump intensities and tessellation}\label{sec_assumptions_relation} To obtain the diffusive limit, we need to properly relate the objects that define the dynamics, namely the jump kernels and the stationary measures $\{(\kappa^h, \pi^h)\}_{h>0}$, with the geometric properties of the tessellation. In this section, we emphasize all the assumptions we need to make on $\{(\kappa^h, \pi^h)\}_{h>0}$ in relation to $\{(\calT^h, \Sigma^h)\}_{h>0}$.

\vspace{0.3cm}
\framebox{ \centering \begin{minipage}{0.92\linewidth}
\textbf{Assumptions on $\pi^h$}. Let $\pi^h$ be a stationary measure for \eqref{eq_Kolmogorov} satisfying \emph{the detailed balance condition}
\begin{equation}\label{eq_detailed_balance}\tag{DB}
    \vartheta^h(K,L):=\pi^h(K) \kappa^h(K, L) = \pi^h(L) \kappa^h(L,K) \qquad \forall (K, L) \in \Sigma^h.
\end{equation}

We assume $\pi^h$ to have a density uniformly bounded from above and away from zero:
\begin{equation}\label{assumption_pi}
    0 < \pi_{\min} \le \inf_{h>0} \min_{K\in\calT^h} \frac{\pi^h(K)}{|K|} \le \sup_{h>0} \max_{K\in\calT^h} \frac{\pi^h(K)}{|K|} \le \pi_{\max} < \infty.
    \tag{B$\pi$}
\end{equation}

The continuous reconstruction $\hat\pi^h$ (cf.\ Section~\ref{section_reconstr_compact}) converges in the following sense:
$$
    \dd \hat\pi^h / \dd \calL^d \to \dd\pi / \dd\calL^d \quad \text{in } L^1(\Omega).
$$ 
We further assume that $\log (\dd\pi / \dd\calL^d) \in \Lip_b(\Omega)$.
\end{minipage}}
\vspace{0.3cm}

Without loss of generality, and for simplicity, we assume $\pi^h$ to have unit mass. 

\begin{example}
    A stationary measure $\pi^h$ satisfying the above mentioned assumptions can be obtained from a continuous measure $\pi$. In practice, the stationary measure is usually given in terms of a potential $V: \Omega\to\R$, i.e. $\pi = e^{-V} \calL^d$. In this case, we assume $V\in\Lip_b(\Omega)$ and set 
    $$
        \pi^h(K) := \pi(K) = \int_{K} e^{-V(x)} \dd x.
    $$
    Then $\pi^h$ converges to $\pi$ in the sense specified in Section~\ref{sec_recostruction}, since
    \begin{align*}
       \left\| \frac{\dd\hat\pi^h}{\dd\calL^d}- \frac{\dd\pi}{\dd\calL^d} \right\|_{L^1(\Omega)} &= \sum_{K\in\calT^h} \int_{K} \left|\intbar_{K} e^{-V(y)}\,\dd y-e^{-V(x)}\right|\,\dd x \\
       &\le \sum_{K\in\calT^h} \int_{K} \intbar_{K} \left|e^{-V(y)}-e^{-V(x)}\right|\dd y \,\dd x \\
       &\le C\sum_{K\in\calT^h} \int_{K} \intbar_{K} |y-x|\,\dd y \,\dd x \le Ch|\Omega|,
    \end{align*}
    and, therefore, $\pi^h$ satisfies the required assumptions.
\end{example}


Now we introduce scaling assumptions on the joint measure $\vartheta^h$ defined in \eqref{eq_detailed_balance}.

\vspace{0.3cm}
\framebox{ \centering \begin{minipage}{0.92\linewidth}
\textbf{Scaling of $\vartheta^h$}. We assume the existence of constants
$0<C_l<C_u<\infty$ independent of $h$:
\begin{equation}\label{assumption_nu}
    C_l \frac{|(K|L)|}{|x_L - x_K|} \leq \vartheta^h(K, L) \leq C_u \frac{|(K|L)|}{|x_L - x_K|}\qquad\forall (K, L)\in\Sigma^h.
    \tag{B$\vartheta$}    
\end{equation}
\end{minipage}}
\vspace{0.3cm}

\begin{remark}
    Combining \eqref{assumption_nu} with \eqref{assumption_pi} and the non-degenerate assumption gives rise to many possible formulations of uniform lower and upper bounds. We mention a reformulation of the upper bound that appears frequently in the proofs. 
    
    Dividing \eqref{assumption_nu} by $|K|$ and using the non-degeneracy assumption on $\calT^h$ yields
$$
    \frac{\pi^h(K)}{|K|} \kappa^h(K, L) \leq C_u \frac{|(K|L)|}{|K| |x_L - x_K|} \leq C_u \frac{\zeta h^{d-1}}{C_d (\zeta h)^{d+1}} = \frac{C_u}{C_d \zeta^d} \frac{1}{h^2}.
$$
Taking into account that $\pi^h(K) / |K| \geq \pi_{\min}$ due to \eqref{assumption_pi}, we arrive at
$$
    h^2 \sum_{L\in\calT^h_K} \kappa^h(K, L) \leq \frac{C_u}{C_d \zeta^d \pi_{\min}}(\text{card\,}\calT^h_K) .  
$$
Recalling that the non-degeneracy assumption provides a uniform upper bound on the cardinality of cells in $\calT^h_K$ (cf.\ Remark~\ref{rem_cardinality}), we obtain
\begin{equation}\label{assumption_kernel_uniform_upper_bound}
    \sup_{h>0} \sup_{K\in\calT^h} h^2 \sum_{L\in\calT^h_K} \kappa^h(K, L) \leq \frac{C_u C_\calN}{C_d \zeta^d \pi_{\min}}=:C_\kappa < \infty.
    \tag{{\sf UB}}
\end{equation}
\end{remark}


We need one final assumption on the compatibility of the joint measure $\vartheta^h$ and the geometry of the tessellation, namely the so-called \emph{zero-local-average} assumption. Intuitively, this assumption ensures that the limiting system remains a gradient flow.

\vspace{0.3cm}
\framebox{ \centering \begin{minipage}{0.92\linewidth}
\textbf{Zero-local-average}. For all cells $K\in\calT^h$ not touching the boundary, i.e.\ $\overline{K} \cap \partial\Omega = \emptyset$,
\begin{equation}\label{assumption_local_average}
    \sum_{L\in\calT^h_K} \vartheta^h(K, L) (x_K - x_L) = 0.
    \tag{{\sf A$_\text{loc}$}}
\end{equation}
\end{minipage}}
\vspace{0.3cm}

Similar assumptions to \eqref{assumption_local_average} have emerged in finite-volume schemes as explained in \cite[Section~5.2.6]{barth2018finite}, and in (stochastic) homogenization to ensure that the corrector problem has a solution \cite{faggionato2008random, gloria11, kozlov85}. Later, we will see that \eqref{assumption_local_average} is only a sufficient condition for the proofs, and can be replaced by a weaker asymptotic assumption (see \eqref{assumption_almost_minimizer} in Section~\ref{sec_fisher}). We stress that the assumptions on $\calT^h$, $\pi^h$ and scaling of $\vartheta^h$ are required in throughout this paper, but \eqref{assumption_local_average} or its asymptotic variant \eqref{assumption_almost_minimizer} are only required for the identification of the limit given in Section~\ref{sec_fisher}.

\begin{remark}
    If the tessellations $\{(\calT^h, \Sigma^h)\}_{h>0}$ are such that $(x_L - x_K) \perp (K|L)$ for all $(K, L)\in\Sigma^h$ and $\vartheta^h(K, L) = C |(K|L)| / |x_L - x_K|$, then \ref{assumption_local_average} is satisfied.
\end{remark}







\subsection{Main result}\label{sec_main_results}
\begin{definition}[Admissible continuous reconstruction]
    We call a pair $(\hat\rho^h, \hat\jmath^h)$ an {\em admissible continuous reconstruction} for a pair $(\rho^h, j^h)$ satisfying \eqref{eq_CE_discrete} if
    \begin{enumerate}[label=(\roman*)]
        \item $\hat\rho^h$ is defined by the piecewise constant reconstruction of the density:
        $$
            \frac{\dd \hat\rho^h_t}{\dd \calL^d} := \sum_{K\in\calT^h} \frac{\rho_t^h(K)}{|K|}\Ind_K;
        $$
        \item $\hat{\jmath}^h$ is such that $(\hat\rho^h, \hat{\jmath}^h)$ satisfies \eqref{eq_CE}.
    \end{enumerate}
\end{definition}

We state our main result in the following theorem.
\begin{mainthm}\label{th_main_result}
Let $\{(\calT_h,\Sigma_h)\}_{h>0}$ be a family of tessellations satisfying \eqref{ass:tessellation}. Further, let $\{(\rho^h,j^h)\}_{h>0}$ be a family of GGF-solutions to \eqref{eq_Kolmogorov} with $\{(\kappa^h,\pi^h)\}_{h>0}$ satisfying \eqref{assumption_pi}, \eqref{assumption_nu}, and \eqref{assumption_local_average}, and 
initial data $\{\bar\rho^h\}_{h>0}$ satisfying \[\sup\nolimits_{h>0} \calE_h(\bar\rho^h) < \infty.\] Then there exists a (not relabelled) subsequence of admissible continuous reconstructions $\{(\hat{\rho}^h, \hat{\jmath}^h)\}_{h>0}$ and a limit pair $(\rho,j)$ such that
\begin{enumerate}
    \item $(\rho, j)$ satisfies \eqref{eq_CE} with the density $u := \dd \rho/\dd \pi \in L^1((0, T); L^1(\Omega,\pi))$ and
        \begin{enumerate}[label=(\roman*)]
            \item  $\dd\hat{\rho}^h_t/\dd\hat\pi^h \to u$ strongly in $L^1(\Omega,\pi)$ for any $t\in [0, T]$;
            \item $\int_\cdot \hat{\jmath}^h_t \dd t \rightharpoonup^* \int_\cdot j_t \dd t$ in $\calM([0, T]\times \Omega)$.
        \end{enumerate}
        
    \item the following liminf estimate holds:
        $$
            \calI(\rho, j) \le \liminf_{h\to 0} \calI_h(\rho^h, j^h).
        $$
        
    \item $(\rho,j)$ is the gradient flow solution with the energy-dissipation functional given as
    $$
        \calI (\rho, j) = \int_0^T \int_\Omega \left|\frac{\dd j_t}{\dd \rho_t}\right|^2_{\T^{-1}} \dd \rho_t + \int_\Omega \left| \nabla \sqrt{u_t} \right|^2_{\T} \dd \pi \dd t + \calE(\rho_T) - \calE(\rho_0),
    $$
    where $\T:\R^d\to \R^{d\times d}$ is a symmetric and positive definite diffusion tensor.
    \end{enumerate}
\end{mainthm}

\begin{remark}\label{remark_limit_diffusion} The equation corresponding to the energy-dissipation balance given by $\calI(\rho, j)$ is
\begin{equation}\label{eq_diffusion2}
    \partial_t \rho_t = \text{div} \bigl(\T (\nabla\rho_t + \rho_t \nabla V) \bigr) \qquad \text{on } (0, T)\times\Omega,
\end{equation}
with the no-flux boundary conditions
$\T (\nabla\rho_t + \rho_t \nabla V) \cdot n = 0$ on $\partial\Omega$, where $V = -\log(\dd\pi / \dd\calL^d)$ is the potential corresponding to the limit stationary measure $\pi$. 
\end{remark}


To characterize the diffusion tensor $\T$, we introduce a tensor $\T^h$:
$$
    \T^h(x) = \sum_{K\in\calT^h} \Ind_K(x) \sum_{L\in\calT^h_K} \kappa^h(K, L) (x_L - x_K) \otimes (x_L - x_K),\qquad x\in\R^d,
$$
where $x_K = \intbar_K x \dd x$. Then $\T$ is obtained as a limit of $\T^h$.

\subsection{Examples}\label{sec_examples} To help with getting an intuition for our assumptions and the main result, we present several examples in which the diffusion tensor $\T$ can be calculated explicitly.

\begin{example}[Lattice $h\Z^d$]\label{ex_lattice}
    Consider the simplest tessellation which corresponds to the lattice $\calT^h = h \Z^d$ with $d\geq 2$. We choose the uniform stationary measure $\pi^h(K) = |K| = h^d$ for all $K\in\calT^h$ and the uniform joint measure $\vartheta^h(K, L) 
    = h^{d-2} / 2$ for all $(K, L)\in\Sigma^h$. It follows that the transition kernel is $\kappa^h(K, L) = 1 / (2 h^2)$. Let $\{e_1, \cdots, e_d\}$ be the basis vectors in $\R^d$. One can always choose the orientation of the basis to be such that for any $K\in\calT^h$ the vectors $(x_L - x_K)$ pointing to the neighboring cells are $\{h e_1, -h e_1, h e_2, -h e_2, \cdots, h e_d, - h e_d\}$. Then the diffusion tensor $\T^h$ becomes the identity matrix for all $h>0$ and, consequently, $\T=\text{Id}$.
    
    We can also make a slightly different choice of the joint measure:
    $$
        \vartheta^h(K, L) = \frac{c_i}{2} h^{d-2} \quad \text{ if the face } (K|L) \text{ is orthogonal to } e_i,
    $$
    where $c_i > 0$ are independent of $h$. In this way, we make the jump intensities different in the different directions, i.e.\
    $ \kappa^h(K, L) = c_i /(2h^2)$ if the face $(K|L)$ is orthogonal to $e_i$. Note that \eqref{assumption_local_average} is still satisfied. In this case,
    $$
        \T = \left( \begin{array}{cccc}
            c_1 & 0 & \cdots & 0 \\
            0 & c_2 & \cdots & 0 \\
            \cdots & \cdots & \ddots & \cdots \\
            0 & 0 & \cdots & c_d
        \end{array} \right).
    $$
    Unsurprisingly, this example illustrates that different transitions intensities for the same family of tessellations may lead to different limit diffusion tensors.
\end{example}

\begin{example}[Tilted $h\Z^2$] \label{example_tilted}
Let $\calT^h$ be a tilted version of the lattice $h\Z^2$ as shown in the Figure~\ref{fig_tilted_example}. The tilt is given by the parameter $\alpha = \cos \gamma$, $\gamma\in [0, \pi/2 )$, where $\alpha = \cos (\pi/2)$ corresponds to $h\mathbb{Z}^2$. Each cell $K \in \calT^h$ has four neighbors $\left\{ K_r, K_u, K_l, K_d \right\}$, where the subscript stands for right, up, left, and down neighbors of $K$.

\begin{figure}[h]
\includegraphics[width=4cm, height=4cm]{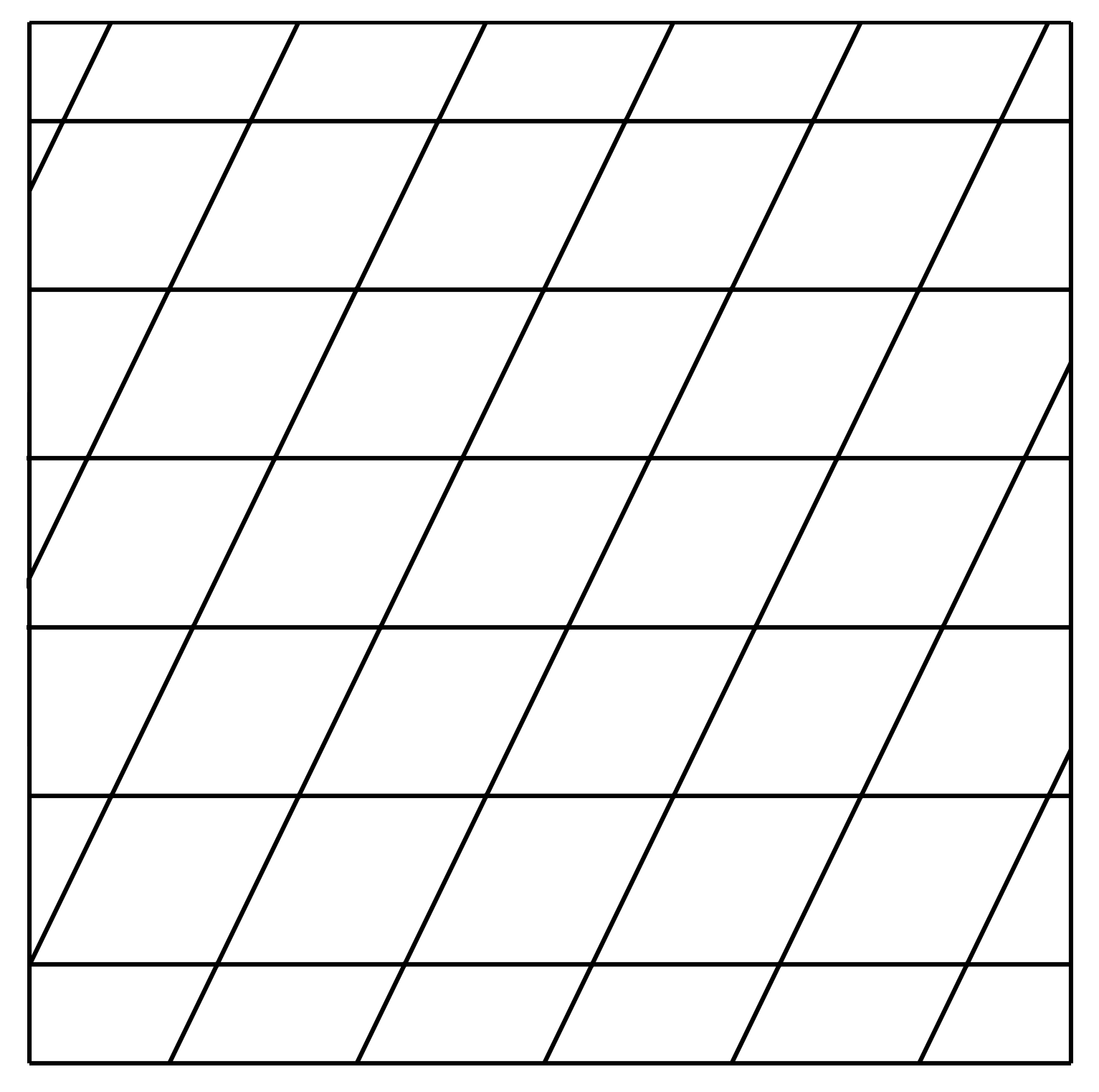}
\caption{Tilted $\Z^d$ tessellation.}
\label{fig_tilted_example}
\end{figure}

 We fix the basis $\{e_1, e_2\}$ such that $(x_{K_r} - x_K) = h e_1$. In this basis, we have that 
 \begin{gather*}
    (x_{K_l} - x_K) = - h e_1,\qquad (x_{K_u} - x_K) = h (\alpha^2 e_1 + (1 - \alpha^2) e_2),\\ (x_{K_d} - x_K) = - h (\alpha^2 e_1 + (1 - \alpha^2) e_2) .
 \end{gather*}
 The tensor $\T^h$ then takes the form
\begin{align*}
    \T^h(x) = \big[ &\left( \kappa^h(K, K_l) + \kappa^h(K, K_r) \right) h^2 e_1 \otimes e_1 \\
    &+ \left( \kappa^h(K, K_u) + \kappa^h(K, K_d) \right) h^2 \left( \alpha^2 e_1 + (1 - \alpha^2) e_2 \right) \otimes \left( \alpha^2 e_1 + (1 - \alpha^2) e_2 \right) \big]. 
\end{align*}
Notice that for any nonnegative $\kappa^h$, 
we can never obtain $\T = \Id$. For uniform kernels $\kappa^h = 1 / (2h^2)$, we get
\begin{equation*}
    \T(x) = 
        \left( \begin{array}{cc}
        1 + \alpha^4 & \alpha^2 (1 - \alpha^2) \\
        \alpha^2 (1 - \alpha^2) & (1 - \alpha^2)^2 
    \end{array} \right).
\end{equation*}
Analogous to the previous example, this example illustrates that the same transition intensities for different families of tessellations may lead to different limit diffusion tensors.
\end{example}

\begin{example}[] Let $\Omega=[-1, 1]\subset\R$. Consider the tessellation $\calT^h = \calT^h_- \cup \calT^h_+$ consisting of cells with length $h/2$ on $(-\infty, 0]$, i.e.\ $\calT^h_- = \{(-kh/2, -(k-1)h/2), k\in\N \}$ and the cells with the length $h$ on $[0, \infty)$, i.e.\ $\calT^h_+ = \{((k-1)h, kh), k\in\N\}$ (see Figure~\ref{fig_two_parts}). 

\begin{figure}[h]
\includegraphics[width=0.8\textwidth]{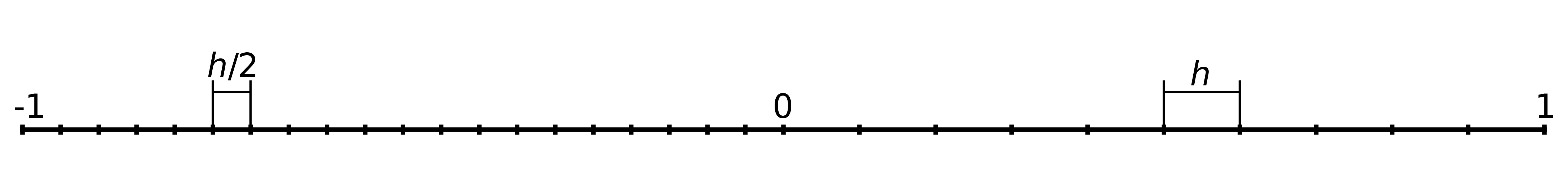}
\caption{A tessellation on $[-1, 1]$ consisting of cells with length $h/2$ in the left half and cells with length $h$ in the right half.}
\label{fig_two_parts}
\end{figure}

Consider $\vartheta^h(K, L) = 1 / |x_L - x_K|$ for $(K, L)\in\Sigma^h$, which immediately implies that \eqref{assumption_local_average} is satisfied, 
with the uniform stationary measure $\pi^h(K) = h$ for all $K\in\calT^h$. Then the tensor $\T^h$ reads
$$
    \T^h(x) = \sum_{L\in\calT^h_K} \frac{\vartheta^h(K, L)}{\pi^h(K)} |x_L - x_K|^2 = \sum_{L\in\calT^h_K} \frac{|x_L - x_K|}{\pi^h(K)} \qquad \text{for } x\in K.
$$
In particular, for $x \in (-1, -h/2)$, $\T^h(x) = 1$, and for $x \in (h, 1)$, $\T^h(x) = 2$. Therefore, in the limit $h\to 0$, we obtain $\T(x) = 2 \Ind_{(0, 1)}(x) + \Ind_{(-1, 0)}(x)$.

This last example illustrates how one obtains spatially inhomogeneous diffusion tensors in the limit from the inhomogeneity in the tessellations.
\end{example}

\section{Gradient structures: discrete and continuous}\label{sec_gf}

In this section, we collect all the necessary definitions and statements regarding the gradient flow formulation of the discrete random walk and of the postulated continuous diffusion governed by \eqref{eq_Kolmogorov} and \eqref{eq_limit_Kolmogorov} respectively.

\subsection{Generalized gradient structure for random walks}\label{sec_generalized_gf}

In the introduction we outlined the generalized gradient flow formulation for random walks. Gradient structures for Markov jump processes on graphs were first introduced in the independent works of Maas \cite{maas2011gradient}, Mielke \cite{mielke2013geodesic}, and  Chow, Huang and Zho \cite{chow2012fokker}. Motivated by large-deviation theory, a different form of gradient structure for discrete random walks was discovered by Mielke, Peletier and Renger in \cite{mielke2014relation}. Unlike the earlier gradient structures, these large-deviation inspired gradient structures did not fit into the classical framework of gradient flow theory (see Section~\ref{sec_gs_diffusion}). Based on the energy-dissipation balance, a new framework for these structures, now known as {\em generalized gradient structures}, was recently established in \cite{peletier2022jump}. This section collects rigorous definitions and concepts following the framework developed in \cite{peletier2022jump}.

\medskip

We use the {\em graph gradient} and {\em graph divergence} defined respectively as
\begin{align*}
   \dnabla: \calB(\calT^h) \to \calB(\Sigma^h),\qquad  & (\dnabla \varphi^h)(K, L) = \varphi^h(L) - \varphi^h(K) \qquad \text{for all } (K, L)\in\Sigma^h; \\
   \ddiv: \calM(\Sigma^h) \to \calM(\calT^h),\qquad & (\ddiv j) (K) = \sum_{L\in\calT^h_K} \bigl[j(K, L) - j(L, K)\bigr] \qquad \text{for all } K\in\calT^h;
\end{align*}
where $\calB(X)$ denote the space bounded measurable functions on $X$, and $\calM(X)$ the space of finite (signed) measures equipped with the topology of weak$^*$ convergence, i.e.\ convergence against $\calC_0(X)$, the space of continuous functions that vanish at infinity. Furthermore, we denote by $|\nu|$ the total variation measure of a measure $\nu\in\calM(X)$, and by $\calP(X)$ the space of probability measures equipped with the topology of narrow convergence, i.e.\ convergence against $\calC_b(X)$, the space of continuous and bounded functions.

\medskip

We begin by defining the class of solutions for the continuity equation \eqref{eq_CE_discrete}.

\begin{definition}\label{def_CE_discrete} 

We call a pair $(\rho^h, j^h)$, where
\begin{itemize}
    \item $\rho^h\in\calC([0,T];\calP(\calT^h))$ is a curve of measures defined on the tessellation $\calT^h$, and
    \item $j^h$ is a measurable family of {\em fluxes} $j^h = (j^h_t)_{t\in[0,T]} \subset \calM(\Sigma^h)$ with $\int_0^T | j^h_t | (\Sigma^h) \dd t < \infty$,
\end{itemize}
a solution of the discrete continuity equation
    \begin{equation*}
        \partial_t \rho^h + \ddiv j^h = 0 \qquad \text{in } (0, T) \times \calT^h,
        \tag{CE$_h$}
    \end{equation*}
if for all $\varphi^h\in \calB(\calT^h)$ and $[s, t] \subset [0, T]$,
    \begin{equation}\label{eq_CE_discrete_distr}
        \sum_{K\in \calT^h} \varphi^h(K) \rho^h_{t}(K) - \sum_{K\in \calT^h} \varphi^h(K) \rho^h_{s}(K) = \int_{s}^{t} \sum_{(K,L)\in\Sigma^h} (\dnabla \varphi^h)(K, L) j^h_r(K, L) \dd r.
    \end{equation}
We denote by $\mathcal{CE}_h(0, T)$ the set of solutions to the discrete continuity equation \eqref{eq_CE_discrete}.
\end{definition}

Following \cite{peletier2022jump}, we define a generalized gradient flow solution of \eqref{eq_Kolmogorov} as follows:

\begin{definition}[GGF solutions]\label{def_GGF_solution}
    A curve $\rho^h\in\calC([0,T]; \calP(\calT^h))$ is said to be an $(\calE_h, \calR_h, \calR^*_h)$ -generalized gradient flow solution of \eqref{eq_Kolmogorov} with initial data $\bar\rho^h\in \calP(\calT^h)\cap \text{dom}(\calE_h)$ if
    \begin{enumerate}[label=(\roman*)]
        \item $\rho_0^h= \bar\rho^h$ in $\calP(\calT^h)$;
        \item there exists a measurable family $(j^h_t)_{t\in[0, T]} \subset \calM(\Sigma^h)$ such that $(\rho^h, j^h) \in \mathcal{CE}_h(0, T)$ with
        $$
            \int_s^t \calR_h(\rho^h_r, j^h_r) + \calD_h(\rho^h_r) \, \dd r + \calE_h(\rho^h_t) = \calE_h(\rho^h_s) \quad \text{for all } [s,t]\subset [0,T];
        $$
        where
        \[
            \calD_h(\rho) := \inf\Bigl\{ \liminf_{n\to\infty} \calR^*_h(\rho_n,-\dnabla \calE_h'(\rho_n)) : \rho_n\rightharpoonup \rho, \quad \sup_{n\ge 0} \calE_h(\rho_n) <\infty, \quad \rho_n > 0\Bigr\},
        \]
        i.e.\ $\calD_h$ is a lower-semicontinuous envelope of $\rho\mapsto \calR^*(\rho,-\dnabla \calE_h'(\rho))$.
        
        \item the chain rule holds, i.e.\
        \begin{align*}\tag{CR$_h$}
            \frac{\dd}{\dd t}\calE_h(\rho_t^h) =  \langle j_t^h,\dnabla \calE_h'(\rho_t^h)\rangle\qquad\text{for almost every $t\in (0,T)$}.
        \end{align*}
    \end{enumerate}
\end{definition}

\medskip

We now make specific choices for all the components of the energy-dissipation functional \eqref{eq_edb_discrete} introduced in Section~\ref{sec_intro}.

\medskip

\noindent{\em The driving energy} $\calE_h: \calP(\calT^h) \to [0, +\infty]$ is taken to be the relative entropy with respect to the stationary measure $\pi^h$, i.e.
\begin{equation*}
    \calE_h(\rho^h) := \Ent (\rho^h | \pi^h) = 
    \begin{cases}
        \displaystyle \sum_{K\in\calT^h} \phi\bigl( u^h(K) \bigr) \pi^h(K) & \text{if } \rho^h \ll \pi^h \text{ with } u^h := \frac{\dd \rho^h}{\dd \pi^h}, \\
       +\infty & \text{otherwise,}
    \end{cases}
\end{equation*}
with the energy density $\phi(s) = s \log s - s + 1$.

\medskip

\noindent{\em The dual dissipation potential} $\calR_h^*: \calP(\calT^h)\times \calB(\Sigma^h) \to [0, \infty)$, as defined in the introduction, takes the form
\begin{equation*}
    \calR^*_h(\rho^h, \xi^h) = \frac{1}{2} \sum_{(K, L)\in\Sigma^h} \Psi^* \left( \xi^h(K, L) \right) \sqrt{u^h(K) u^h(L)}\, \theta^h(K, L),
\end{equation*}
where 
$
    \Psi^*(\xi) = 4 \left( \cosh{(\xi/2)} - 1 \right).
$

\medskip

\noindent{\em The dissipation potential} $\calR_h: \calP(\calT^h)\times \calM(\Sigma^h) \to [0, +\infty]$ is the Legendre--Fenchel dual of $\calR^*$ w.r.t.\ to its second variable. In particular, it takes the explicit form
\begin{equation}\label{eq_dissipation_potential}
    \calR_h(\rho^h, j^h) = \begin{cases}
         \displaystyle \frac{1}{2} \sum_{(K, L)\in\Tilde{\Sigma}^h} \Psi\left( \frac{w^h(K, L)}{\sqrt{u^h(K) u^h(L)}}\right) \sqrt{u^h(K) u^h(L)} \,\vartheta^h(K, L) & \text{if } |j|(\Sigma^h \backslash \Tilde{\Sigma}^h) = 0, \\
        +\infty & \text{if } |j|(\Sigma^h \backslash \Tilde{\Sigma}^h) > 0,
    \end{cases}
\end{equation}
where $w^h := \dd j^h / \dd \vartheta^h$, $\Tilde{\Sigma}^h := \{(K, L)\in\Sigma^h: u^h(K)\, u^h(L) > 0 \}$, and
$$
    \Psi(s) = 2 s \log \left( \frac{s + \sqrt{s^2 + 4}}{2} \right) - \sqrt{s^2 + 4} + 4.
$$


\noindent{\em The Fisher information} $\calD_h: \calP(\calT^h) \to [0, +\infty]$ is defined as
    \begin{equation*}
        \calD_h(\rho^h) = \sum_{(K, L)\in\Sigma^h} \left|\left(\dnabla \sqrt{u^h} \right) (K, L) \right|^2 \vartheta^h(K, L) \quad \text{ with } u^h = \frac{\dd \rho^h}{\dd \pi^h}. 
    \end{equation*}

With the introduced choice of the energy-dissipation functional, the chain-rule estimate holds \cite[Corollary~5.6]{peletier2022jump} for any admissible curve with finite dissipation:
\begin{proposition}[Chain-rule estimate]
    For any curve $(\rho^h, j^h)\in\mathcal{CE}_h(0, T)$ with finite dissipation, i.e.\
    \[
      \int_0^T \calR_h(\rho^h_t, j^h_t) + \calD_h(\rho^h_t) \, \dd t <\infty,
    \]
    the chain rule \eqref{eq_CR_discrete} holds, thus leading to
    $$
        \calI_h (\rho^h, j^h) = \int_0^T \calR_h(\rho^h_t, j^h_t) + \calD(\rho^h_t) \dd t + \calE(\rho^h_T) - \calE(\rho^h_0) \geq 0.
    $$
\end{proposition}

In the next lemma, we list the properties of $(\Psi, \Psi^*)$ that will be used in the proof of Lemma~\ref{lemma_properties_flux}.

\begin{lemma}\label{lemma_Psi} The Legendre-Fenchel pair $(\Psi, \Psi^*)$ are such that
    \begin{enumerate}[label=(\roman*)]
        \item $\Psi$ is even and convex, $\Psi(0) = 0$, and $\Psi$ is strictly increasing for $s > 0$.
        \item For $s, p > 0$ the mapping $s\mapsto s\Psi(p/s)$ is decreasing.
        \item For $s>0$, $\Psi(s)$ has a strictly increasing inverse $\Psi^{-1} : [0, \infty] \to [0, \infty]$ satisfying
        $$
            \Psi^{-1} (r) \leq \frac{r}{\xi} + \frac{\Psi^*(\xi)}{\xi} \quad \text{for all } \xi > 0.
        $$
        \item $\Psi^*(\xi) \leq \xi^2 \cosh(\xi/2)$.
    \end{enumerate}
\end{lemma}
\begin{proof}
    (ii) By convexity for $0 < s < t$ it holds that:
    $$
        \Psi \left( \frac{s}{t} \cdot \frac{p}{s} \right) \leq \frac{s}{t} \Psi \left( \frac{p}{s} \right) \quad \Rightarrow \quad t \Psi \left( \frac{p}{t} \right) \leq s \Psi \left( \frac{p}{s} \right).
    $$
    
    (iii) Since $(\Psi, \Psi^*)$ are convex conjugate, then
    $$
        \Psi(s) \geq s \xi - \Psi^*(\xi) \quad \text{for any } s, \xi > 0.
    $$
    Therefore,
    \begin{align*}
        s \geq \Psi^{-1} \left( s \xi - \Psi^*(\xi) \right) \quad \Rightarrow \quad \Psi^{-1} (r) \leq \frac{r + \Psi^*(\xi)}{\xi},
    \end{align*}
    thus concluding the proof.
\end{proof}


\subsection{Gradient structure for continuous diffusion}\label{sec_gs_diffusion}

We noted in Remark~\ref{remark_limit_diffusion} that the limit forward Kolmogorov equation (also known as the Fokker--Planck equation) takes the form
\begin{equation}\tag{fK}
        \partial_t \rho_t = \text{div} \left(\T (\nabla\rho_t + \rho_t \nabla V) \right) \qquad \text{on } (0, T)\times\Omega,
\end{equation}
where $\Omega\subset\R^d$ is a bounded convex domain and $V\in \Lip_b(\Omega)$, the space of Lipschitz bounded functions. Such type of equations have been extensively studied over the last century, but the uncovering of their gradient structure in the space of measures only happened about two decades ago in the seminal work \cite{jordan1998variational} by Jordan, Kinderlehrer and Otto, where the 2-Wasserstein metric played a central role. Shortly after, a general framework for gradient flows in metric spaces was developed by Ambrosio, Gigli and Savar\'e in \cite{ambrosio2008gradient}, and the study of gradient flows for various evolution equations in spaces of measures has been an active area of research ever since. 
While there exist several ways to define gradient flow solutions to \eqref{eq_diffusion}, we take the same approach as for GGF-solution for \eqref{eq_Kolmogorov}, namely, the approach based on the energy-dissipation balance. 

\medskip

The class of curves we consider is the solutions of the continuity equation \eqref{eq_CE} in the sense of the following definition.
\begin{definition}\label{def_CE} The set of solutions $\mathcal{CE}(0, T)$ is given by all pairs $(\rho, j)$, where
\begin{itemize}
    \item $\rho\in\calC([0,T];\calP(\Omega))$ is a curve of positive measures defined on $\Omega$, and
    \item $j$ is a measurable family of {\em fluxes} $j = (j_t)_{t\in[0,T]} \subset \calM(\Omega;\R^d)$ with
    \[
        \int_0^T \int_\Omega \left|\frac{\dd j_t }{\dd \rho_t} \right|^2 \dd \rho_t \dd t < \infty,
    \]
\end{itemize}
    satisfying the continuity equation
    \begin{equation}
        \partial_t \rho + \text{div} j = 0 \qquad \text{in } (0, T) \times \Omega,
        \tag{CE}
    \end{equation}
    in the following sense:
    \begin{equation}\label{eq_CE_distr}
       \langle \varphi,\rho_t\rangle - \langle \varphi,\rho_s\rangle = \int_s^t \langle \nabla \varphi, j_r\rangle \dd r \qquad\text{for all $\varphi\in \calC_c^\infty(\R^d)$ and $[s,t]\subset [0,T]$.}
    \end{equation}
\end{definition}

\begin{remark}\label{rem_W2AC}
    It is known that if $\rho$ solves \eqref{eq_CE} with
    \[
        \int_0^T \int_\Omega \left|\frac{\dd j_t }{\dd \rho_t} \right|^2 \dd \rho_t \dd t < \infty,
    \]  
    then $\rho$ is an absolutely continuous curve in $\calP(\Omega)$ w.r.t.\ the 2-Wasserstein distance \cite[Chapter 8]{ambrosio2008gradient}. 
\end{remark}

\begin{definition}\label{def_GF_solution}
    A curve $\rho\in\calC([0,T];\calP(\Omega))$ is said to be an $(\calE, \calR, \calR^*)$-gradient flow solution of \eqref{eq_limit_Kolmogorov} with initial data $\bar\rho\in \calP(\Omega)\cap\text{dom}(\calE)$ if
    \begin{enumerate}[label=(\roman*)]
        \item $\rho_0=\bar\rho$ in $\calP(\Omega)$;
        \item there exists a measurable family $(j_t)_{t\in[0, T]} \subset \calM(\Omega; \R^d)$ such that $(\rho, j) \in \mathcal{CE}(0, T)$ with
        $$
            \int_s^t \int_\Omega \calR(\rho_r, j_r) + \calD(\rho_r) \, \dd r + \calE(\rho_t) = \calE(\rho_s) \quad \text{for all } [s,t]\subset [0,T],
        $$
        where
        \[
            \calD(\rho) := \inf\left\{ \liminf_{n\to\infty} \calR^*(\rho_n,-\dnabla \calE'(\rho_n)) : \rho_n\rightharpoonup \rho,\quad \sup\nolimits_{n\ge 0} \calE(\rho_n) <\infty, \quad \rho_n > 0\right\},
        \]
        i.e.\ $\calD$ is a lower-semicontinuous envelope of $\rho\mapsto \calR^*(\rho,-\dnabla \calE'(\rho))$.
        \item the following chain rule inequality holds:
        $$
            -\frac{\dd}{\dd t} \calE(\rho_t) \le \calR(\rho_t, j_t) + \calD(\rho_t)\qquad\text{for almost every $t\in(0,T)$.}
        $$
    \end{enumerate}
\end{definition}

\medskip

    According to the strategy explained in Section~\ref{sec_intro}, we will obtain the energy-dissipation functional $\calI$ by proving the corresponding $\liminf$ inequality for the discrete energy-dissipation functional $\calI_h$ introduced in Section~\ref{sec_generalized_gf}. For a family of GGF solutions $\{\rho^h\}_{h>0}$ of \eqref{eq_Kolmogorov}, we will immediately have
    $$
        \calI(\rho, j) \leq \liminf_{h\to 0} \calI_h(\rho^h, j^h) = 0.
    $$
    Then to prove that the limit curve $\rho$ indeed satisfies Defintion~\ref{def_GF_solution}, it is left to show that $\calI(\rho, j) \geq 0$, which is established by proving the chain rule inequality (iii) (cf.\ Theorem~\ref{th_chain_rule}).

\section{Continuous reconstruction and compactness}\label{section_reconstr_compact}
In this section, we first introduce our continuous reconstruction procedure for the density-flux pairs $(\rho^h, j^h)\in\mathcal{CE}_h(0, T)$ (cf.\ Section~\ref{sec_recostruction}). We then provide a compactness result for the sequence of continuous reconstructions $\{(\hat\rho^h,\hat\jmath^h)\}_{h>0}$ in Section~\ref{subsection_compactness}.

\subsection{Continuous reconstruction}\label{sec_recostruction}
Throughout this paper we will extensively use two operations: projecting functions supported on $\Omega$ on the tessellation $\calT^h$ and lifting discrete functions supported on $\calT^h$ to $\Omega$. Specifically, we define the following operators
\begin{align*}
    &\bbP: L^1(\Omega) \to \calB(\calT^h),~ \qquad v^h (K) = \bbP v\, (K) = \frac{1}{|K|} \int_{K} v(x) \dd x, \quad K\in\calT^h, \\
    &\bbL: \calB(\calT^h) \to \text{PC}(\calT^h), \qquad \hat{v}^h := \bbL v^h = \sum_{K\in\calT^h} v^h(K) \Ind_{K},
\end{align*}
where $\text{PC}(\calT^h)\subset L^1(\Omega)$ is the set of functions that are piecewise-constant on cells $K\in\calT^h$.

The motivating idea for the reconstruction procedure is to embed the curve $(\rho^h, j^h)\in \mathcal{CE}_h(0, T)$ into the continuous space in such a way that the lifted curve $(\hat{\rho}^h, \hat{\jmath}^h)$ belong to $\mathcal{CE}(0, T)$. Assuming that $\varphi^h = \bbP \varphi$, we transform the left-hand side of (\ref{eq_CE_discrete_distr}) into
\begin{align*}
    \sum_{K\in \calT^h} \varphi^h(K) \rho_t^h(K) 
    = \sum_{K\in \calT^h} \frac{\rho_t^h(K)}{|K|} \int_{K} \varphi(x) \dd x 
    &= \int_\Omega \varphi(x) \left( \sum_{K\in\calT^h} \frac{\rho_t^h(K)}{|K|}\Ind_{K}(x)\right) \dd x. 
\end{align*}
Defining the reconstructed measure $\hat{\rho}^h$ via its density as
\begin{equation}\label{eq_lift_measure}
    \frac{\dd \hat\rho^h_t}{\dd \calL^d} := \sum_{K\in\calT^h} \frac{\rho_t^h(K)}{|K|}\Ind_{K},
\end{equation}
we then obtain equality in the first parts of (\ref{eq_CE_discrete_distr}) and \eqref{eq_CE_distr}:
\begin{equation}\label{eq_equality_left_CE}
    \sum_{K\in \calT^h} \varphi^h(K) \rho^h_t(K) - \sum_{K\in \calT^h} \varphi^h(K) \rho^h_s(K)
    = \int_{\Omega} \varphi(x) \hat\rho^h_t (\dd x) - \int_{\Omega} \varphi(x) \hat\rho^h_s (\dd x).
\end{equation}

In what follows we will also frequently use the formulation in terms of density with respect to the stationary measure $u^h := \dd \rho^h /\dd \pi^h$:
\begin{align*}
    \sum_{K\in \calT^h} \varphi^h(K) \rho_t^h(K) 
    &= \int_\Omega \varphi(x) \left( \sum_{K\in\calT^h} u_t^h(K)\frac{\pi^h(K)}{|K|}\Ind_{K}(x)\right) \dd x 
    = \int_\Omega \varphi(x)\, \hat u^h_t(x) \hat\pi^h(\dd x),
\end{align*}
with
\[
    \hat{u}^h = \bbL u^h\qquad\text{and}\qquad \frac{\dd \hat\pi^h}{\dd \calL^d} := \sum_{K\in\calT^h} \frac{\pi^h(K)}{|K|}\Ind_{K}.
\]

Assuming the same relation between the test functions $\varphi^h = \bbP \varphi$, we now look for a reconstruction formula for the flux that gives the equality in the right-hand sides of \eqref{eq_CE_discrete_distr} and \eqref{eq_CE_distr}. For this purpose, we find a relation between the corresponding gradients of functions.

\begin{lemma}\label{lemma_relation_gradients} Let $\varphi^h := \bbP \varphi$ be the projection of $\varphi$ on $\calT^h$. Then there exists a vector-valued measure $\sigma_{KL} \in \calM(\Omega; \R^d)$ such that
    \begin{equation}\label{eq_disc_cont_gradients}
        (\dnabla \varphi^h)(K, L) = \int_{\Omega} (\nabla \varphi) (x) \cdot \sigma_{KL} (\dd x),\qquad \forall (K, L)\in\Sigma^h, \quad \forall\varphi\in \calC_b^1(\Omega).
    \end{equation}
    Furthermore, $|\sigma_{KL}|(\Omega) \le 2dh.$ 
\end{lemma}

Before presenting the proof of this lemma, let us show its application to the definition of the reconstructed fluxes. Applying Lemma~\ref{lemma_relation_gradients} and the definitions of $\mathcal{CE}_h(0, T)$ and $\mathcal{CE}(0, T)$, we note that 
\begin{align*}
    \sum_{(K,L)\in\Sigma^h} (\dnabla \varphi^h)(K, L) j^h_t(K, L) 
    &= \sum_{(K,  L) \in \Sigma^h} j^h_t(K, L) \int_{\Omega} (\nabla \varphi) (x) \sigma_{KL} (\dd x) \\
    &= \int_{\Omega} (\nabla \varphi) (x) \sum_{(K, L) \in \Sigma^h} j^h_t(K, L) \sigma_{KL}(\dd x).
\end{align*}
Therefore, we define
\begin{equation}\label{eq_lift_flux}
    \hat{\jmath}^h_t := \sum_{(K, L) \in \Sigma^h} j^h_t(K, L) \sigma_{KL}.
\end{equation}

Then for a given $(\rho^h, j^h)\in\mathcal{CE}_h(0, T)$, the pair $(\hat{\rho}^h,\hat{\jmath}^h)$ defined in \eqref{eq_lift_measure} and \eqref{eq_lift_flux} solves \eqref{eq_CE}
$$
    \partial_t \hat\rho^h_t + \nabla\cdot \hat \jmath^h_t = 0 \quad \text{in } (0, T)\times \Omega,
$$
in the sense of Definition~\ref{def_CE}.

\begin{proof}[Proof of Lemma~\ref{lemma_relation_gradients}]
    For any pair of neighboring cells $(K,L)\in\Sigma^h$:
\begin{align*}
    (\dnabla \varphi^h)(K, L) 
    = \varphi^h(L) - \varphi^h(K) 
    &= \int_{\Omega} \varphi(y) \frac{\Ind_{L}(y)}{|L|} \dd y - \int_{\Omega} \varphi(x) \frac{\Ind_{K}(x)}{|K|} \dd x \\
    &= \iint_{\Omega\times\Omega} (\varphi(y) - \varphi(x))\, \gamma^h_{KL}(\dd x\, \dd y),
\end{align*}
where $\gamma_{KL}^h$ is an arbitrary coupling between the measures $\mathfrak{m}_K= |K|^{-1}\calL^d|_K$ and $\mathfrak{m}_L = |L|^{-1}\calL^d|_L$. We assume that the coupling is produced by a transport map $T_{KL}$, meaning that for all $x\in K$ there exist unique $y\in L$ such that $T_{KL} x = y$. In this case, the coupling has the form
$$
    \gamma_{KL}^h = (\Id \times T_{KL})_\# \mathfrak{m}_K^h.
$$

For a smooth $\varphi$ the fundamental theorem of calculus gives:
$$
    \varphi(y) - \varphi(x) = \int_0^1 (\nabla \varphi) (x + \tau (y - x)) \dd \tau \cdot (y - x).
$$
Rewriting the coupling in terms of the transport map yields:
\begin{align*}
    \iint_{\Omega \times \Omega} &(\varphi(y) - \varphi(x)) \gamma^h_{KL}(\dd x\, \dd y) \\
    &= \iint_{\Omega \times \Omega} \int_0^1 (\nabla \varphi) (x + \tau (y - x)) \dd \tau \cdot (y - x) \left( (\Id \times T_{KL})_\# \mathfrak{m}_K \right) (\dd x\dd y) \\
    &= \int_0^1 \int_{\Omega} (\nabla \varphi) (x + \tau (T_{KL}x - x)) \cdot (T_{KL}x - x) \mathfrak{m}_K(\dd x) \dd \tau.
\end{align*} 
Introducing the notation $r_{KL} (x) := T_{KL} x - x$ and $\Phi^\tau_{KL} (x) := x + \tau r_{KL} (x)$ we proceed:
\begin{align*}
    \int_0^1 \int_{\Omega} (\nabla \varphi) (x + \tau r_{KL} (x)) \cdot r_{KL} (x) \mathfrak{m}_K(\dd x) \dd \tau
    &= \int_0^1 \int_{\Omega} (\nabla \varphi) (x) \left[ (\Phi^\tau_{KL})_\# \left( r_{KL} \mathfrak{m}_K \right) \right] (\dd x) \dd \tau \\
    &= \int_{\Omega} (\nabla \varphi) (x) \left[ \int_0^1 (\Phi^\tau_{KL})_\# \left( r_{KL} \mathfrak{m}_K \right) \dd \tau \right] (\dd x).
\end{align*}
Denoting by $\sigma_{KL}$ the measure $ \int_0^1 \left( (\Phi^\tau_{KL})_\# (r_{KL} \mathfrak{m}_K) \right) \dd \tau$ we obtain (\ref{eq_disc_cont_gradients}).

{\color{black} To estimate the total variation of $\sigma_{KL}$, we notice that
$$
    \left| \int_0^1 \int_{\Omega} f (x + \tau r_{KL}(x) ) r^i_{KL}(x) \mathfrak{m}_K (\dd x) \dd \tau \right|
    \leq \|f\|_{L^\infty} \sup_{x\in K} |r^i_{KL}(x)| \qquad \text{for } f\in \calB(\Omega),
$$
where $\displaystyle\sup_{x\in K} |r^i_{KL}(x)| \leq \sup_{x\in K, y\in L} |x - y| \leq 2h $. Therefore,
$$
    |\sigma_{KL}|(\Omega) = \sum_{i=1}^d |\sigma^i_{KL}|(\Omega) \leq 2dh.
$$
} 
\end{proof}

\begin{remark}
    One can notice that the measures $\sigma_{KL}$ constructed in the proof are not uniquely defined due to the freedom in choosing transport maps $T_{KL}$. However, we will see that the compactness result in Lemma~\ref{lemma_properties_flux} does not depend on the specific choice of $T_{KL}$.
\end{remark}

In the case of a lattice, the measure $\sigma_{KL}$ can be calculated explicitly. 
\begin{example}
    Consider the tessellation $h\Z^d$. For any pair of neighboring cells $K$ and $L$ the optimal transport map is $T_{KL} x = x + h n_{KL}$, with $n_{KL}$ being the (outward) normal on the cell face $(K|L)$, and, respectively, $r_{KL}(x) = h n_{KL}$. The function $\Phi_{KL}^\tau(x) = x + \tau h n_{KL}$ has an inverse $(\Phi_{KL}^\tau)^{-1} (y) = y - \tau h n_{KL}$. Therefore, it is possible to calculate the measure $\sigma_{KL}$ explicitly:
    \begin{align*}
        \int_{\R^d} f(x)\, \sigma_{KL}(\dd x) 
        &= h n_{KL} \int_{\R^d} \int_0^1 f(x + \tau h n_{KL})  \frac{\Ind_K(x)}{|K|} \dd \tau \dd x \\
        &= h n_{KL} \int_{\R^d} f(x) \int_0^1  \frac{\Ind_{K + \tau h n_{KL}}(x)}{|K|} \dd \tau \dd x.
    \end{align*}
    Notice that for any $x \in K$ the indicator function $\Ind_{K + \tau h n_{KL}} (x)$ is equal to 1 for $\frac{1}{h} ( h - \text{dist}(x, (K | L) )$ and equal to 0 afterwards. Therefore, for $x \in K$:
    $$
        \int_0^1  \frac{\Ind_{K + \tau h n_{KL}}(x)}{|K|} \dd \tau = \frac{1}{h |K|} ( h - \text{dist}(x, (K | L) ).
    $$
    A similar property holds for $x \in L$:
    $$
        \int_0^1  \frac{\Ind_{K + \tau h n_{KL}}(x)}{|K|} \dd \tau = \frac{1}{|K|}  \int_{\frac{1}{h} \text{dist}(x, (K | L) ) }^1 \dd \tau = \frac{1}{h |K|} ( h - \text{dist}(x, (K | L) ).
    $$
    We conclude that 
    $$
        \sigma_{KL} (\dd x) = \frac{n_{KL}}{|K|} ( h - \text{dist}(x, (K | L) ) \dd x.
    $$
\end{example}
\subsection{Compactness}\label{subsection_compactness}
Throughout this section, we consider a family $\{(\rho^h,j^h)\}_{h>0}$ of the GGF-solutions to \eqref{eq_Kolmogorov} with initial data $\{\rho^h_0\}_{h>0}$ satisfying $\sup_{h>0} \calE_h(\rho^h_0) < \infty$. With the non-degeneracy assumption on $\{(\calT^h, \Sigma^h)\}_{h>0}$, and the assumptions \eqref{assumption_pi}, \eqref{assumption_nu} on $\{\pi^h\}_{h>0}$, and  $\{\vartheta^h\}_{h>0}$, we deduce compactness for the continuous reconstructions of the solutions.

\begin{lemma}\label{lemma_properties_flux}
    Let $(\hat\jmath_t^h)_{t\in(0,T)}\subset \calM(\Omega; \R^d)$, $h>0$, be defined as in (\ref{eq_lift_flux}). Then
    \begin{enumerate}
        \item the family
        \[
            \left\{\int_{\cdot}\, \hat{\jmath}_t^h\,\dd t \right\}_{h>0}\quad\text{is (sequentially) weakly-$*$ compact in $\calM((0,T)\times\Omega; \R^d)$;}
        \]
        \item the family $\{t\mapsto |\hat{\jmath}_t^h|(\Omega)\}_{h>0}$ is equi-integrable.
    \end{enumerate}
    In particular, there exists a Borel family $(j_t)_{t\in(0,T)}\subset \calM(\Omega;\R^d)$ such that
    \[
        \int_{\cdot}\,\hat{\jmath}_t^h\,\dd t \rightharpoonup^* \int_{\cdot}\, j_t\,\dd t\quad\text{weakly-$*$ in $\calM((0,T)\times\Omega; \R^d)$}\qquad\text{for a (not relabelled) subsequence.}
    \]
\end{lemma}
 \begin{proof} Recall that for almost every $t\in(0,T)$,
$$
    \hat{\jmath}_t^h = \sum_{(K,L) \in \Sigma^h} j_t^h(K, L) \sigma_{KL}, \quad \text{with}\quad \sigma_{KL} = \int_0^1  (\Phi^\tau_{KL})_\# \left(r_{KL} \mathfrak{m}_K \right) \dd \tau.  
$$
For any measurable set $A\subset[0,T]$ and any $(K, L)\in\Sigma^
h$ denote:
$$
    Q^i_{KL} (A\times\Omega) := \int_A \vartheta_{\rho_t^h}(K,L) |\sigma^i_{KL}| (\Omega) \dd t \qquad \text{with } \vartheta_{\rho_t^h}(K,L) := \rho_t^h(K) \kappa^h(K, L).
$$
Note that $Q^i(A \times \Omega) := \displaystyle\sum_{(K, L)\in\Sigma^h} Q^i_{KL}(A\times\Omega)$ multiplied by $h$ is uniformly bounded because of~\eqref{assumption_kernel_uniform_upper_bound}:
$$
    h Q^i(A \times \Omega) = h \sum_{(K, L)\in\Sigma^h} Q^i_{KL}(A\times \Omega) \le 2C_r h^2 \sum_{L\in\calT^h_K} \kappa^h(K,L)\mathscr{L}^1(A) \le 2C_r C_\kappa \mathscr{L}^1(A).
$$
Setting $J^h:=\int_\cdot \hat\jmath_t^h\,\dd t$, we will show that the sequence of measures $\{J^h(\cdot\times\Omega)\}_{h>0}\subset\calM(0,T)$ is uniformly integrable. The properties collected in Lemma~\ref{lemma_Psi}(i) together with \eqref{assumption_kernel_uniform_upper_bound} provide the following estimate:
\begin{align*}
    \Psi\left(\frac{h|J^{h,i}| (A\times\Omega)}{hQ^i(A\times\Omega)} \right) &\le \Psi\left(\frac{1}{h Q^i(A\times\Omega)} \int_A \sum_{(K,L)\in\Sigma^h} h\left| j_t^{h}(K, L) \right| |\sigma^i_{KL}| (\Omega)\,\dd t \right) \\
    &\le \frac{1}{h Q^i(A\times\Omega)} \int_A \sum_{(K,L)\in\Sigma^h} \Psi\left(h\frac{ j_t^h(K, L)}{\vartheta_{\rho_t^h}(K,L)} \right) \vartheta_{\rho_t^h}(K,L) |\sigma^i_{KL}| (\Omega)\,\dd t \\
    &\le \frac{C_r h^2}{h Q^i(A\times\Omega)} \int_A \sum_{(K,L)\in\Sigma^h} \Psi\left(\frac{ j_t^h(K, L)}{\vartheta_{\rho_t^h}(K,L)} \right) \vartheta_{\rho_t^h}(K,L)\,\dd t \\
    &= \frac{C_r h^2}{h Q^i(A\times\Omega)} \int_A \calR_h(\rho_t^h, j_t^h)\dd t.
\end{align*}
Since $h Q^i(A \times \Omega) \le C_\kappa \mathscr{L}^1(A)$, Lemma~\ref{lemma_Psi}(ii) gives that:
\[
    C_\kappa \mathscr{L}^1(A) \Psi\left(\frac{h|J^{h,i}| (A\times\Omega)}{C \mathscr{L}^1(A)} \right) \le h Q^i(A \times \Omega) \Psi\left(\frac{h|J^{h,i}| (A\times\Omega)}{h Q^i(A\times\Omega)} \right) \le C_r h^2 \int_A \calR_h(\rho_t^h, j_t^h)\dd t.
\] 
Taking the inverse yields:
$$
    |J^{h,i}| (A\times\Omega) \leq \frac{C_\kappa \mathscr{L}^1(A)}{h} \Psi^{-1} \left( \frac{C_r h^2}{C \mathscr{L}^1(A)} \int_A \calR_h(\rho_t^h, j_t^h)\dd t \right).
$$
Since $(\rho^h,j^h)$ are generalized gradient flow solutions in the sense of Definition~\ref{def_GGF_solution}, the integral of $\calR_h$ is bounded uniformly in $h$ under the assumption on the initial conditions:
$$
    \int_A \calR_h(\rho_t^h, j_t^h)\dd t \leq \int_0^T \calR_h(\rho_t^h, j_t^h)\dd t \leq \sup_{h>0} \calE_h(\rho^h_0) =: M_0 < \infty.
$$
Now we use the upper bound from  Lemma~\ref{lemma_Psi}(iii) to obtain
\begin{align*}
    |J^{h,i}| (A\times\Omega) &\leq \frac{C_\kappa \mathscr{L}^1(A)}{h} \left( \frac{1}{\xi} \frac{C_r h^2}{C_\kappa \mathscr{L}^1(A)} M_0 + \frac{\Psi^*(\xi)}{\xi} \right) \\
    &\leq C_r \frac{h}{\xi} M_0 + \frac{C_\kappa \mathscr{L}^1(A)}{h} \frac{\Psi^*(\xi)}{\xi}\qquad\text{for any $\xi > 0$.}
\end{align*}
Choosing $\xi = \beta h$, $\beta>0$, and using the property $\Psi^*(\xi) \le \xi^2\cosh(\xi/2)$, we find
\[
 \sup_{h\in(0,1)}|J^{h,i}| (A\times\Omega) \le \frac{C_r}{\beta}M_0 + C_\kappa \mathscr{L}^1(A)\beta\cosh\left(\frac{\beta}{2}\right)\qquad\text{for any $\beta>0$.}
\]
Now let $\varepsilon>0$ be arbitrary. By choosing $\beta>0$ such that $C_rM_0/\beta < \varepsilon/2$, and subsequently $\mathscr{L}^1(A)$ such that $C_\kappa \mathscr{L}^1(A)\beta\cosh(\beta/2)<\varepsilon/2$, we then conclude that
\[
    \sup_{h\in(0,1)}|J^{h,i}|(A\times\Omega) <\varepsilon,\qquad i=1,\ldots,d.
\]


Moreover, by applying the estimate above to $A=[0, T]$, we simply obtain
$$
    \sup_{h>0} |\, J^h | ([0,T]\times\Omega) \leq \left( C_r M_0 + C_\kappa \right) \sqrt{T},
$$
i.e.\ the total variation of $J^h$ is uniformly bounded, which allows us to extract a converging subsequence (not relabelled) and some $J$ such that $J^h\rightharpoonup^* J$ holds.

The equi-integrability of $t\mapsto |\hat{\jmath}_t^h|(\Omega)$ readily follows from the estimate above. Since the limit $J$ also satisfies the inequality above (weakly-$*$ lower-semicontinuity of the total variation), we conclude that $|J|(\cdot\times\Omega)$ on $[0,T]$ has Lebesgue density. By disintegration, $J$ has the representation $J = \int_\cdot j_t\dd t$ for a Borel family $(j_t)\subset\calM(\Omega; \R^d)$.
 \end{proof}

As a consequence of the previous lemma, we obtain the following result for density-flux pairs.
\begin{lemma}\label{lemma_limit_pair}
There exist a (not relabelled) subsequence of pairs $(\hat{\rho}^h, \hat{\jmath}^h)$ defined as in (\ref{eq_lift_measure}) and (\ref{eq_lift_flux}) and a pair $(\rho, j)\in\mathcal{CE}(0, T)$ such that
    $$
        \begin{array}{ll}
            \hat{\rho}^h_t \rightharpoonup^* \rho_t & \text{weakly-$*$ in } \calP(\Omega) \text{ for all } t \in [0, T]. 
        \end{array}
     $$
 \end{lemma}
 \begin{proof}
   Since $(\rho^h, j^h)$ satisfies $\eqref{eq_CE_discrete}$, then for all $h>0$ and all $[s, t] \subset [0, T]$ we have that
    \begin{align*}
        \left| \langle \varphi, \hat{\rho}^h_t \rangle - \langle \varphi, \hat{\rho}^h_s \rangle \right|
        &= \left| \int_s^t \int_{\Omega} \nabla \varphi (x) \cdot \hat{\jmath}^h_r (\dd x) \dd r \right| 
        \leq \| \varphi \|_{\text{Lip}}\sup_{h>0} |\, \hat{\jmath}^h | ([s, t]\times\Omega).
    \end{align*}
    Hence, the bounded Lipschitz distance is uniformly bounded:
    $$
        \sup_{h>0} d_{BL}(\hat{\rho}^h_s, \hat{\rho}^h_t) = \sup_{h>0} \sup_{\varphi} \left\{ \bigl| \langle \varphi, \hat{\rho}^h_t \rangle - \langle \varphi, \hat{\rho}^h_s \rangle \bigr| \right\} \leq \sup_{h>0} |\, \hat{\jmath}^h | ([s, t]\times\Omega),
    $$
    where the supremum is taken over all $1$-Lipschitz functions $\varphi$.
    
    From the equi-integrability of $t \mapsto |\, \hat{\jmath}^h_t | (\Omega)$ it follows that $\{ \hat{\rho}^h_t \}$ satisfies the refined version of Ascoli-Arzel\'a theorem (\cite[Theorem~3.3.1]{ambrosio2008gradient}) and there exist a (not relabelled) subsequence $\{\hat\rho^h\}_{h>0}$ and a limit curve $\rho\in \calC([0,T];\calP(\Omega))$, such that the asserted convergence holds.
\end{proof}

In the next lemma, we provide the uniform bound on the BV-norm for the reconstructed densities $\hat{u}^h := \bbL (\dd \rho^h/\dd \pi^h)$. As a preparation, we state the following property of non-degenerate tessellations $\calT^h$ \cite[Lemma~2.12(ii)]{gladbach2020scaling}.
\begin{proposition}\label{prop_conseq_zeta_regularity}
    Let $\calT^h$ satisfy the non-degeneracy assumption, and $x\in K$, $y\in L$ be arbitrary with $K, L \in\calT^h$. The cells $K$ and $L$ can be connected by a path $(K_i)_{i=0}^{n-1} \subset \calT^h$ with $K_0 = K$, $K_{n-1} = L$, $(K_i, K_{i+1})\in\Sigma^h$, and $[x, y]\cap (K|L) \neq \emptyset$, and $n \leq C_\zeta |x - y|/h$, where $C_\zeta>0$ depends only on $\zeta$.
\end{proposition}

\begin{lemma}\label{lemma_BV_bound} Let $\rho^h \in \calP(\calT^h)$ with $\calD_h(\rho^h) < \infty$. Then $\hat{u}^h = \bbL (\dd \rho^h/\dd \pi^h)$ satisfies 
\[
    \|\hat u^h\|_{L^1(\Omega)} \le \frac{1}{\pi_{\min}},\qquad |D\hat u^h|(\Omega) \le 2 \frac{\sqrt{C_\kappa}}{C_l}\sqrt{\calD_h(\rho^h)}.
\]
\end{lemma}
\begin{proof}
For a fixed $\psi\in \calC_c^1 (\Omega)$ we consider any $\eta\in\R^d$ such that $0 < |\eta| < \text{dist} (\supp (\psi), \partial \Omega)$, then
\begin{align*}
    \int_\Omega \hat{u}^h(x)\frac{\psi(x + \eta) - \psi(x)}{|\eta| } \dd x
    &= \frac{1}{|\eta|} \int_\Omega \hat{u}^h(x) (\psi(x + \eta) - \psi(x)) \dd x \\
    &=\frac{1}{|\eta|} \int_\Omega \left( \hat{u}^h(x - \eta) - \hat{u}^h(x) \right) \psi(x) \dd x \\
    &\leq \frac{1}{|\eta|} \|\psi \|_{L^\infty} \int_{\supp (\psi)} \left| \hat u^h(x - \eta) - \hat u^h(x) \right| \dd x.
\end{align*}
Note that
\begin{align*}
    \int_{\supp (\psi)} \left| \hat u^h(x - \eta) - \hat u^h(x) \right| \dd x = \sum_{K\in\calT^h} \int_{K\cap \supp (\psi)} \left| \hat u^h(x - \eta) - \hat u^h(x) \right| \dd x.
\end{align*}
Since $|\eta| < \text{dist} (\supp (\psi), \partial \Omega)$, we have that $x - \eta \in \Omega$ for any $x\in \supp (\psi)$. Therefore, we can find a unique cell $L\in \calT^h$ such that $x - \eta \in L$. The line segment $[x, x - \eta]$ between the points $x$ and $x - \eta$ defines a path between cell $K$ and cell $L$, consisting of pairs $(K_i, K_{i+1}) \in \Sigma^h$ such that $[x, x - \eta] \cap (K_i|K_{i+1}) \neq \emptyset$. We denote this sequence of pairs by $\{ (K_0=K, K_1), ~ (K_1, K_2), ~ \dots, ~ (K_{n-1}, K_n=L) \}$. We further define the sets
\begin{align*}
    \text{Cyl}_{\Sigma^h}(x,\eta) &:=\left\{ (\tilde M,\tilde L)\in \Sigma^h\, :\, [x, x - \eta] \cap (\tilde M|\tilde L) \neq \emptyset \right\},\qquad x\in\Omega\,,\\
    \text{Cyl}_\Omega((K,L),\eta) &:= \Bigl\{ x\in \Omega\, :\, [x, x - \eta] \cap (K | L) \neq \emptyset\Bigr\}\,,\qquad (K,L)\in\Sigma^h\,.
\end{align*}
Applying the triangle inequality, we have
     \begin{align*}
         \int_{\supp (\psi)} |\hat u^h(x - \eta) - \hat u^h(x)| \dd x
         &\le \sum_{K\in\calT^h} \int_{K\cap \supp (\psi)} \sum_{i=0}^{n-1}\left|u^h(K_{i+1}) - u^h(K_i)\right|  \dd x \\
         &\le \sum_{K\in\calT^h} \int_{K\cap \supp (\psi)} \sum_{(M,L)\in\Sigma^h}\left|u^h(L) - u^h(M)\right| \Ind_{\text{Cyl}_{\Sigma^h}(x,\eta)}(M,L)\,\dd x \\
         &= \sum_{(M,L)\in\Sigma^h} \left|u^h(L) - u^h(M)\right| \int_{\supp (\psi)} \Ind_{\text{Cyl}_\Omega((M,L),\eta)}(x)\,\dd x\\
        &\le \sum_{(K,L)\in\Sigma^h} \left|u^h(L) - u^h(M)\right| |(K|L)| |\eta|,
     \end{align*}
where the last inequality follows from the geometric argument that $\Ind_{\text{Cyl}_\Omega((K,L),\eta)}(x) = 1$ if and only if the point $x\in\Omega$ is in the cylinder $\text{Cyl}_\Omega((K|L), \eta)$ with base $(K|L)$ and axis parallel to $\eta$.

Applying the lower bound from \eqref{assumption_nu} and then the H\"older inequality, we then obtain
\begin{align*}
     \int_{\supp (\psi)} |\hat u^h(x - \eta) - \hat u^h(x)| \dd x
    &\le \frac{|\eta|}{C_l} \sum_{(K,L)\in\Sigma^h} \left|u^h(L) - u^h(K)\right| h \vartheta^h(K, L) \\
    &\le 2 \frac{|\eta|}{C_l} \left( \sum_{K\in\calT^h} \rho^h(K) \sum_{L\in\calT^h_K} h^2 \kappa^h(K,L) \right)^{1/2} \sqrt{\calD_h(\rho^h)} \\
    &\le 2 \frac{|\eta| \sqrt{C_\kappa}}{C_l} \sqrt{\calD_h(\rho^h)}.
\end{align*}
Therefore,
$$
    \int_\Omega \hat{u}^h(x)\frac{\psi(x + \eta) - \psi(x)}{|\eta| } \dd x
    \leq 2 \frac{\sqrt{C_\kappa}}{C_l} \|\psi \|_{L^\infty} \sqrt{\calD_h(\rho^h)}.
$$
Taking the limit superior as $|\eta|\to 0$, and applying the dominated convergence theorem, we obtain
\[
     \int_\Omega \hat u^h(x) (\partial_\eta\psi)(x)\dd x \le 2 \frac{\sqrt{C_\kappa}}{C_l} \| \psi \|_{L^\infty} \sqrt{\calD_h(\rho^h)}\,.
\]
Finally, we take the supremum over $\psi \in \calC_c^1 (\Omega)$ satisfying $ \| \psi \|_{L^\infty}\le 1$ and use the variational characterization of the BV-seminorm to obtain
$$
    |Du^h|(\Omega) \leq 2 \frac{\sqrt{C_\kappa}}{C_l} \sqrt{\calD_h(\rho^h)} \qquad\text{for all $h > 0$}\,.
$$
The bound on the $L^1$-norm follows directly from assumption \eqref{assumption_pi}.
\end{proof}

With the BV-bound proven in Lemma~\ref{lemma_BV_bound}, we are now prepared to prove the compactness result for the GGF-solutions of \eqref{eq_Kolmogorov}.
\begin{theorem}[Strong compactness]\label{th_compactness}
Let the family of curves $\{\rho^h\}_{h>0}$ be the GGF-solutions of \eqref{eq_Kolmogorov} with $ \sup_{h>0} \calE_h(\rho^h_0) < \infty$. Then there exists $u \in L^1( (0, T); L^1(\Omega))$ and a (not relabelled) subsequence such that 
\[
    \hat{u}^h_t \to u_t \quad\text{strongly in $L^1(\Omega)$\; for $\calL^1$-a.e. $t\in(0,T)$.}
\]
\end{theorem}
\begin{proof} We first notice that the BV bound from Lemma~\ref{lemma_BV_bound} holds for almost every $t\in [0, T]$.  Therefore, $\{t\mapsto\hat{u}^h_t\}$ is tight with respect to the BV-norm in the sense that
$$
    \sup_{h>0} \int_0^T \| \hat{u}^h_t \|_{BV(\Omega)}^2 \dd t \leq 2C^2 \left( T + \sup_{h>0} \int_0^T \calD_h(\rho^h_t) \dd t \right) \leq 2C^2 \left( T +  \sup_{h>0} \calE_h(\rho^h_0) \right).
$$
Moreover, Lemma~\ref{lemma_limit_pair} provides weak integral equicontinuity, i.e.\
\begin{align*}
    \lim_{\tau\to 0} \sup_{h>0} \int_0^{T-\tau} d_{BL}(\hat{\rho}^h_{t+\tau}, \hat{\rho}^h_t) \dd t
    &\leq \lim_{\tau\to 0} \sup_{h>0} \int_0^{T-\tau} |\, \hat{\jmath}^h | \left([t, t + \tau] \times \Omega \right) \dd t \\
    &\leq  \lim_{\tau\to 0} \int_0^{T-\tau} C \tau \dd t = 0.
\end{align*}
Together, the tightness condition and the weak integral equicontinuity yield the relative compactness of $\{\hat u^h\}_{h>0}$ in $\calM((0, T); L^1(\Omega))$ \cite[Theorem 2]{rossi2003tightness}. The relative compactness in $\calM((0, T); L^1(\Omega))$ combined with the uniform integrability estimate
$$
   \sup_{h>0} \int_A \| \hat{u}^h_t \|_{L^1(\Omega)} \dd t \leq \pi_{\min}^{-1}|A| \qquad \text{for any $\calL^1$-measurable set} A\subset [0, T]. 
$$
provides that $\{\hat u^h\}_{h>0}$ is relatively compact in $L^1((0,T); L^1(\Omega))$ \cite[Proposition~1.10]{rossi2003tightness}.
Therefore, there exists some $u\in L^1((0,T); L^1(\Omega))$ and a subsequence of $\hat u^h$ (not relabelled) such that $\hat{u}_t^h \to u_t$ in $L^1(\Omega)$ for almost every $t\in(0,T)$. 
\end{proof}



\section{Gamma-convergence results}\label{sec_gamma_convergence}

This section contains convergence results for the Fisher information $\calD_h$ and the dual dissipation potential $\calR^*_h$. Since $\calR^*_h$ and $\calD_h$ are closely related, we will first introduce the results that hold for both functionals using a generic notation (cf.\ Section~\ref{sec_general_gamma_convergence}). Then we deal with the dual dissipation potential in Section~\ref{sec_dissipation}, where we show the asymptotic upper bound. In Section~\ref{sec_fisher} we apply the general results from Section~\ref{sec_general_gamma_convergence} to prove the $\Gamma$-convergence of the Fisher information.  

\subsection{General Gamma-convergence results}\label{sec_general_gamma_convergence} The notation throughout the first part of this section is as follows: 
\begin{enumerate}[label=(\roman*)]
    \item Let $\calO$ be the family of all open subsets of $\Omega$ with Lipschitz boundary. We denote by $\calT^h|_A$ the restriction of $\calT^h$ to $A$, i.e.\ $\calT^h|_A := \left\{ K \in \calT^h:\, K \cap A \neq \emptyset \right\}$. Furthermore, we introduce the set $A_{\calT^h} := \Omega\cap \text{int} \bigcup_{K\in\calT^h|_A} \overline{K}$, which can be larger then $A$ (see Figure~\ref{fig_restricted}). In what follows, we will use the convergence of the domain $A_{\calT^h}$ to $A$ in the following sense:
\end{enumerate}
    \begin{proposition}\label{prop_convergence_of_covering}
        For any $A\in\calO$ the indicator functions $\Ind_{A_{\calT^h}}$ converge pointwise $\calL^d$-a.e.\ to $\Ind_A$. 
    \end{proposition}
    \begin{figure}[h]
    \includegraphics[height=5cm]{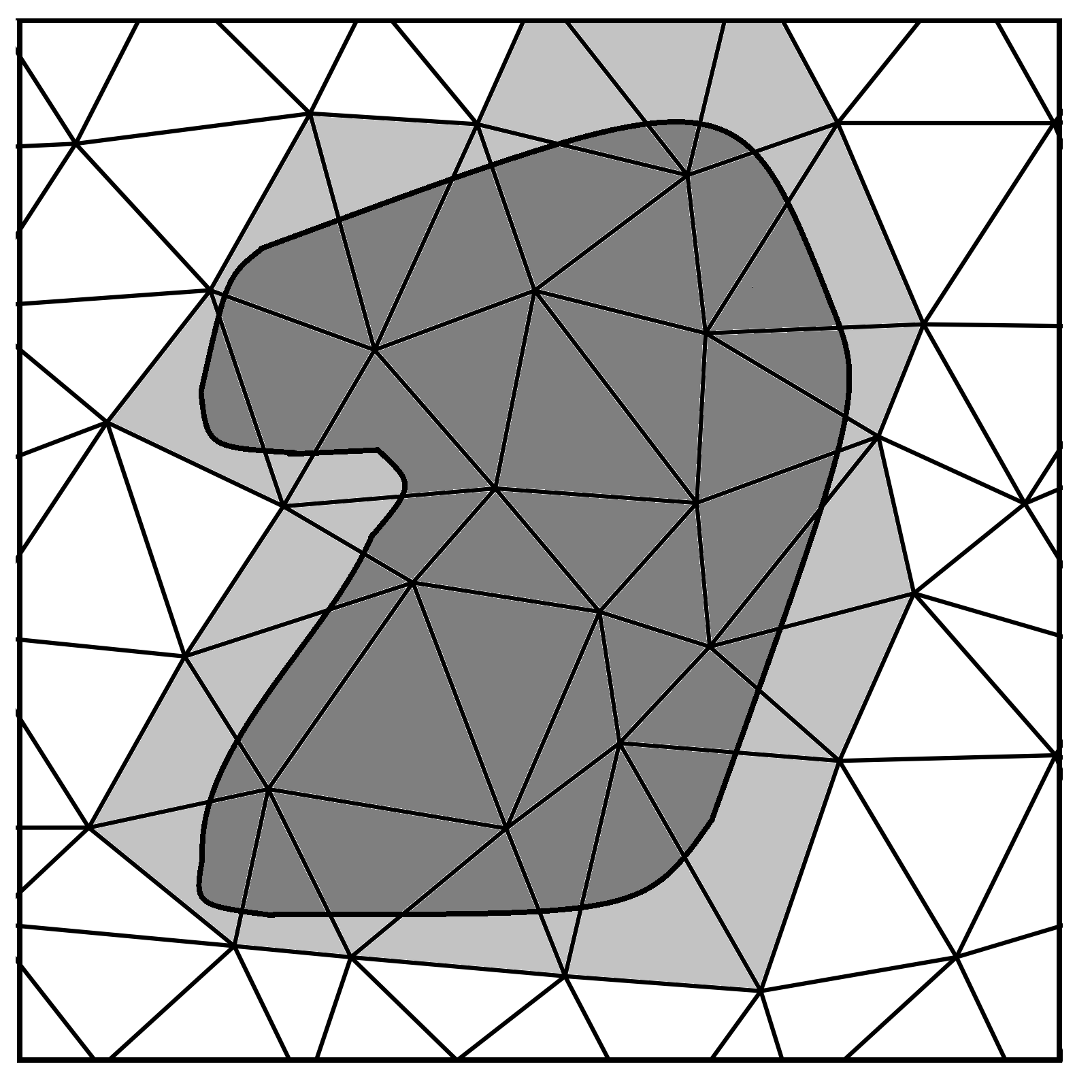}
    \caption{An example of a triangular tessellation on the square domain $\Omega$: Consider a set $A$ that is colored in dark gray. Then the set $A_{\calT^h}$ is the union of the set $A$ and the light gray area.}
    \label{fig_restricted}
    \end{figure}
\begin{enumerate}
    \item[(ii)] $\{\mu^h\}_{h>0}$ is a family of probability measures on $\calT^h$ such that $\dd \hat\mu^h/\dd \calL^d\in L^1(\Omega)$ for all $h>0$, where we use the reconstruction procedure defined in \eqref{eq_lift_measure}, i.e.\ 
    \begin{equation*}
        \frac{\dd \hat\mu^h}{\dd \calL^d} = \sum_{K\in\calT^h} \frac{\mu^h(K)}{|K|}\Ind_K,
    \end{equation*}
    and there exists $\mu\in\calP(\Omega)$ with Lebesgue density $\dd \mu/\dd \calL^d \in L^1(\Omega) $ such that $$
        \frac{\dd \hat\mu^h}{\dd \calL^d}  \to \frac{\dd \mu}{\dd \calL^d} \quad \text{in } L^1(\Omega) \text{ and pointwise } \calL^d\text{-a.e.\ as } h\to 0.
    $$
    
    \item[(iii)] For each $h>0$, the measure $\mu^h$ plays the role of the reference measure for the functional
    \begin{equation}\label{eq_general_functional_def}
        \calF_h^{\mu} (v^h) := \sum_{(K, L)\in\Sigma^h} \left| (\dnabla v^h) (K, L) \right|^2 \kappa^h(K, L)\, \mu^h(K),\qquad\,v^h \in \calB(\calT^h).
    \end{equation}
    \item[(iv)]  We introduce a \textit{localized} version of the functional $\calF^\mu_h$:
    \begin{equation}\label{eq_general_localized_func}
        \calF_h^\mu (v^h, A) := \sum_{(K, L)\in\Sigma^h|_A} \left| (\dnabla v^h) (K, L) \right|^2 \kappa^h(K, L) \mu^h(K),\qquad A\in\calO,
    \end{equation}
    where the summation goes over the restriction of $\Sigma^h$ to $A$, i.e. $$\Sigma^h|_A = \left\{ (K, L) \in \Sigma^h:\, K,L\in\calT^h|_A \right\}.$$

    \item[(v)] Eventually, we will prove $\Gamma$-convergence with respect to the $L^2$-topology. Therefore, we embed the discrete functional into the full $L^2(\Omega)$ space as:
    \begin{equation}\label{eq_lifted_functional}
        \Tilde{\calF}_h^\mu (v, A) := \begin{cases}
            \calF_h^\mu (v^h, A) & \text{if } v \in \text{PC}(\calT^h), \\
            +\infty & \text{otherwise.}
        \end{cases}
    \end{equation}
\end{enumerate}

\begin{remark}
This generic notation relates to the Fisher information $\calD_h$ in the following way:
$$
    \calD_h(\rho^h) = \calF_h^\pi\Bigl(\sqrt{u^h}\Bigr) \quad \text{with } u^h = \frac{\dd \rho^h}{\dd \pi^h}.
$$
The relation of $\calF_h^\mu$ with $\calR^*_h$ is more subtle. We show in Lemma~\ref{lemma_limsup_dualdiss} below that for a smooth $\varphi$ and a specific choice of approximating sequence $\varphi^h \to \varphi$ it holds that
$$
    \calR^*_h (\rho^h, \dnabla \varphi^h) = \frac{1}{4} \calF_h^\rho(\varphi^h) + o(1)|_{h\to 0}.
$$
\end{remark}

With the notation at hand, we outline the steps to prove $\Gamma$-convergence for $\calF^\mu_h$ by means of the localization technique:
\begin{enumerate}[label=(\roman*)]
        \item The family of functionals $\{\Tilde{\calF}_h^\mu (\cdot, A)\}_{h>0}$ has a subsequential $\Gamma$-limit $\calF^\mu(\cdot, A)$ for all $A\in\calO$ (cf.\ Lemma~\ref{lemma_existence_gamma_limit}). 
        \item The functionals $\calF^\mu(v, A)$ and, in particular, $\calF^\mu(v, \Omega)$ have an integral representation:
        $$
            \calF^\mu(v, A) = \int_A f(x, \nabla v) \dd \mu.
        $$
        For this, we need to prove that $\calF^\mu(v, \cdot)$ satisfies several properties as a set function, namely that $\calF^\mu(v, \cdot)$ is a measure and is local (cf.\ Proposition~\ref{prop_properties_Gamma_limit}).
        \item For $v\in H^1(\Omega, \mu)$, the integrand has an explicit upper bound $ f(x, \nabla v) \leq \langle \nabla v, \mathbb{T} \nabla v \rangle$ (cf.\ Lemma~\ref{lemma_limit_upper_bound}) with some tensor $\T$ that comprises the properties of the tessellations and the kernel (cf.\ Lemma~\ref{lemma_properties_tensor}). For a measure $\mu$ with the density $\dd \mu / \dd \calL^d$ bounded away from zero, we prove the exact integral representation, i.e.\ $ f(x, \nabla v) = \langle \nabla v, \mathbb{T} \nabla v \rangle$ (cf.\ Theorem~\ref{th_identification}).
    \end{enumerate}

\subsubsection*{Definitions and compactness}
We define
\begin{align}\label{eq_D_inf_D_sup}
    \left.\begin{aligned}
        \calF^\mu_{\inf}(\cdot, A) &:= \Gamma\text{-}\liminf_{h\to 0} \Tilde{\calF}^\mu_h (\cdot, A) \\
        \calF^\mu_{\sup}(\cdot, A) &:= \Gamma\text{-}\limsup_{h\to 0} \Tilde{\calF}^\mu_h (\cdot, A)
    \end{aligned}\qquad\right\}  \quad \text{for every } A\in\calO,
\end{align}
where, by the usual definition,
\begin{align*}
    \left.\begin{aligned}
        &\calF^\mu_{\inf}(\cdot, A) = \inf \big\{  \liminf_{h\to 0} \Tilde{\calF}^\mu_h (v_h, A) ~: ~ v_h \to v \big\}
        \\
        &\calF^\mu_{\sup}(\cdot, A) = \inf \big\{  \limsup_{h\to 0} \Tilde{\calF}^\mu_h (v_h, A) ~: ~ v_h \to v \big\}
    \end{aligned}\qquad\right\}  \quad \text{for every } v\in L^2(\Omega, \mu), ~ A\in\calO.
\end{align*}

Dealing with the functionals on the product space $L^2(\Omega, \mu)\times\calO$ has a few subtleties due to the set dependence. We proceed in accordance with the theory presented in \cite[Chapters 16-20]{maso1993introduction}. Since $\Tilde{\calF}^\mu_h(\cdot, A)$ are increasing functionals, i.e.\ $\Tilde{\calF}^\mu_h(\cdot, A') \leq \Tilde{\calF}^\mu_h(\cdot, A)$ for $A' \subset A$, we can apply the next definition 
\begin{definition}
We say that $\calF^\mu_h$ $\overline{\Gamma}-$converges to $\calF^\mu$ (in $L^2(\Omega, \mu)$) if $\calF^\mu$ is the
inner regular envelope of both functionals $\calF^\mu_{\inf}$ and $\calF^\mu_{\sup}$.
\end{definition}

The compactness result is standard.
\begin{lemma}\label{lemma_existence_gamma_limit}
    The family of functionals $\{\Tilde{\calF}_h^\mu\}_{h>0}$ defined in \eqref{eq_lifted_functional} is sequentially $\overline{\Gamma}$-compact, i.e.\ there exists a functional $\calF^\mu: L^2(\Omega, \mu)\times\calO \to [0, +\infty]$ such that $\overline{\Gamma}$-$\displaystyle \lim_{h\to 0} \Tilde{\calF}_h^\mu = \calF^\mu$ for some subsequence.
\end{lemma}
\begin{proof}
    The compactness theorem for localized functionals is similar to the standard compactness theorem for the $\Gamma$-convergence (see \cite[Proposition 16.9]{maso1993introduction}).
\end{proof}

\begin{remark}
    If one knows a priori that $\calF^\mu_{\sup}$ is inner regular, then the $\overline{\Gamma}$-limit is equivalently characterized by the usual $\Gamma$-limits for all $A\in\calO$ (\cite[Proposition 16.4, Remark 16.5]{maso1993introduction}):
\begin{enumerate}
    \item[($\Gamma_{\inf}$)] for every $v\in L^2(\Omega)$, for every $A\in \calO$, and for every sequence $v_h \to v$ in $L^2(\Omega)$ it holds that
    $$
        \calF(v, A) \leq \liminf_{h\to 0} \Tilde{\calF}_h^\mu(v_h, A); 
    $$
    \item[($\Gamma_{\sup}$)] for every $v\in L^2(\Omega)$ and for every $A \in \calO$, there exists a sequence $v_h \to v$ in $L^2(\Omega)$ such that
    $$
        \calF(v, A) \geq \limsup_{h\to 0} \Tilde{\calF}_h^\mu(v_h, A).
    $$
\end{enumerate}
\end{remark}

\subsubsection*{Integral representation}
Since the subsequential $\overline{\Gamma}$-limit $\calF^\mu$ exists, it is equal to the inner regular envelope of both functionals $\calF^\mu_{\inf}$ and $\calF^\mu_{\sup}$. Therefore, it suffices to show that $\calF^\mu_{\sup}$ is inner regular to conclude that $\calF^\mu = \calF^\mu_{\sup}$. We will establish inner regularity together with other properties of $\calF^\mu_{\sup}$ as a set function in Propositon~\ref{prop_properties_F_sup}. All these properties, as well as possible integral representation, rely on growth conditions for $\calF^\mu_{\sup}$.

To prove the growth conditions with respect to integrating against a possibly unbounded measure $\mu$, we fix a suitable definition of $H^1(\Omega, \mu)$:
\begin{definition}\label{def_sobolev_space}
    We define $H^1(\Omega,\mu)$ to be the completion of $\calC_b^2(\Omega)$ w.r.t.\ the norm
    \[
        \|f\|_{H^1(\Omega, \mu)}^2 := \|f\|_{L^2(\Omega, \mu)}^2 + \|\nabla f\|_{L^2(\Omega, \mu)}^2 \,.
    \]
\end{definition}

A useful observation is the convergence of the discrete approximations $\bbL \bbP v$ to $v$.
\begin{lemma}
    Let $v\in H^1(\Omega,\mu)$, then $\hat v^h:=\bbL \bbP v \to v$ in $L^2(\Omega, \mu)$.
\end{lemma}
\begin{proof} By density arguments, it suffices to consider $v\in \calC_b^2(\Omega)$.

The fact that $\hat v^h\in L^2(\Omega,\mu)$ follows directly from the boundedness of $v$, and the convergence follows from the following inequality:
\begin{align*}
    \| \hat{v}^h - v \|^2_{L^2(\Omega, \mu)}
    &= \sum_{K\in\calT^h} \int_{K} \left| \intbar_K v(y) \dd y - v(x) \right|^2 \mu(\dd x)
        \\
    &\leq \sum_{K\in\calT^h} \int_{K} \intbar_K \left| v(y) - v(x) \right|^2 \dd y\, \mu(\dd x) \leq h^2\| \nabla v\|_{L^\infty}^2\, \mu(\Omega)\,.
\end{align*}
    Passing to the limit $h\to 0$ yields the statement.
\end{proof}

Now we establish the Sobolev upper bound for $\calF^\mu_{\sup}$. 
\begin{lemma}\label{lemma_general_upper_bound}
    For any $v \in H^1(\Omega, \mu)$ and $A\in\calO$,
    $$
        \calF_{\sup}^\mu (v, A) \leq 4 C_\kappa \int_A |\nabla v|^2 \dd \mu\,, 
    $$
    where $C_\kappa$ is as defined in \eqref{assumption_kernel_uniform_upper_bound}.
\end{lemma}
\begin{proof}
    For any $v\in \calC_b^2(\Omega)$ and $h>0$, we set  $\hat v^h:=\bbL \bbP v\in \text{PC}(\calT^h)$. Then
    \begin{align*}
        \calF^\mu_h(\bbP v, A) &= \sum_{(K,L)\in\Sigma^h|_A} \left| \bbP v(L) - \bbP v(K) \right|^2 \kappa(K, L) \mu^h(K) \\
        &=  \sum_{(K,L)\in\Sigma^h|_A} \kappa(K, L) \mu^h(K) \left| \int v(y) \mathfrak{m}_L(\dd y) - \int v(x) \mathfrak{m}_K(\dd x) \right|^2 \\
        &\le \sum_{(K,L)\in\Sigma^h|_A} \kappa(K, L) \mu^h(K) \iint \left| v(y) - v(x) \right|^2 \gamma_{KL} (\dd x\, \dd y )\,,
    \end{align*}
    where $\gamma_{KL}$ is a coupling between $\mathfrak{m}_K= \calL^d|_K$ and $\mathfrak{m}_L = \calL^d|_L$. Since $v$ is smooth and $x$ and $y$ are in neighboring cells, it holds that $\left| v(y) - v(x) \right| \leq 2 \left| \nabla v(x) \right| h + O(h^2)$, therefore,
    \begin{align*}
        \calF^\mu_h(\bbP v, A)
        &\leq \sum_{(K,L)\in\Sigma^h|_A} \kappa(K, L) \mu^h(K)  \left( 4 \iint \left| \nabla v(x) \right|^2 h^2 \gamma_{KL} (\dd x\, \dd y ) + O(h^4) \right)
    \end{align*}
    Applying \eqref{assumption_kernel_uniform_upper_bound} yields
    \begin{align*}
        \calF^\mu_h(\bbP v, A) &\le 4 C_\kappa \sum_{K\in\calT^h|_A} \mu^h(K) \left( \int_K \left| \nabla v(x) \right|^2 \mathfrak{m}_K(\dd x) + O(h^2) \right) \\
        &\le 4 C_\kappa \left( \int_{A_{\calT^h}} \left| \nabla v(x) \right|^2 \hat{\mu}^h(\dd x) + O(h^2) \hat{\mu}^h\bigl( A_{\calT^h} \bigr)  \right).
    \end{align*}
    The second term vanishes in the limit $h\to 0$ since $\limsup_{h>0}|\hat\mu^h|(\Omega)<\infty$. For the first term, we also notice that $\Ind_{A_{\calT^h}} \to \Ind_A$ pointwise $\calL^d$-a.e.\ (see Proposition~\ref{prop_convergence_of_covering}) and $|\nabla v|$ is bounded on $\Omega$. Thus, we can apply the generalized dominated convergence theorem \cite[Theorem 1.20]{evans2015measure} to obtain
    \begin{align*}
        \lim_{h\to 0} \int_{A_{\calT^h}} \left| \nabla v \right|^2 \dd \hat\mu^h = \int_A \left| \nabla v \right|^2 \dd \mu.
    \end{align*}
    Altogether, we obtain the following bound for any $v \in \calC_b^2(\Omega)$:
    $$
        \calF^\mu_{\sup}(v, A) \leq \limsup_{h\to 0} \calF^\mu_h(\bbP v, A) \leq 4C_\kappa \int_A \left| \nabla v(x) \right|^2 \mu(\dd x).
    $$
    For arbitrary $v\in H^1(\Omega, \mu)$, we consider a sequence $\{v_n\}_{n\in\N} \subset \calC_b^2(\Omega)$ such that $v_n \to v$ in $H^1(\Omega, \mu)$, then the lower semicontinuity of $\calF_{\sup}$ yields
    $$
        \calF^\mu_{\sup}(v, A) \leq \liminf_{n\to\infty} \calF^\mu_{\sup}(v_n, A) \leq \lim_{n\to\infty}4 C_\kappa \int_A \left| \nabla v_n(x) \right|^2 \mu(\dd x) = 4C_\kappa \int_A \left| \nabla v(x) \right|^2 \mu(\dd x)\,,
    $$
    thereby concluding the proof.
\end{proof}

The properties of $\calF_{\sup}^\mu$ as a set function, namely, inner regularity, subadditivity, and locality, play a crucial role for the integral representation. The proofs of these properties follow the strategy of De Giorgi's cut-off functions argument \cite{maso1993introduction}. For the discrete functionals, the proofs were established in \cite{alicandro2004general, forkert2020evolutionary} and can be applied with minor modification to our settings. For completeness, we include the proofs adapted to our projections and reconstruction procedures in Appendix~\ref{appendix_proofs}.
\begin{proposition}[Properties of $\calF_{\sup}^\mu$] \label{prop_properties_F_sup} The functional $\calF_{\sup}^\mu$ defined in \eqref{eq_general_localized_func} has the following properties:
    \begin{enumerate}[label=(\roman*)]
        \item Inner regularity: For any $v\in H^1(\Omega, \mu)$ and for any $A\in\calO$ it holds that
        $$
            \sup_{A'\ssubset A} \calF^\mu_{\sup}(v, A') = \calF^\mu_{\sup}(v, A);
        $$
        \item Subadditivity: For any $v\in H^1(\Omega, \mu)$ and for any $A, A', B, B' \in \calO$ such that $A'\ssubset A$ and $B'\ssubset B$ it holds that:
        $$
            \calF^\mu_{\sup}(v, A'\cup B') \leq \calF^\mu_{\sup}(v, A) + \calF^\mu_{\sup}(v, B);
        $$
        \item Locality: For any $A\in \calO$ and any $v, \psi \in H^1(\Omega, \mu)$ such that $v=\psi$ $\mu$-a.e. on $A$ there holds
        $$
            \calF^\mu_{\sup}(v, A) = \calF^\mu_{\sup}(v, A).
        $$
    \end{enumerate}
\end{proposition}

A direct consequence of Lemma~\ref{lemma_general_upper_bound} and Proposition~\ref{prop_properties_F_sup} is the integral representation of $\calF_{\sup}^\mu$ \cite{maso1993introduction}.

\begin{proposition}[Properties of the $\Gamma$-limit]\label{prop_properties_Gamma_limit} Let $\calF^\mu : L^2(\Omega, \mu)\times\calO \to [0, +\infty]$ be $\overline{\Gamma}$-limit of $(\Tilde{\calF}^\mu_h)$. For every $v\in H^1(\Omega, \mu)$ and every $A\in\calO$ the following properties hold:
    \begin{enumerate}[label=(\roman*)]
        \item $\calF^\mu_{\sup}(v, A) = \calF^\mu(v, A)$;
        \item $\calF^\mu(v + c, A) = \calF^\mu(v, A)$ for every $c\in \R$;
        \item $\calF^\mu(v, \cdot)$ is the restriction to $\calO$ of a Radon measure;
        \item $\calF^\mu( \cdot, A)$ is $L^2(\Omega, \mu)$-lower semicontinuous;
        \item $\calF^\mu( \cdot, A)$ is local, which means $\calF^\mu(v, A) = \calF^\mu(w, A)$ if $v=w$ $\mu$-a.e. on $A$;
        \item $\calF^\mu(v, A)$ satisfies the growth condition:
        $$
            0 \leq \calF^\mu(v, A) \leq C \int_A \left| \nabla v \right|^2 \dd \mu,
        $$
        with some $C > 0$.
        \item $\calF^\mu(v, A)$ has the integral representation
            \begin{equation*}
                \calF^\mu(v, A) = \int_A f \left( x, \nabla v \right) \dd \mu
            \end{equation*}
            where $v|_A \in H^1 (A,\mu)$.
    \end{enumerate}
\end{proposition}
\begin{proof}
     (i) Since $\calF_{\sup}^\mu$ is inner regular, we conclude by definition that $\calF = \calF_{\sup}^\mu$.
     
     (ii) It is easily seen that the equality holds for any $h>0$ and arbitrary $\mu^h, v^h, A$, and $c \in \R$.
     
     To conclude (iii) it is enough to show that $\calF_{\sup}^\mu$ is subadditive, superadditive, and inner regular on $\calO$ (see, for instance, \cite[Theorem 14.23]{maso1993introduction}). Proposition~\ref{prop_properties_F_sup} provides subadditivity and inner regularity, and here we only comment on superadditivity. By definition of $\calF_h^\mu$, for any $A, B\in \calO$ such that $A\cap B = \emptyset$ and $\text{dist}(A, B) > 0$ there exists small enough $h_0>0$ such that for any $h<h_0$:
     $$
        \calF^\mu_h (v, A\cup B) \geq \calF^\mu_h (v, A) + \calF^\mu_h (v, B).
     $$
     If $\text{dist}(A, B) = 0$, then the required property follows from inner regularity.
     
     Properties (iv)-(vi) directly follow from (i) and Proposition~\ref{prop_properties_F_sup}. 

     (vii) Properties (ii)-(vi) allow to conclude the integral representation \cite[Theorem 20.1]{maso1993introduction}.
\end{proof}

\begin{remark}
    The integrand $f(x, \xi)$ can be obtained for all $\xi\in\R^d$ and a.e.\ $x\in\Omega$ as
    \begin{equation}\label{eq_integrand}
        f(x, \xi) = \frac{\calF^\mu(\varphi_\xi, Q_\varepsilon(x))}{|Q_\varepsilon(x)|},
    \end{equation}
    where $\varphi_\xi(z) = \langle\xi, z\rangle$ (for details see \cite[Remark 4.5]{braides2006handbook}).
\end{remark}

\subsubsection*{Upper bound for the integral representation}
We derive an upper bound for $f$ using the representation \eqref{eq_integrand}. For a fixed $\xi\in \R^d$ the projection of $\varphi_\xi$ on $\calT^h$ is 
$$
    \bbP \varphi_\xi = \sum_{K\in\calT^h} \langle\xi, x_K\rangle \, \Ind_K \quad \text{with} \quad x_K = \intbar_K x \dd x.
$$
Substituting $\bbP \varphi_\xi$ into $\calF^\mu_h$ yields
\begin{align*}
    \calF^\mu_h (\bbP \varphi_\xi, A)
    &= \sum_{(K, L)\in\Sigma^h|_A} \left| \langle\xi, x_L - x_K \rangle \right|^2 \kappa^h(K, L) \mu^h(K) \\
    &= \sum_{K\in\calT^h|_A} \Big\langle\xi , \sum_{L\in\calT^h_K|_A} \kappa^h(K, L) (x_L - x_K) \otimes (x_L - x_K) \xi \Big\rangle \, \mu^h(K) \\
    &= \sum_{K\in\calT^h|_A} \int_K \langle \xi , \T^h(x)\, \xi \rangle \, \hat{\mu}^h (\dd x) = \Big\langle \xi , \int_\Omega \T^h(x) \Ind_{A_{\calT^h}} \hat{\mu}^h (\dd x) \, \xi \Big\rangle,
\end{align*}
where we denoted by $\T^h$ the tensor
\begin{equation}\label{eq_def_tensor}
    \T^h(x) := \sum_{K\in\calT^h} \Ind_K(x) \sum_{L\in\calT^h_K} \kappa^h(K, L) (x_L - x_K) \otimes (x_L - x_K).
\end{equation}

When passing $h\to 0$, we expect $\T^h$ converge to the diffusion tensor, therefore, we establish a number of useful properties of $\T^h$.
\begin{lemma}[Properties of $\T^h$]\label{lemma_properties_tensor} The diffusion tensor \eqref{eq_def_tensor} has the following properties:
    \begin{enumerate}[label=(\roman*)]
        \item $\T^h(x)$ is symmetric and positive-definite for any $x\in \Omega$; 
        
        \item $\{\T^h\}_{h>0}$ is bounded in $ L^\infty (\Omega; \R^{d\times d})$: 
        \[
            \text{for all the components $\T^h_{ij}$ it holds that  $\displaystyle \sup_{h>0} \|\T^h_{ij}\|_{L^\infty(\Omega)} < \infty$;}
        \]
        \item $\{\T^h\}_{h>0}$ has a weakly-$*$ limit in the $\sigma(L^\infty, L^1)$ topology, i.e.\ there exist a subsequence and a tensor $\T\in L^\infty (\Omega; \R^{d\times d})$ such that
        $$
            \lim_{h\to 0} \int_{\Omega} \T^h_{ij} f \dd x = \int_{\Omega}\T_{ij} f \dd x \qquad \text{for all } f\in L^1(\Omega).
        $$
        
    \end{enumerate}
\end{lemma}
\begin{proof}
    (i) Symmetry and positive-definiteness follow directly from the definition. 
    
    (ii) Fix any $x\in\Omega$ and consider the tensor $\mathbb{T}^h$ component-wise:
    $$
        \T^h_{ij} (x) = \sum_{L \in \calT^h_K} \kappa^h(K, L) (x_L^i - x_K^i)(x_L^j - x_K^j).
    $$
    The bound $|x_L^i - x_K^i| \leq 2h$ and \eqref{assumption_kernel_uniform_upper_bound} gives $| \T^h_{ij} (x) | \leq 2 h^2 \sum_{L \in \calT^h_K} \kappa^h(K, L) \leq 2 C_\kappa$. Consequently, $\sup_{h>0} \|\T^h_{ij}\|_{L^\infty(\Omega)} \leq 2 C_\kappa$.
    
    (iii) The weak-$*$ convergence follows from (ii) and the duality of $L^1$ and $L^\infty$ (see for instance \cite[Theorem 8.5, Examples 8.6(1)]{alt2016linear}).
\end{proof}



The next lemma provides an upper bound for the integral representation of $\calF^\mu$
\begin{lemma}\label{lemma_limit_upper_bound}
    Let $\calF^\mu: L^2(\Omega, \mu)\times\calO \to [0, +\infty]$ be the $\overline{\Gamma}$-limit of $\{\Tilde{\calF}^\mu_h\}_{h>0}$, then for every $v\in H^1(\Omega, \mu)$ and every $A\in\calO$,
    $$
        \calF^\mu(v, A) \leq \int_A \langle\nabla v, \T \nabla v \rangle \dd \mu\, ,
    $$
    where $\T$ is defined in Lemma~\ref{lemma_properties_tensor}(iii). 
\end{lemma}
\begin{proof}
    Lemma~\ref{lemma_properties_tensor}(iii) gives, in particular, that if $\T^h \rightharpoonup^* \T$ in $L^\infty(\Omega)$ for some (not relabelled) subsequence, then
    $$
        \lim_{h\to 0} \int_\Omega \T^h(x) \Ind_{A_{\calT^h}} \hat{\mu}^h (\dd x) = \int_A \T(x) \mu(\dd x),
    $$
    which holds due to Proposition~\ref{prop_convergence_of_covering} and the generalized dominated convergence theorem. Therefore,
    $$
        \calF^\mu(\varphi_\xi, Q_\varepsilon(x)) \leq \liminf_{h\to 0} \Tilde{\calF}^\mu_h (\bbP \varphi_\xi, Q_\varepsilon(x)) \leq \Big\langle \xi, \int_{Q_\varepsilon(x)} \T(z) \mu (\dd z) \,\xi \Big\rangle.
    $$
    The representation formula \eqref{eq_integrand} yields
    $$
        f(x,\xi) \le \Big\langle \xi, ~\lim_{\varepsilon\to 0}  \intbar_{Q_\varepsilon(x)} \T(z) \mu (\dd z)\, \xi \Big\rangle= \Big\langle \xi,~\T(x) \frac{\dd \mu}{\dd \calL^d} (x)\, \xi \Big\rangle.
    $$
    Thus, for and $v\in H^1(\Omega, \mu)$ and $A\in\calO$,
    $$
        \calF^\mu (v, A) \leq \int_A \langle \nabla v, \T \nabla v\rangle \, \dd \mu \,,
    $$
    as required.
\end{proof}

\subsection{Dual dissipation potential}\label{sec_dissipation}

Recall that the dual dissipation potential has the form:
\begin{align*}
    \calR^*_h (\rho^h, \dnabla \varphi^h) = \frac{1}{2} \sum_{(K,L)\in\Sigma^h} \Psi^*\left( (\dnabla \varphi^h) (K, L) \right) \sqrt{u^h(K) u^h(L)} \,\vartheta^h(K, L), \quad \text{where } u^h = \frac{\dd \rho^h}{\dd \pi^h},
\end{align*}
with $\Psi^*(\xi) = 4 \left( \cosh{(\xi/2)} - 1 \right)$.

In Lemma~\ref{lemma_limit_upper_bound}, we derived an upper bound for the integral representation of the $\overline{\Gamma}$-limit of $\{\Tilde{\calF}^\rho_h\}_{h>0}$. We now prove that the same bound applies asymptotically to $\calR^*_h$ for some specific choice of approximating functions $\varphi^h$.
\begin{lemma}\label{lemma_limsup_dualdiss}
     Let $\varphi\in \calC_b^2(\Omega)$ and assume that $\{\rho^h\}_{h>0}$ is a family of probability measures on $\calT^h$ such that $\dd \hat{\rho}^h / \dd \calL^d\to \dd \rho / \dd \calL^d$ in $L^1(\Omega)$ (cf.\ Section~\ref{sec_general_gamma_convergence}(ii)). Moreover, let $\{\varphi^h\}_{h>0}$ be the family of discrete functions on $\{\calT^h\}_{h>0}$ defined by $\varphi^h(K) := \varphi(x_K)$ for $K\in\calT^h$. 
     
     Then $\bbL \varphi^h \to \varphi$ in $L^2(\Omega, \rho)$, and
     \begin{equation*}
        \limsup_{h\to 0} \calR^*_h(\rho^h, \dnabla\varphi^h) \leq \frac{1}{4} \int_\Omega \langle \nabla\varphi, ~ \T \nabla\varphi \rangle \dd \rho \,.
    \end{equation*}
\end{lemma}
\begin{proof} 
We first observe that $\bbL \varphi^h =: \hat{\varphi}^h \to \varphi$ in $L^2(\Omega, \rho)$. This follows directly from estimate
\begin{align*}
    \int_\Omega |\hat{\varphi}^h(x) - \varphi(x)|^2 \rho(\dd x)
    &= \sum_{K\in\calT^h} \int_K |\varphi(x_K) - \varphi(x)|^2 \rho(\dd x) \\
    &\leq \sum_{K\in\calT^h} \int_K \left(\|\nabla \varphi \|^2_{L^\infty} h^2 + o(h^2) \right) \rho(\dd x) \\
    &\leq \|\nabla \varphi \|^2_{L^\infty} h^2 + o(h^2).
\end{align*}
Now, we show that $\{\hat{\varphi}^h\}_{h>0}$ realises the upper bound for the $\overline{\Gamma}$-limit $\calF^\rho$ proven in Lemma~\ref{lemma_limit_upper_bound}, i.e.\
$$
    \limsup_{h\to 0} \Tilde{\calF}^\rho_h (\hat{\varphi}^h, A) \leq \int_A \langle \nabla \varphi ,~ \T \nabla\varphi\rangle \dd \rho.
$$
Since $\varphi$ is smooth, the discrete gradient for $\varphi^h$ can be approximated by
\begin{align*}
    (\dnabla \varphi^h)(K, L) = \varphi(x_L) - \varphi(x_K) &= \int_0^1 \langle (\nabla \varphi)(x_K + \tau (x_L - x_K)) ,~ x_L - x_K \rangle \dd\tau  \\
    &= \langle (\nabla \varphi)(x_K) ,~ x_L - x_K \rangle + o(h).
\end{align*}
Moreover, the difference between $(\nabla\varphi)(x_K)$ and $\intbar_K (\nabla\varphi)(x)\dd x$ is of a small order:
\begin{align*}
    \intbar_K (\partial_i \varphi) (x) \dd x - (\partial_i \varphi) (x_K) &= \intbar_K \int_0^1 \langle (\nabla \partial_i \varphi) (x_K + \tau (x - x_K)),~ x - x_K \rangle \dd\tau \dd x \\
    &= \left\langle (\nabla \partial_i \varphi) (x_K), \intbar_K (x - x_K)\dd x \right\rangle + o(h) = o(h),
\end{align*}
which implies that
$$
    (\dnabla \varphi^h)(K, L) =  \intbar_K \langle (\nabla \varphi)(x),~ x_L - x_K\rangle \dd x + o(h).
$$
Substituting $\dnabla \varphi^h$ into $\calF^\rho_h$ yields:
\begin{align*}
    \calF^\rho_h (\varphi^h, A) &= \sum_{(K, L)\in\Sigma^h|_A} \left| (\dnabla \varphi^h) (K, L) \right|^2 \kappa^h(K, L) \rho^h(K) \\
    &= \sum_{(K, L)\in\Sigma^h|_A} \left| \langle (\nabla \varphi)(x),~ x_L - x_K\rangle \dd x \right|^2 \kappa^h(K, L) \rho^h(K) + o(1)|_{h\to 0},
\end{align*}
where we used \eqref{assumption_kernel_uniform_upper_bound}. Applying Jensen's inequality we get:
\begin{align*}
    \calF^\rho_h (\varphi^h, A) &\leq \sum_{(K, L)\in\Sigma^h|_A}  \intbar_K \left| \langle (\nabla \varphi)(x),~ x_L - x_K\rangle \dd x  \right|^2 \dd x\,  \kappa^h(K, L) \rho^h(K) + o(1)|_{h\to 0} \\
    &\le \int_{A_{\calT^h}} \big\langle (\nabla \varphi)(x),~ \T^h(x) (\nabla \varphi)(x) \big\rangle  \hat{\rho}^h(\dd x) + o(1)|_{h\to 0}.
\end{align*}
By Lemma~\ref{lemma_properties_tensor}(iii) and Proposition~\ref{prop_convergence_of_covering} one can pass $h\to 0$ on the right-hand side to obtain:
\begin{align*}
    \limsup_{h\to 0} \calF^\rho_h (\varphi^h, A) \leq \int_A \big\langle (\nabla \varphi)(x),~ \T(x) (\nabla \varphi)(x) \big\rangle \rho(x) \dd x.
\end{align*}
The last step we need to take is to show that 
$\calR^*_h (\rho^h, \dnabla \varphi^h) = \calF_h^\rho(\varphi^h) + o(1)|_{h\to 0}$. We consider the expansion of $\calR^*_h$:
\begin{align*}
    \calR^*_h (\rho^h, \dnabla \varphi^h)
    &= \frac{1}{2} \sum_{(K,L)\in\Sigma^h} \Psi^*\left( (\dnabla \varphi^h) (K, L) \right) \sqrt{u^h(K) u^h(L)}\, \vartheta^h(K, L) \\
    &\leq \frac{1}{4} \sum_{(K,L)\in\Sigma^h} \left( \left| (\dnabla \varphi^h) (K, L) \right|^2 + g\left( (\dnabla \varphi^h) (K, L) \right)\right) \kappa^h(K, L) \rho^h(K) \\
    &\leq \frac{1}{4} \calF^\rho_h(\varphi^h) + \frac{1}{4} \sum_{(K,L)\in\Sigma^h} g\Bigl( (\dnabla \varphi^h) (K, L) \Bigr) \,\kappa^h(K, L) \rho^h(K),
\end{align*}
where
\[ g(r) = \Psi^*(r) - \frac{r^2}{2} = \displaystyle\sum_{k=2}^\infty \frac{4}{(2k)!} \left( \frac{r}{2} \right)^{2k} = O(r^4).
\]Since $|(\dnabla\varphi^h)(K, L)| \leq C \|\nabla \varphi\|_{L^\infty} h$ and recalling \eqref{assumption_kernel_uniform_upper_bound} once again, we conclude the proof.
\end{proof}

\subsection{Fisher information}\label{sec_fisher} In this section, we prove the $\Gamma$-convergence for the family of discrete Fisher information $\{\calD_h\}_{h>0}$ defined as
\begin{equation}\label{eq_fisher_info}
    \calD_h(\rho^h) = \sum_{(K,L)\in\Sigma^h} \left| \dnabla \sqrt{u^h} (K, L) \right|^2 \vartheta^h(K, L), \qquad \text{ with } u^h = \frac{\dd \rho^h}{\dd \pi^h}.
\end{equation}

We state the main result of this section.
\begin{theorem}\label{th_limit_fisher_info} Up to passing to a subsequence, the family of functionals $\{\calD_h\}_{h>0}$ has a $\Gamma$-limit $\calD$ w.r.t.\ the $L^2$-topology taking the form
\begin{equation}\label{eq_limit_Fisher}
    \calD(\rho) = \begin{cases}\displaystyle
        \int_\Omega \big\langle \nabla \sqrt{u}, \mathbb{T} \nabla \sqrt{u} \big\rangle \dd\pi & \text{if } \sqrt{\frac{\dd \rho}{\dd \pi}} =: \sqrt{u} \in H^1(\Omega), \\
        +\infty & \text{otherwise,}
    \end{cases}
\end{equation}
where $\T$ defined in Lemma~\ref{lemma_properties_tensor}.
\end{theorem}

\begin{remark}
    Since we assume that the densities $\{\dd \hat\pi^h / \dd \calL^d\}_{h>0}$ are uniformly bounded from above and away from 0, and $\dd \hat\pi^h / \dd \calL^d\to \dd \pi / \dd \calL^d$ in $L^1(\Omega)$, then $\pi$ is bounded in the same way. Consequently, the norms in $L^p(A, \pi)$ and $L^p(A)$ are equivalent.
\end{remark}

In Theorem~\ref{th_limit_fisher_info} we implicitly consider $\calD$ to depend on $\sqrt{\dd\rho / \dd\pi} \in L^2(\Omega)$ for all $\rho\ll\pi$ and take the $\Gamma$-limit in the corresponding topology. To simplify the notation we set $v^h:=\sqrt{u^h}$ and consider again the localized functional:
\begin{equation*}
    \calF^\pi_h(v^h, A) = \sum_{(K,L)\in\Sigma^h|_A} \left| \dnabla v^h (K, L) \right|^2 \vartheta^h(K, L).
\end{equation*}
Notice that $\calF^\pi_h(v^h, \Omega) = \calD_h(u^h \pi^h)$.
In what follows we set $\calF^\pi:=\Gamma$-$\lim \calF^\pi_h$, $\calF^\pi_{\inf}:=\Gamma$-$\liminf \calF^\pi_h$, and $\calF^\pi_{\sup}:=\Gamma$-$\limsup \calF^\pi_h$. 

Proposition~\ref{prop_properties_Gamma_limit} provides the existence of an integral representation:
$$
    \calF^\pi(v, A) = \int_A f(x, \nabla v) \dd \pi,\qquad v\in H^1(\Omega).
$$
Unlike in Section~\ref{sec_dissipation}, we are interested in the exact $\Gamma$-limit for $\calF^\pi_h$. In fact, Proposition~\ref{prop_properties_Gamma_limit} provides almost all necessary properties to apply a general representation theorem \cite[Theorem 2]{bouchitte2002global}, except for a Sobolev lower bound. Therefore, we show the lower bound in Lemma~\ref{lemma_bounds_Dsup}, and then proceed to the representation. 

\begin{lemma}\label{lemma_bounds_Dsup}
    For any $v \in H^1(\Omega)$ and $A \in \calO$ it holds that:
    \begin{equation*}
         \calF^\pi_{\sup}(v, A) \geq \frac{C_l}{C_\zeta \pi_{\max}} \int_A \left| \nabla v \right|^2 \dd \pi \ge \frac{C_l \pi_{\min}}{C_\zeta \pi_{\max}} \int_A \left| \nabla v \right|^2 \dd x.
    \end{equation*}
\end{lemma}
\begin{proof}
    Let $\{v_h\}_{h>0}\in L^2(\Omega)$ be a sequence with $v_h \to v$ in $L^2(\Omega)$ such that
    \[
        \calF^\pi_{\sup}(v, A) = \limsup_{h\to 0} \Tilde \calF^\pi_h(v_h, A).
    \]
    In particular, $v_h\in PC(\calT^h|_A)$ for all $h>0$.
    
    To prove the lower bound, one can repeat some elements of the proof of Lemma \ref{lemma_BV_bound}. We fix an arbitrary $\varepsilon>0$ and denote $A_\varepsilon := \left\{ x\in A \,:\, \text{dist}(x, \partial A) > \varepsilon \right\}$. Let $\eta\in\R^d$ be such that $|\eta| < \varepsilon$, then 
    \begin{align*}
        \int_{A_\varepsilon} | v_h(x+\eta) &- v_h(x) |^2 \dd x
        = \sum_{K\in\calT^h|_{A_\varepsilon}} \int_{K\cap A_\varepsilon} \left| v_h(x+\eta) - v_h(x) \right|^2 \, \dd x \\
        &\leq \sum_{K\in\calT^h|_{A_\varepsilon}} \int_{K\cap A_\varepsilon} n \sum_{i=0}^{n-1} \left| v_h(K_{i+1}) - v_h(K_i) \right|^2 \, \dd x \\
        &\leq \frac{C_\zeta|\eta|}{h} \sum_{K\in\calT^h|_{A_\varepsilon}} \int_{K\cap A_\varepsilon} \sum_{(M,L)\in\Sigma^h}\left|v_h(L) - v_h(M)\right|^2 \Ind_{\text{Cyl}_{\Sigma^h}(x,\eta)}(M,L)\, \dd x.
    \end{align*}
    where we used Proposition~\ref{prop_conseq_zeta_regularity} to assert that $n\le C_\zeta|\eta|/h$ for some constant $C_\zeta>0$, independent of $x$. Now one can repeat the transformations from the proof Lemma~\ref{lemma_BV_bound} to obtain:
    \begin{align*}
          \int_{A_\varepsilon} | v_h(x+\eta) &- v_h(x) |^2 \dd x \leq \frac{C_\zeta}{C_l}|\eta|^2 \calD_h(v_h, A_\varepsilon).
    \end{align*}
    Passing to the limit superior as $h\to0$ then yields
    $$
        \calF^\pi_{\sup}(v, A_\varepsilon) \geq \frac{C_l}{C_\zeta} \frac{\left\|v(\cdot+\eta) - v \right\|^2_{L^2(A_\varepsilon)}}{|\eta|^2}.
    $$
    For $v\in H^1(\Omega)$, passing $|\eta|\to 0$ yields
    $$
        \calF^\pi_{\sup}(v, A_\varepsilon) \geq \frac{C_l}{C_\zeta} \int_{A_\varepsilon} \left| \nabla v \right|^2 \dd x \geq \frac{C_l}{C_\zeta \pi_{\max}} \int_{A_\varepsilon} \left| \nabla v \right|^2 \pi(\dd x).
    $$
    Since $\calF^\pi_{\sup}$ is inner regular (Proposition~\ref{prop_properties_F_sup}), then 
    $$
        \calF^\pi_{\sup}(v, A) = \sup_{\varepsilon>0} \calF^\pi_{\sup}(v, A_\varepsilon) \geq \frac{C_l}{C_\zeta \pi_{\max}} \int_{A} \left| \nabla v \right|^2 \pi(\dd x)\,,
    $$
    where the inequality follows from the monotone convergence theorem.
\end{proof}

Now we are in the position to use the following proposition from \cite[Theorem 2]{bouchitte2002global}.
\begin{proposition}\label{prop_reconstruction}
    Let $\calF: H^1(\Omega)\times \calO \to [0, +\infty]$ be a functional satisfying: 
    \begin{enumerate}[label=(\roman*)]
        \item $\calF(v, \cdot)$ is the restriction to $\calO$ of a Radon measure;
        \item $\calF( \cdot, A)$ is $L^2(\Omega)$ lower semicontinuous;
        \item $\calF( \cdot, A)$ is local, which means $\calF(v, A) = \calF(w, A)$ if $v=w$ $\calL^d$-a.e. on $A$;
        \item $\calF(v + c, A) = \calF(v, A)$ for every $c\in \R$;
        \item $\calF(v, A)$ satisfies the growth condition:
        $$
            \frac{1}{C} \int_A \left|\nabla v \right|^2 \dd x \leq \calF(v, A) \leq C \int_A \left( 1 + \left|\nabla v \right|^2 \right) \dd x
        $$
        for some $C>0$. 
    \end{enumerate}
    Then for every $v\in H^1(\Omega)$ and $A\in\calO$
    $$
        \calF(v, A) = \int_A f(x, v, \nabla v) \dd x,
    $$
    where
    $$
        f(x_0, v_0, \xi) := \limsup_{\varepsilon\to 0+} \frac{1}{\left| Q_\varepsilon (x) \right|} \text{M}(\langle v_0 + \xi\cdot(\cdot - x), Q_\varepsilon (x))
    $$
    for all $x_0\in\Omega$, $v_0\in\R^d$, $\xi\in\R^d$, and where, for $(v, A)\in H^1(\Omega)\times\calO$,
    $$
        \text{M}(v, A) := \inf\left\{ \calF(w, A)~:~ w\in H^1(A) \text{ with } v=w \text{ in a neighborhood of } \partial A \right\}.
    $$
\end{proposition}

Applying Proposition~\ref{prop_reconstruction} to our setting, we notice that $f$ does not depend explicitly on $v$ (since $\calF^\mu(v + c, A) = \calF^\mu(v, A)$ for every $c\in \R$ Proposition~\ref{prop_properties_Gamma_limit}). This gives
\begin{equation}\label{eq_representation}
    \calF^\pi(v, A) = \int_A f_\pi (x, \nabla v) \dd x \quad \text{with } f_\pi(x, \xi) = \limsup_{\varepsilon\to 0} \frac{1}{\left| Q_\varepsilon (x) \right|} \text{M}( \xi\cdot (\cdot - x) , Q_\varepsilon (x)).
\end{equation}

The identification formula \eqref{eq_representation} suggests looking for the minimizer of $\psi\mapsto \calF^\pi(\psi, A)$ w.r.t.\ the Dirichlet boundary condition. To relate this minimization problem to our discrete formulation, we follow similar approach as in \cite{forkert2020evolutionary}, specifically, we define
$$
    \text{M}_h(\varphi^h, A) := \inf_{\psi^h} \Bigl\{ \calF^\pi_h (\psi^h, A) ~ : ~ \psi^h \text{ on } \calT^h|_A \quad\text{with}\quad  \psi^h = \varphi^h \text{ on } \calT^h|_{A^c} \Bigr\}.
$$
We also make use of the following definition of $\text{M}$
$$
    \text{M}(\varphi, A) = \inf_{\psi} \Bigl\{ \calF^\pi(\psi, A) ~ : ~ \psi \in H^1(A) \quad\text{with}\quad \psi-\varphi \in H^1_0(A) \Bigr\},
$$
that was proven to be equivalent to the one from Proposition~\ref{prop_reconstruction}
in \cite[Remark~7.4]{forkert2020evolutionary}.

Note that the $\Gamma$-convergence of $\calF^\pi_h$ to $\calF^\pi$ does not suffice to conclude the convergence of $M_h$ to $M$. Hence, we define
$$
    \calF^{\pi, \varphi}(v, A) := \begin{cases}
            \calF^\pi(v, A) & \text{if } v - \varphi\in H^1_0(A), \\
            +\infty & \text{otherwise,}
        \end{cases}
$$
and the corresponding discrete counterpart of $\calF^{\pi, \varphi}$ is
$$
    \calF^{\pi, \varphi}_h(v^h, A) := \begin{cases}
            \calF^\pi_h(v^h, A) & \text{if } v^h = \bbP \varphi =: \varphi^h \text{ on } A^c_{\calT^h}, \\
            +\infty & \text{otherwise.}
        \end{cases}
$$
Similarly to \cite[Lemma 7.9]{forkert2020evolutionary}, the next proposition claims that $\calF^{\pi, \varphi}_h(\cdot, A) \xrightarrow{\Gamma} \calF^{\pi, \varphi}(\cdot, A)$ for any $A\in\calO$ with Lipschitz boundary. For completeness, we include the proof in Appendix~\ref{appendix_proofs}.
\begin{proposition}\label{prop_Gamma_convegence_bc}
    Let $A\in\calO$ 
    and $\varphi\in\text{Lip}(\Omega)$. For any sequence $\{\calF^{\pi, \varphi}_h(\cdot, A)\}_{h>0}$, there exists a subsequence that $\Gamma$-converges in the $L^2(\Omega)$-topology to $\calF^{\pi, \varphi}(\cdot, A)$.
\end{proposition}

Now we comment on the assumptions on tessellations and kernels needed for proving the representation result. It suffices to use \eqref{assumption_nu} and \eqref{assumption_local_average}. The last assumption  \eqref{assumption_local_average} was not necessary for any preceding statements and its role here is to ensure that the discrete functions
$$
    \varphi_h^{x, \xi}(K) := \langle \xi, x_K - x \rangle \quad \text{ for  all } K\in\calT^h
$$
are minimizers for $\calF^{\pi, \varphi}_h(\cdot, Q_\varepsilon(x))$, where $x_K = \intbar_K x \dd x$. To relax \eqref{assumption_local_average} we can introduce an asymptotic assumption involving \emph{almost minimizers}. Namely, we assume that
\begin{equation}\label{assumption_almost_minimizer}
    \lim_{h\to 0} \left( \calF^\pi_h (\varphi_h^{x, \xi}, Q_\varepsilon(x)) - \text{M}_h (\varphi_h^{x, \xi}, Q_\varepsilon(x)) \right) = 0.
    \tag{{\sf AMin}}
\end{equation}

Finally, we state the representation result.
\begin{theorem}\label{th_identification} Let $\calF^\pi_h : L^2(\Omega, \pi)\times\calO \to [0, +\infty]$ be the $\overline{\Gamma}$-limit of $\{\Tilde{\calF}_h\}_{h>0}$ defined as in Lemma~\ref{lemma_existence_gamma_limit}, then
the functional $\calF^\pi(v, A)$ has the integral representation
\begin{equation*}
    \calF^\pi(v, A) = \begin{cases}
        \displaystyle \int_A  \langle \nabla v, \T \nabla v \rangle \dd \pi, &\text{if } v\in H^1(A, \pi), \\
        +\infty, &\text{otherwise.}
    \end{cases} 
    \end{equation*}
with the tensor $\mathbb{T}$ defined in Lemma~\ref{lemma_properties_tensor}.
\end{theorem}
\begin{proof}Proposition~\ref{prop_Gamma_convegence_bc} and the theorem fundamental on convergence of minimizers (see, for instance, \cite[Theorem 1.21]{braides2002gamma}) together with \eqref{assumption_almost_minimizer} provides:
    $$
        \text{M} (\varphi^{x,\xi}, Q_\varepsilon(x)) = \lim_{h\to 0} \text{M}_h(\varphi_h^{x,\xi}, Q_\varepsilon(x)) = \lim_{h\to 0} \calF^\pi_h(\varphi_h^{x,\xi}, Q_\varepsilon(x)).
    $$
    Substituting $\varphi_h^{x,\xi}(K) = \langle \xi, x - x_K \rangle$ into $\calF^\pi_h$ yields
    \begin{align*}
        \calF^\pi_h(\varphi_h^{x,\xi}, Q_\varepsilon(x)) &= \sum_{(K,L)\in\Sigma^h|_{Q_\varepsilon(x)}} \vartheta^h(K, L) \left| \langle \xi, x_L - x_K \rangle \right|^2 = \bigl\langle \xi,\T^{h,\varepsilon}(x) \xi\bigr\rangle ,
    \end{align*}
    with
    \[
        \T^{h,\varepsilon}(x) := \sum_{(K,L)\in\Sigma^h|_{Q_\varepsilon(x)}} \vartheta^h(K, L) (x_L - x_K) \otimes (x_L - x_K).
    \]
    Since $\T^h$ defined in Lemma~
    \ref{lemma_properties_tensor} is piecewise constant on the tessellation, we can rewrite $\T^{h,\varepsilon}$ as
    \begin{align*}
        \T^{h,\varepsilon}(x) = \sum_{K\in\calT^h|_{Q_\varepsilon(x)}} \pi^h(K) \intbar_K \T^h(z) \dd z 
        = \int_{[Q_\varepsilon(x)]_{\calT^h}} \T^h(z)\, \hat{\pi}^h(\dd z).
    \end{align*}
    Using that $\Ind_{Q_\varepsilon(x)_{\calT^h}} \dd \hat{\pi}^h / \dd \calL^d \to \Ind_{Q_\varepsilon(x)} \dd \pi / \dd \calL^d$ in $L^1(\Omega)$ and Lemma~\ref{lemma_properties_tensor}(iii), we then obtain
    $$
        \lim_{h\to 0} \T^{h,\varepsilon}(x) = \int_{Q_\varepsilon(x)} \T(z)\, \pi (\dd z).
    $$
    Therefore,
    \begin{align*}
        \text{M} (\varphi^{x,\xi}, Q_\varepsilon(x)) = \lim_{h\to 0} \calF^\pi_h(\varphi_h^{x,\xi}, Q_\varepsilon(x)) = \lim_{h\to 0} \bigl\langle \xi, \T^{h, \varepsilon}(x)\, \xi \bigr\rangle = \left\langle \xi, \int_{Q_\varepsilon(x)} \T(z)\, \pi (\dd z)\, \xi \right\rangle.
    \end{align*}
    Finally, we substitute $\text{M}$ into the expression \eqref{eq_representation} for $f$, and obtain for almost every $x\in\Omega$:
    \begin{align*}
        f(x, \xi) &= \limsup_{\varepsilon\to 0+} \frac{1}{\left| Q_\varepsilon (x) \right|} \text{M}(\varphi^{x,\xi}, Q_\varepsilon (x)) \\
        &= \left\langle \xi, \lim_{\varepsilon\to 0+}  \intbar_{Q_\varepsilon(x)} \T(z) \pi (\dd z)\, \xi \right\rangle = \left\langle \xi, \T(x) \frac{\dd \pi}{\dd \calL^d} (x)\, \xi \right\rangle,
    \end{align*}
    thereby concluding the proof.
\end{proof}

\begin{cor}
    If \eqref{assumption_local_average} holds, then the functions $\varphi^{x, \xi}_h$ are minimizers for $\calF^\pi_h(\cdot, Q_\varepsilon(x))$, i.e.\
    $$
        \calF^\pi_h(\varphi_h^{x,\xi}, Q_\varepsilon(x)) = \text{M}_h(\varphi_h^{x,\xi}, Q_\varepsilon(x)) \quad \text{for any } (x, \xi) \in \Omega\times \R^d,\, h>0,\, \varepsilon>0.
    $$
    In particular, the conclusion of Theorem~\ref{th_identification} holds true.
\end{cor}
\begin{proof}
    Computing the first variation for $\calF^\pi_h(\varphi^h, Q_\varepsilon(x))$ gives
    \begin{align*}
        \delta \calF^\pi_h (\varphi^h, Q_\varepsilon(x)) [w_h] = 2\sum_{(K,L)\in\Sigma^h|_{Q_\varepsilon(x)}}  \vartheta^h(K, L) \left( \varphi^h(L) - \varphi^h(K) \right) \left( w^h(L) - w^h(K) \right),
    \end{align*}
    where $w^h$ satisfies the boundary condition $w^h = 0$ on $[Q^c_\varepsilon(x)]_{\calT^h}$.
    
    Substituting $\varphi_h^{x,\xi}$ and then using the symmetry, we have
    \begin{align*}
        \delta \calF^\pi_h (\varphi^{x,\xi}_h, Q_\varepsilon(x)) [w_h]
        &= 2 \sum_{(K,L)\in\Sigma^h|_{Q_\varepsilon(x)}} \vartheta^h(K, L) \langle \xi, x_L - x_K \rangle \left( w_h(L) - w_h(K) \right) \\
        &= 4 \sum_{K\in\calT^h|_{Q_\varepsilon(x)}} w_h(K)  \Big\langle \xi, \sum_{L\in\calT^h_K|_{Q_\varepsilon(x)}} \vartheta^h(K, L)  (x_L - x_K) \Big\rangle.
    \end{align*}
    Notice that the boundary condition implies that the summation goes over cells $K \subset Q_\varepsilon(x)$ strictly contained within $Q_\varepsilon(x)$. This means that the inside sum goes over all the neighbors of the cell $K$. This allows us to apply assumption \eqref{assumption_local_average} to obtain
    $$
        \sum_{L \in\calT^h_K} \vartheta^h(K, L) \left( x_K - x_L \right) = 0,
    $$
    and conclude that $\delta \calF^\pi_h (\varphi_h^{x,\xi}, Q_\varepsilon(x)) [w_h] = 0$ for all $w_h$ with $w_h\equiv 0$ on $[Q_\varepsilon(x)]^c_{\calT^h}$.  Since $\calF^\pi_h(\cdot, Q_\varepsilon(x))$ is convex, this implies that $\varphi^{x,\xi}_h$ is the minimizer.
\end{proof}

\begin{proof}[Proof of Theorem~\ref{th_limit_fisher_info}] The result readily follows from Theorem~\ref{th_identification}.
\end{proof}

\begin{lemma}
    The diffusion tensor $\T$ is uniformly elliptic and uniformly bounded:
    $$
        \lambda |\xi|^2 \leq \langle \xi,~ \T(x)\, \xi \rangle \leq \Lambda |\xi|^2 \quad \text{for any } x\in\Omega \text{ and } \xi\in\R^d,
    $$
    with some $\lambda, \Lambda > 0$.
\end{lemma}
\begin{proof}
    The upper bound follows from Lemma~\ref{lemma_properties_tensor}(2) with $\Lambda = 2 C_\kappa$. The lower bound can be deduced from Theorem~\ref{th_identification}. Indeed, since
    \begin{align*}
        \langle \xi, \T(x)\, \xi \rangle &= \left\langle \xi, \lim_{\varepsilon\to 0} \intbar_{Q_\varepsilon(x) }\T(z) \dd z ~ \xi \right\rangle 
        = \lim_{\varepsilon\to 0} \intbar_{Q_\varepsilon(x)} \langle \xi, \T(z)\, \xi \rangle \dd z \\
        &= \lim_{\varepsilon\to 0} \frac{1}{|Q_\varepsilon(x)|} \calF^\pi \left( \langle \cdot, \xi \rangle, Q_\varepsilon(x) \right),
    \end{align*}
    applying the lower bound for $\calF^\pi$ from Lemma~\ref{lemma_bounds_Dsup} gives
    \begin{align*}
        \langle \xi, \T(x)\, \xi \rangle \geq \lim_{\varepsilon\to 0} \frac{C_l}{C_\zeta \pi_{\max}} \frac{2}{|Q_\varepsilon(x)|} \int_{Q_\varepsilon(x)} |\xi|^2 \dd \pi 
        \geq 2 \frac{C_l \pi_{\min}}{C_\zeta \pi_{\max}} |\xi|^2,
    \end{align*}
    from which the lower bound follows.
\end{proof}

\section{Convergence result}\label{sec_result}

With the results above, we are now in the position of proving our main result, Theorem~\ref{th_main_result}. We begin by recalling our reconstruction procedure for discrete density-flux pairs from the beginning of Section~\ref{section_reconstr_compact}. We then proceed to show $\liminf$ inequalities for the functionals $\calE_h$, $\calD_h$ and $\calR_h$, therewith establishing the $\liminf$ inequality for the energy-dissipation functional $\calI_h$ (cf.\ Theorem~\ref{th_liminf_inequalities}). The chain rule is established in Section~\ref{sec_chainrule}, which is essential in guaranteeing the nonnegativity of the limit energy-dissipation functional $\calI$. Finally, we conclude with the proof of Theorem~\ref{th_main_result}.

\subsection{Lim inf inequalities}

Given a density-flux pair $(\rho^h, j^h)\in\mathcal{CE}_h(0,T)$ we define
\begin{equation}\label{eq_lift_pair}
    \frac{\dd\hat{\rho}^h}{\dd \calL^d} := \sum_{K\in\calT^h} \frac{\rho^h(K)}{|K|} \Ind_K, \qquad \hat{\jmath}^h := \sum_{(K,L) \in \Sigma^h} j^h(K, L) \sigma_{KL},
\end{equation}
where $\sigma_{KL}\in\calM(\Omega; \R^d)$ is defined in the way that
$(\hat{\rho}^h, \hat{\jmath}^h)\in\mathcal{CE}(0,T)$ (we constructed $\sigma_{KL}$ explicitly in Lemma~\ref{lemma_relation_gradients}). 

\begin{definition}[Density-flux convergence]\label{def_measure_flux_convergence}
    A discrete density-flux pair $(\rho^h, j^h)\in\mathcal{CE}_h(0,T)$ is said to converge to a density-flux pair $(\rho, j)\in\mathcal{CE}(0,T)$ if the pair of reconstructions $(\hat \rho^h, \hat\jmath^h)\in\mathcal{CE}(0,T)$ defined as in \eqref{eq_lift_pair} converges in the following sense
    \begin{enumerate}
        \item $\dd \hat\rho^h_t/\dd \calL^d \to \dd \rho_t/\dd \calL^d$ in $L^1(\Omega)$ for almost every $t\in[0, T]$,
        \item $\int_{\cdot}\hat\jmath_t^h \,\dd t \rightharpoonup^* \int_{\cdot} j_t\,\dd t$ in $\calM([0, T]\times\Omega)$.
    \end{enumerate}
\end{definition}

We now summarize the lower bounds for all components of the energy-dissipation functional $\calI_h$ defined in Section~\ref{sec_generalized_gf}. The form of the lower bounds is already suggested by Lemma~\ref{lemma_limsup_dualdiss} for the dissipation potential $\calR$ and by Theorem~\ref{th_limit_fisher_info} for the Fisher information $\calD$. Let us first give the definitions of $\calR$, $\calR^*$, $\calD$, and $\calE$ and then summarize the corresponding $\liminf$ inequalities in Theorem~\ref{th_liminf_inequalities}.

\medskip

\noindent{\em The dual dissipation potential} $\calR^*: \calP(\Omega)\times \calC_c^2(\Omega) \to [0, \infty)$ takes the form
\begin{equation*}
    \calR^*(\rho, \varphi) = \frac{1}{4} \int_\Omega \langle \nabla\varphi, \T\nabla\varphi \rangle \dd \rho.
\end{equation*}

\noindent{\em The dissipation potential} $\calR: \calP(\Omega)\times \calM(\Omega; \R^d) \to [0, +\infty]$ is $$
    \calR(\rho, j) = \begin{cases}
    \displaystyle\frac{1}{4} \int_\Omega  \Big\langle \frac{\dd j}{\dd \rho}, \T^{-1} \frac{\dd j}{\dd \rho} \Big\rangle \dd \rho &\text{if $j\ll \rho$}, \\
    +\infty &\text{otherwise}
    \end{cases}
$$

\noindent{\em The Fisher information} $\calD: \calP(\Omega) \to [0, +\infty]$ is defined as
\begin{equation*}
    \calD(\rho) = \begin{cases}\displaystyle
        \int_\Omega \big\langle \nabla \sqrt{u}, \mathbb{T} \nabla \sqrt{u} \big\rangle \dd\pi & \text{if } \sqrt{\frac{\dd \rho}{\dd \pi}} =: \sqrt{u} \in H^1(\Omega), \\
        +\infty & \text{otherwise.}
    \end{cases}
\end{equation*}

\noindent{\em The energy functional} $\calE: \calP(\Omega) \to [0, +\infty]$ is given by $\calE(\rho) =\Ent(\rho|\pi)$.

\begin{theorem}\label{th_liminf_inequalities} Let $(\rho^h, j^h)\in\mathcal{CE}_h(0,T)$ converge to $(\rho, j)\in\mathcal{CE}(0,T)$ in the sense of Definition~\ref{def_measure_flux_convergence}. Then the following lower bounds hold for
\begin{enumerate}[label=(\roman*)]
    \item the dissipation potential:
    $$
        \liminf_{h\to 0} \int_0^T \calR_h(
        \rho^h_t, j^h_t) \dd t \geq \int_0^T \calR(\rho_t, j_t) \dd t;
    $$
    \item the Fisher information:
    $$
        \liminf_{h\to 0} \int_0^T \calD_h(
        \rho^h_t) \dd t \geq \int_0^T \calD(\rho_t) \dd t;
    $$
    \item the energy functional:
    $$
        \liminf_{h\to 0} \calE_h(\rho^h_t) \geq \calE(\rho_t)\qquad\text{for all $t\in[0,T]$.}
    $$
\end{enumerate}
\end{theorem}
\begin{proof}
    {\em (i) The dissipation potential.} We employ the dual formulation of $\calR$. Let $\chi\in \calC^\infty_c ((0,T))$ and $\varphi\in \calC_c^\infty(\Omega)$ be arbitrary. Then from the weak$^*$-convergence of $\int_{\cdot} \hat\jmath_t^h\,\dd t$ and Lemma~\ref{lemma_limsup_dualdiss} we obtain
    \begin{align*}
        \int_0^T \langle \chi(t)\nabla\varphi,j_t\rangle - \calR^*(\rho_t,\chi(t)\nabla\varphi) \dd t
        \le \lim_{h\to 0} \int_0^T \langle \chi(t)\nabla\varphi,\,\hat\jmath_t^h\rangle \dd t - \limsup_{h\to 0} \int_0^T \calR_h^*(\rho_t^h,\chi(t)\dnabla\varphi^h)\dd t,
    \end{align*}
    where $\varphi^h(K) = \varphi(x_K)$ for all $K\in\calT^h$. For the first term on the right-hand side, we have
    \[
        \int_0^T \langle \chi(t)\nabla\varphi,\,\hat\jmath_t^h\rangle \dd t  = \int_0^T \langle \chi(t)\dnabla \bbP\varphi,j_t^h\rangle \dd t = \int_0^T \langle \chi(t)\dnabla \varphi^h,j_t^h\rangle \dd t + o(h),
    \]
    owing to the regularity of $\varphi$ and Lemma~\ref{lemma_properties_flux}, and therefore
    \[
        \lim_{h\to 0} \int_0^T \langle \chi(t)\nabla\varphi,\hat\jmath_t^h\rangle \dd t = \lim_{h\to 0} \int_0^T \langle \chi(t)\dnabla \varphi^h,j_t^h\rangle \dd t.
    \]
    Consequently, we obtain
    \begin{align*}
        \int_0^T \langle \chi(t)\nabla\varphi,j_t\rangle - \calR^*(\rho_t,\chi(t)\nabla\varphi) \dd t &\le \lim_{h\to 0} \int_0^T \langle \chi(t)\dnabla \varphi^h,j_t^h\rangle \dd t - \limsup_{h\to 0} \int_0^T \calR_h^*(\rho_t^h,\chi(t)\dnabla\varphi^h)\dd t \\
        &\le \liminf_{h\to 0} \int_0^T \langle \chi(t)\dnabla \varphi^h,j_t^h\rangle - \calR_h^*(\rho_t^h,\chi(t)\dnabla\varphi^h)\dd t \\
        &\le \liminf_{h\to 0} \int_0^T \calR_h(\rho_t^h,j_t^h) \dd t.
    \end{align*}
    To conclude, we will make use of Legendre--Fenchel's duality. In what follows, we set $P:=\int_{\cdot} \rho_t\,\dd t$, $J:=\int_{\cdot} j_t\,\dd t$, and $\mathcal{V}$ as the closure in $L^2(\Omega,P;\R^d)$ of the subspace $V:=\{ \T^{1/2}\nabla(\chi\varphi)\,:\, \chi\in \calC_c^\infty((0,T)),\,\varphi\in \calC_c^\infty(\Omega)\}$, where $\T^{1/2}$ denotes the square root of the positive definite matrix $\T$. Writing the term on the left in the previous inequality as
    \[
        \iint_{(0,T)\times\Omega} \T^{1/2}\nabla(\chi\varphi)\cdot \T^{-1/2}\frac{\dd J}{\dd P}\,\dd P - \frac{1}{2} \|\T^{1/2}\nabla(\chi\varphi)\|_{L^2(\Omega,P;\R^d)}^2,
    \]
    the Fenchel--Moreau duality theorem then gives
    \begin{align*}
        &\sup_{\psi\in V}\left\{ \iint_{(0,T)\times\Omega} \T^{1/2}\nabla(\chi\varphi)\cdot \T^{-1/2}\frac{\dd J}{\dd P}\,\dd P - \frac{1}{2} \|\T^{1/2}\nabla(\chi\varphi)\|_{L^2(\Omega,P;\R^d)}^2\right\} \\
        &\qquad= \sup_{\psi\in \mathcal{V}}\left\{ \iint_{(0,T)\times\Omega} \psi\cdot \T^{-1/2}\frac{\dd J}{\dd P}\,\dd P - \frac{1}{2} \|\psi\|_{L^2(\Omega,P;\R^d)}^2\right\} \\
        &\qquad= \frac{1}{2} \left\|\T^{-1/2}\frac{\dd J}{\dd P} \right\|_{L^2(\Omega,P;\R^d)}^2 = \int_0^T \calR(\rho_t,j_t)\dd t\,,
    \end{align*}
    where the last equality follows from the fact that $(\dd J/\dd P)(t,x) = (\dd j_t/\dd\rho_t)(t,x)$ for $P$-almost every $(t,x)\in(0,T)\times\Omega$.

{\em (ii) The Fisher information.} Since $\dd\hat\rho_t^h/\dd\calL^d \to \dd\rho_t/\dd\calL^d$ strongly in $L^1(\Omega)$ for almost every $t\in(0,T)$, we have by Theorem~\ref{th_limit_fisher_info} that
\[
    \liminf_{h\to 0}\calD_h(
        \hat\rho^h_t) \ge \calD(\rho_t)\qquad\text{for almost every $t\in(0,T)$.}
\]
Applying Fatou's lemma then yields
\[
    \liminf_{h\to 0} \int_0^T \calD_h(
        \rho^h_t) \dd t \ge  \int_0^T \liminf_{h\to 0}\calD_h(
        \rho^h_t) \dd t.
\]
    
{\em (iii) The energy functional.} As the following calculations hold for any $t\in[0, T]$, we drop the subscript $t$. We recall that
$$
    \calE_h(\rho^h) = \begin{cases}
        \displaystyle\sum_{K\in\calT^h} \phi \left( u^h(K) \right) \pi^h(K) & \text{ if } \rho^h\ll\pi^h, \text{ with } u^h(K) = \frac{\rho^h(K)}{\pi^h(K)}, \\
            +\infty & \text{otherwise,}
    \end{cases}
$$
where $\phi(z) = z \log z - z + 1$. Since $\rho^h$ and $\pi^h$ are probability measures,
    $$
        \calE_h(\rho^h) = \sum_{K\in\calT^h} u^h(K) \log \left( u^h(K) \right) \pi^h(K)\qquad\text{if $\rho^h\ll\pi^h$.}
    $$
   Our piecewise constant reconstruction provides that
    $$
        \calE_h(\rho^h) = \int_\Omega \hat u^h(x) \log \bigl( \hat u^h(x) \bigr) \hat\pi^h(\dd x) = \Ent (\hat{\rho}^h | \hat{\pi}^h).
    $$
    The narrow convergence of $\rho^h$ and $\pi^h$ in $\calP(\Omega)$, along with the joint lower semicontinuity of the relative entropy \cite[Lemma 9.4.3]{ambrosio2008gradient} then gives
    $$
        \liminf_{h\to 0} \calE_h(\rho^h) = \liminf_{h\to 0} \Ent(\hat\rho^h|\hat\pi^h) \geq \Ent(\rho|\pi)=\calE(\rho),
    $$
    as required.
\end{proof}

\subsection{Chain rule}\label{sec_chainrule} In this section we aim to establish the chain rule inequality:
\[
   -\frac{\dd}{\dd t} \calE(\rho_t) \le \calR(\rho_t, j_t) + \calD(\rho_t)\qquad\text{for almost every $t\in(0,T)$,}
\]
from which we establish the nonnegativity of the limit energy-dissipation functional, i.e.\
$$
    \calI(\rho, j) = \int_0^T \left\{ \calR(\rho_t, j_t) + \calD(\rho_t) \right\} \dd t + \calE(\rho_T) - \calE(\rho_0) \geq 0.
$$
We will show that this inequality can be obtained from the chain rule for the relative entropy $\Ent (\rho| \pi)$ along $W_2$-absolutely continuous curves.

\medskip

We begin by rewriting $\Ent (\rho| \pi)$ in a more convenient form for the purpose of this section. We denote $V := - \log\left(\dd\pi / \dd\calL^d\right)$. If measures $\rho$ and $\pi$ have Lebesgue densities, then it holds that
\begin{align*}
    \Ent(\rho|\pi) = \int_{\Omega} \frac{\dd \rho}{\dd \pi} \log \frac{\dd \rho}{\dd \pi} \dd \pi 
    &= \int_\Omega \frac{\dd \rho}{\dd \calL^d} \log \frac{\dd \rho}{\dd \calL^d} \dd \calL^d + \int_\Omega V \dd \rho = \Ent(\rho|\calL^d) + \int_\Omega V\,d\rho,
\end{align*}
which can be justified by monotone convergence. Recall that by the assumptions on $\{\pi^h\}_{h\geq 0}$ (Section~\ref{sec_assumptions_relation}) $V\in \Lip_b(\Omega)$, and, therefore, $\Ent(\rho|\calL^d)$ is finite whenever $\Ent(\rho|\pi)$ is finite.

Now we extend the definition of the energy for all measures $\rho$ with Lebesgue densities. First, we define an extended potential $V_E:\R^d\to (-\infty,+\infty]$ by
$$
    V_E(x) := \begin{cases}
        V(x) & \text{if } x\in\overline{\Omega}, \\
        +\infty & \text{otherwise.}
    \end{cases}
$$
Since $V\in \calC_b(\overline\Omega)$, $V_E$ is a lower semicontinuous on $\R^d$. Then for any $\rho\in \calP(\R^d)$ we consider the extended energy functional $\calE_E:\calP_2(\R^d) \to (-\infty,+\infty]$ defined by
$$
    \calE_E(\rho) := \begin{cases}
        \displaystyle\; \Ent(\rho|\calL^d) + \int_{\R^d} V_E \dd \rho & \text{for $\rho\ll \calL^d$,} \\
        \;+\infty &\text{otherwise.}
    \end{cases}
$$
We remark that the functionals $\calE_E$ and $\Ent(\rho|\pi)$ coincide on their sublevel sets. We also mention that $\Ent(\rho|\calL^d) > -\infty$ if $\rho\in\calP_2(\R^d)$ \cite{ambrosio2008gradient}.

\medskip

The following lemma results from a minor modification of \cite[Theorem~10.4.13]{ambrosio2008gradient}. In our case, $V$ is not $\lambda$-convex, but the result remains true due to the regularity assumed on $V$, i.e.\ $V\in \text{Lip}_b(\Omega)$.

\begin{lemma}\label{lemma_differential}
    A measure $\rho=\varrho\calL^d \in \text{dom}(\calE_E)$ belongs to $\text{dom} (\partial \calE_E)$ if and only if $\varrho\in W^{1, 1}_{loc}(\Omega)$ and
    \begin{equation}\label{eq_vector_field}
        \varrho w = \nabla\varrho + \varrho \nabla V_E \qquad \text{for some $w\in L^2(\R^d, \rho; \R^d)$}.
    \end{equation}
    In this case, $w$ is the minimal selection in $\partial \calE_E$.
\end{lemma}

\begin{theorem}[Chain rule]\label{th_chain_rule}
    Let $(\rho, j)\in \mathcal{CE}(0,T)$ be such that
    \[
        \int_0^T \left\{ \calR(\rho_t, j_t) + \calD(\rho_t) \right\} \dd t <\infty\quad\text{and}\quad \sup\nolimits_{t\in[0,T]}\calE(\rho_t)<\infty.
    \]
    Then the map $t \mapsto \calE(\rho_t)$ is absolutely continuous, and
    \[
        -\frac{\dd}{\dd t} \calE(\rho_t) \le \calR(\rho_t, j_t) + \calD(\rho_t)\qquad\text{for almost every $t\in(0,T)$.}
    \]
    In particular, this implies
    $$
        \calI(\rho, j) = \int_0^T \left\{ \calR(\rho_t, j_t) + \calD(\rho_t) \right\} \dd t + \calE(\rho_T) - \calE(\rho_0) \geq 0.
    $$
\end{theorem}
\begin{proof}
From the continuity equation and finiteness of $\int_0^T \calR(\rho_t,j_t)\,\dd t$, we deduce from \cite[Theorem 8.3.1]{ambrosio2008gradient} that $[0,T]\ni t\mapsto \rho_t$ is $W_2$-absolutely continuous (cf.\ Remark~\ref{rem_W2AC}).

Furthermore, it is not difficult to show that the extended functional $\calE_E$ defined above is a regular functional (according to \cite[Definition 10.1.4]{ambrosio2008gradient}) satisfying the properties in \cite[Equations (10.1.1a,b)]{ambrosio2008gradient}. In particular, \cite[E.\ Chain rule in Section~10.1.2 ]{ambrosio2008gradient} applies, i.e.\ we have that
\[
    \frac{\Tilde{\text{d}}}{\dd t}\calE_E(\rho_t) = \int_{\R^d} \left\langle w_t, \frac{\dd j_t}{\dd\rho_t}\right\rangle \dd\rho_t\qquad\text{for all $w_t\in\partial\calE_E(\rho_t)$ and $t\in A$},
\]
where $A\subset(0,T)$ is the set of points satisfying the properties in \cite[(a,b,c) of E.\ Chain rule in Section~10.1.2]{ambrosio2008gradient}. In the following, we show that the set $(0,T)\setminus A$ is $\calL^1$-negligible.

Due to the $\lambda$-convexity of $\rho\mapsto\Ent(\rho|\calL^d)$ w.r.t.\ the $W_2$-metric \cite{ambrosio2008gradient} (see also \cite{mccann1997convexity}), we have that
\[
    t\mapsto \Ent(\rho_t|\calL^d) \qquad\text{is absolutely continuous.}
\]
On the other hand, the Lipschitz continuity of $V$ gives
\[
    \left|\int_\Omega V\,\dd\rho_t - \int_\Omega V\,\dd\rho_s\right| \le \iint_{\Omega\times\Omega} |V(x) - V(y)|\,\pi_s^t (\dd x\dd y) \le \|\nabla V\|_{L^\infty(\Omega)} W_2(\rho_t,\rho_s).
\]
Altogether, we find that
\[
    t\mapsto \calE_E(\rho_t) \qquad\text{is absolutely continuous.}
\]
In particular, the map $t\mapsto \calE_E(\rho_t)$ is differentiable for almost every $t\in(0,T)$.

We now show that $\varrho_t = \dd\rho_t/\dd\calL^d\in W_{loc}^{1,1}(\Omega)$ and that \eqref{eq_vector_field} holds. Notice that if $\calD(\rho_t)<\infty$, then 
\begin{align*}
    \int_B |\nabla \varrho_t|\,\dd x &\le \int_B |\nabla u_t|\,\dd\pi + \int_B u_t|\nabla V|\,\dd\pi \\
    &\le 2\lambda^{-1/2}\sqrt{\rho_t(B)}\sqrt{\calD(\rho_t)} + \|\nabla V\|_{L^\infty(\Omega)}\,\rho_t(B) <\infty\qquad\text{for any Borel set $B\subset\Omega$,}
\end{align*}
thus implying that $|\nabla\varrho_t|\,\calL^d \ll \rho_t=\varrho_t\calL^d$ and $\varrho_t\in W^{1,1}_{loc}(\Omega)$. 
Here we used the fact that
\[
    \calD(\rho) = \int_B \langle \nabla\sqrt{u},\T \nabla\sqrt{u}\rangle\, \dd\pi \ge \lambda \int_B |\nabla\sqrt{u}|^2\,\dd\pi.
\]
We now define
\[
    w_t:= \frac{\nabla\varrho_t}{\varrho_t} + \nabla V_E\qquad\text{$\rho_t$-almost everywhere.}
\]
Then, with a similar computation as above, we obtain
\begin{align*}
    \|w_t\|_{L^2(\rho_t)}^2 &= \int_\Omega \left|\frac{\nabla\varrho_t}{\varrho_t} + \nabla V_E\right|^2 \dd\rho_t = \int_{\{\varrho_t>0\}} \left|\frac{\nabla u_t}{u_t}\right|^2\, \dd\rho_t = 4\int_{\Omega} \left|\nabla \sqrt{u_t}\right|^2 \dd\pi \le 4\lambda^{-1}\calD(\rho_t).
\end{align*}
Hence, Lemma~\ref{lemma_differential} implies that $\rho_t\in \text{dom}(\partial\calE_E)$ and $w_t \in \partial\calE_E(\rho_t)$ is a minimal selection.

We then conclude that $(0,T)\setminus A$ is $\calL^1$-negligible and for all almost every $t\in (0,T)$:
\begin{align*}
    \frac{\dd}{\dd t}\calE_E(\rho_t) &= \int_{\Omega} \left\langle \frac{\nabla\varrho_t}{\varrho_t} + \nabla V, \frac{\dd j_t}{\dd\rho_t}\right\rangle \dd\rho_t = \int_{\Omega} -\left\langle -\mathbb{T}^{1/2} \left(\frac{\nabla u_t}{u_t} \right), \mathbb{T}^{-1/2} \frac{\dd j_t}{\dd \rho_t} \right\rangle \dd \rho_t \\
    &\geq - \frac{1}{2} \int_{\Omega} \left\langle \frac{\dd j_t}{\dd \rho_t}, \mathbb{T}^{-1} \frac{\dd j_t}{\dd \rho_t} \right\rangle \dd \rho_t - \frac{1}{2} \int_{\Omega} \left\langle \frac{\nabla u_t}{u_t}, \mathbb{T} \left(\frac{\nabla u_t}{u_t} \right) \right\rangle \dd \rho_t
    = -\calR(\rho_t,j_t) - \calD(\rho_t).
\end{align*}
We finally obtain the asserted inequality after integrating over time and rearranging the terms.
\end{proof}

\subsection{Proof of Theorem~\ref{th_main_result}}
We now have all the ingredients to summarize the proof of Theorem~\ref{th_main_result}.
\begin{proof}[Proof of Theorem~\ref{th_main_result}]
    Consider a family $\{(\rho^h, j^h)\}_{h>0}$ of GGF-solutions to \eqref{eq_Kolmogorov} according to Definition~\ref{def_GGF_solution}. Let $\{(\hat\rho^h, \hat\jmath^h)\}_{h>0}$ be defined as in \eqref{eq_lift_pair}. Then, the existence of a subsequential limit pair $(\rho, j) \in CE$ and the convergence specified in Theorem~\ref{th_main_result}(1) follow from Lemma~\ref{lemma_properties_flux} and Theorem~\ref{th_compactness}.
    
    The $\liminf$ inequality from assertion (2) is proven in Theorem~\ref{th_liminf_inequalities}, and it immediately follows that $\calI(\rho, j) \leq \liminf_{h\to 0} \calI_h(\rho^h, j^h)= 0$. On the other hand, $\calI(\rho, j) \geq 0$ by the chain rule estimate proven in Theorem~\ref{th_chain_rule}. Therefore, the limit pair $(\rho, j)$ is the  $(\calE, \calR, \calR^*)$-gradient flow solution of \eqref{eq_diffusion2} in the sense of Definition~\ref{def_GF_solution}.
\end{proof}

\appendix
\section{Properties of Gamma-limits as set functions}\label{appendix_proofs}

\begingroup
\def\thetheorem{\ref{prop_properties_F_sup}}
\begin{proposition}
The functional $\calF_{\sup}^\mu$ defined in \eqref{eq_general_functional_def} has the following properties:
    \begin{enumerate}[label=(\roman*)]
        \item Inner regularity:For any $v\in H^1(\Omega, \mu)$ and for any $A\in\calO$ it holds that
        $$
            \sup_{A'\ssubset A} \calF^\mu_{\sup}(v, A') = \calF^\mu_{\sup}(v, A);
        $$
        \item Subadditivity: For any $v\in H^1(\Omega, \mu)$ and for any $A, A', B, B' \in \calO$ such that $A'\ssubset A$ and $B'\ssubset B$ it holds that:
        $$
            \calF^\mu_{\sup}(v, A'\cup B') \leq \calF^\mu_{\sup}(v, A) + \calF^\mu_{\sup}(v, B);
        $$
        \item Locality: For any $A\in \calO$ and any $v, \psi \in H^1(\Omega, \mu)$ such that $v=w$ $\mu$-a.e. on $A$ there holds
        $$
            \calF^\mu_{\sup}(v, A) = \calF^\mu_{\sup}(w, A).
        $$
    \end{enumerate}
\end{proposition}
\addtocounter{theorem}{-1}
\endgroup
\begin{proof}
    {\em (i) Inner regularity.} It is enough to prove that
    $$
        \sup_{A'\ssubset A} \calF^\mu_{\sup}(v, A') \geq \calF^\mu_{\sup}(v, A),
    $$
    because the opposite inequality holds since $\calF^\mu_{\sup}(v, \cdot)$ is an increasing set function. Note that $\calF^\mu_{\sup}(v, A)$ is finite for any $v\in H^1(\Omega, \mu)$ (Lemma~\ref{lemma_general_upper_bound}).
    
    First we need to choose two sequences $(\hat{v}^h)$ and $(\hat{w}^h)$ which converge to $v$ in $L^2(\Omega, \mu)$. To define the first sequence we fix some $\delta>0$ and choose $A'' \ssubset A$ such that
    $$
        \int_{A\backslash \overline{A''}} | \nabla v |^2 \dd \mu \leq \delta.
    $$
    Then by definition of $\calF^\mu_{\sup}$ there exists a sequence $(\hat{v}^h)$ such that
    $$
         \limsup_{h\to 0} \Tilde{\calF}_h^\mu (\hat{v}^h, A\backslash \overline{A''}) = \calF^\mu_{\sup}(v, A\backslash\overline{A''}) \leq C \int_{A\backslash \overline{A''}} | \nabla v |^2 \dd \mu \leq C \delta,
    $$
    where the upper bound was shown in Lemma \ref{lemma_general_upper_bound}. To define the second sequence $(\hat{w}^h)$ let $A'\in\calO$ be such that $A''\ssubset A' \ssubset A$. Again by definition, we find a sequence
    $(\hat{w}^h)$ such that $\hat{w}^h \to v$ in $L^2(\Omega, \mu)$ and
    $$
        \limsup_{h\to 0} \Tilde{\calF}_h^{\mu}(\hat{w}^h, A') = \calF^\mu_{\sup}(v, A').
    $$
    
    Notice that both $\Tilde{\calF}_h^{\mu}(\hat{v}^h, A\backslash \overline{A''})$ and $\Tilde{\calF}_h^{\mu}(\hat{w}^h, A')$ are finite for $h\ll 1$ sufficiently small, which necessarily implies that $\hat{v}^h$ and $\hat{w}^h$ are piecewise constant functions on $\calT^h$. Everywhere in this proof, we use notation with "hats" and superscript $h$ (for example, $\hat{v}^h$) for functions in $\text{PC}(\calT^h)$ and superscript $h$ (for example, $v^h$) for the corresponding discrete functions on $\calT^h$.
    
    Next, we construct a sequence that "interpolates" between $(\hat{v}^h)$ and $(\hat{w}^h)$.
    Set $\varepsilon := \text{dist} \left( A'', A'^c \right)$ and define sets $A_i:= \{ x\in A': ~ \text{dist}(x, A'') < i \varepsilon/N \}$ for $i \in \{1, \dots, N \}$. Note that the following inclusions hold  $A'' \ssubset A_1 \ssubset \dots \ssubset A_N \ssubset A'$. Denote by $\varphi_i^{N}$ a cut-off function between $A_i$ and $A_{i+1}$, i.e.\ $\varphi_i^N\in \calC_c^\infty(A_{i+1})$, $0\le \varphi_i^{N}\le 1$ on $\Omega$, and $\varphi_i^{N}=1$ in a neighborhood of $\overline{A_i}$. with $\| \nabla \varphi_i^{N} \|_{\sup} \leq 2N/\varepsilon$. It has a piecewise constant approximation $\hat{\varphi}^{N,h}_i := \bbL \bbP \varphi_i^{N}$, which satisfies $\hat{\varphi}^{N,h}_i = 1$ on $(A_{i-1})_{\calT^h}$ and $\hat{\varphi}^{N,h}_i = 0$ on $(A\setminus \overline{A_{i+2}})_{\calT^h}$ for $h < h_0^N:=\varepsilon/(3N)$. Now define
    $$
        \hat{w}^{N,h}_i := \hat{\varphi}^{N,h}_i \hat{w}^h + (1 - \hat{\varphi}^{N,h}_i) \hat{v}^h,\qquad i=1,\ldots,N.
    $$
    Observe that since $\hat\varphi^{N,h}_i$ converges pointwisely uniformly to $\varphi_i$ as $h\to 0$, the sequence $(\hat{w}^{N,h}_i)$ still converges to $v$ in $L^2(\Omega, \mu)$ as $h\to 0$ for any $i\in\N$.
    
    For $h<h_0^N$ sufficiently small, the following holds:
    \begin{align*}
        \Tilde{\calF}_h^{\mu} (\hat{w}^{N,h}_i, A) &= \Tilde{\calF}_h^{\mu} (\hat{w}^h, A_{i-1}) + \Tilde{\calF}_h^{\mu} (\hat{v}^h, A\setminus \overline{A_{i+2}}) + \Tilde{\calF}_h^{\mu} (\hat{w}^{N,h}_i, \overline{A_{i+2}}\setminus A_{i-1}) \\
        &\leq \Tilde{\calF}_h^{\mu} (\hat{w}^h, A') + \Tilde{\calF}_h^{\mu} (\hat{v}^h, A\backslash A'') + \Tilde{\calF}_h^{\mu} (\hat w_i^{N,h},G_i^{N,h}),
    \end{align*}
    where
    \[
        G_i^{N,h} :=  \text{int}\,(\overline{A_{i+2}}\setminus A_{i-1}) + B_h(0) = (A_{i+2}\setminus\overline{A_{i-1}}) + B_h(0) \subset A_{i+3}\backslash \overline{A_{i-2}}.
    \]
    We are now left to estimate the last term in the previous inequality. 
    
    We begin by bounding the discrete gradient of $w_i^h$ by
    \begin{align*}
        |\dnabla w_i^{N,h}|(K,L) &= \left| \dnabla (\varphi^{N,h}_i w^h + (1 - \varphi^{N,h}_i) v^h) \right| (K, L) \\
        &\hspace{-2em}= \left|\bigr( w^h(K) - v^h(K)\bigl)\dnabla \varphi^{N,h}_i (K, L) 
         + \varphi^{N,h}_i(L)\dnabla w^h (K, L) + (1 - \varphi^{N,h}_i(L))\dnabla v^h (K, L)  \right| \\
        &\hspace{-2em}\leq | \dnabla \varphi^{N,h}_i| (K, L) | w^h(K) - v^h(K) | + | \dnabla w^h| (K, L)  + |\dnabla v^h| (K, L) .
    \end{align*}
    By Lemma~\ref{lemma_relation_gradients} and since $\| \nabla \varphi \|_{\sup} \leq 2 N/\varepsilon$, then $| \dnabla \varphi^{N,h}_i | \le 2 C_r N h/\varepsilon$, therefore
    \begin{align*}
        &\sum_{(K,L)\in\Sigma^h|_{G_i}} | \dnabla \varphi^{N,h}_i (K, L) |^2  | v^h(K) - w^h(K) |^2 \mu^h(K) \kappa^h(K, L) \\
        &\hspace{6em}\le \frac{4 C_r^2N^2}{\varepsilon^2} h^2 \sum_{(K,L)\in\Sigma^h|_{G_i}} \kappa^h(K, L) \int_K | \hat{v}^h(x) - \hat{w}^h(x) |^2 \mu(\dd x) \\
        &\hspace{6em}\le \frac{4 C_r^2N^2}{\varepsilon^2} C_{\kappa} \| \hat{v}^h - \hat{w}^h \|^2_{L^2(\Omega, \mu)},
    \end{align*}
    where we used the upper bound assumption \eqref{assumption_kernel_uniform_upper_bound}. On the other hand, for any $\eta>0$, we can choose $h=h^{N,\eta}< h_0^N$ such that
    \[
        \Tilde{\calF}_h^{\mu} (\hat w^h,G_i^{N,h}) + \Tilde{\calF}_h^{\mu} (\hat v^h,G_i^{N,h}) \le \Tilde{\calF}_h^{\mu} (\hat w^h,A_{i+2}\setminus\overline{A_{i-1}}) + \Tilde{\calF}_h^{\mu} (\hat v^h,A_{i+2}\setminus\overline{A_{i-1}}) + \eta.
    \]
    In particular, we can choose $\eta=\eta_N$ depending on $N$ such that $\eta_N\to 0$ as $N\to \infty$.
    
    Making use of these estimates gives
    \begin{align*}
        \Tilde{\calF}_h^{\mu} (\hat{w}^{N,h}_i,  G_i^{N,h})
        &\leq 3\left[ \Tilde{\calF}_h^{\mu} (\hat w^h,A_{i+2}\setminus\overline{A_{i-1}}) + \Tilde{\calF}_h^{\mu} (\hat v^h,A_{i+2}\setminus\overline{A_{i-1}}) + \frac{C_N}{\varepsilon^2} \|\hat{v}^h - \hat{w}^h \|^2_{L^2(\Omega, \mu)} + \eta\right],
    \end{align*}
    with $C_N = 4 C_\kappa C_r^2N^2$. Choosing $i(h)\in \{1, \dots, N-3 \}$ such that
    $$
        \Tilde{\calF}_h^{\mu} (\hat w^{N,h}_{i(h)}, A) \leq \frac{1}{N-3} \sum_{j=1}^{N-3} \Tilde{\calF}_h^{\mu} (\hat{w}^{N,h}_j, A),
    $$
    we then obtain
    \begin{align*}
        \Tilde{\calF}_h^{\mu} (\hat{w}^{N,h}_{i(h)}, A) &\leq \Tilde{\calF}_h^{\mu} (\hat w^h, A') + \Tilde{\calF}_h^{\mu} (\hat v^h, A\backslash A'') + \frac{1}{N-3} \sum_{j=1}^{N-3} \Tilde{\calF}_h^{\mu} (\hat{w}^{N,h}_j, G^{N,h}_j).
    \end{align*}
    Combining the estimates together, we have
    \begin{align*}
        \frac{1}{N-3}\sum_{j=1}^{N-3} \Tilde{\calF}_h^{\mu}  (\hat{w}^{N,h}_j,G^{N,h}_j) 
        &\leq 3\left[ \frac{\Tilde{\calF}_h^{\mu} (\hat w^h, A'\backslash A'') + \Tilde{\calF}_h^{\mu} (\hat v^h,  A\backslash A'')}{N-3} + \frac{C_N}{\varepsilon^2} \|\hat{v}^h - \hat{w}^h \|^2_{L^2(\Omega, \mu)} + \eta_N\right] .
    \end{align*}
    Taking the limit superior gives
    \begin{align*}
        \calF^\mu_{\sup}(v, A) &\leq \limsup_{h\to 0} \Tilde{\calF}_h^{\mu} (\hat w^{N,h}_{i(h)}, A) \\
        &\hspace{-2em}\leq \calF^\mu_{\sup}(v, A') + \calF^\mu_{\sup} (v, A\backslash A'') + \frac{3}{N-3} \left[ \limsup_{h\to 0} \Tilde{\calF}_h^{\mu} (\hat w^h, A'\backslash A'') + \calF^\mu_{\sup} (v, A\backslash A'')\right] + 3\eta_N \\
        &\hspace{-2em}\leq \sup_{A'\ssubset A} \calF^\mu_{\sup}(v, A') + C \delta + \frac{3}{N-3} \left[ \limsup_{h\to 0} \Tilde{\calF}_h^{\mu} (\hat w^h, A'\backslash A'') + C \delta \right] + 3\eta_N.
    \end{align*}
    By sending $\delta \to 0$ and $N \to \infty$, we eventually conclude
    $$
        \calF^\mu_{\sup}(v, A) \leq \sup_{A'\ssubset A} \calF^\mu_{\sup}(v, A'),
    $$
    thereby concluding the proof of inner regularity.
    
    {\em (ii) Subadditivity.} The proof follows in a similar fashion as in (1). We begin by choosing two sequences $(\hat{v}^h)$ and $(\hat{w}^h)$ converging to $v$ in $L^2(\Omega,\mu)$ such that
    $$
        \limsup_{h\to 0} \Tilde{\calF}_h^{\mu}(\hat{v}^h, A) = \calF^\mu_{\sup}(v, A) \qquad \text{and} \qquad \limsup_{h\to 0} \Tilde{\calF}_h^{\mu}(\hat{w}^h, B) = \calF^\mu_{\sup}(v, B).
    $$
    Set $\varepsilon := \text{dist} \left( A', A^c \right)$ and define sets $A_i:= \{ x\in A: ~ \text{dist}(x, A') < i \varepsilon/N \}$ for $i \in \{1, \dots, N \}$. Note that the following inclusions hold  $A' \ssubset A_1 \ssubset \dots \ssubset A_N \ssubset A$. Let $\varphi_i^N$ be a cut-off function between $A_i$ and $A_{i+1}$ with $\| \nabla \varphi_i^N \| \leq 2 N/\varepsilon$. We use the piecewise constant approximation $\hat{\varphi}^{N,h}_i$ to define the sequence:
    $$
        \hat{w}^{N,h}_i := \hat{\varphi}^{N,h}_i \hat{v}^h + (1 - \hat{\varphi}^{N,h}_i) \hat{w}^h.
    $$
    For $h\ll 1$ sufficiently small, it holds that
    \begin{align*}
        \Tilde{\calF}_h^\mu (\hat{w}^{N,h}_i, A'\cup B')
        &\leq \Tilde{\calF}_h^\mu (\hat v^h, A) + \Tilde{\calF}_h^\mu (\hat w^h, B) + \Tilde{\calF}_h^\mu(\hat{w}^{N,h}_i, G_i^{N,h}),
    \end{align*}
    with $G_i^{N,h}$ as in (1). The last term on the right-hand side may be estimated as in (1) to obtain
    \begin{align*}
        \Tilde{\calF}_h^\mu(\hat{w}^{N,h}_i, G_i^{N,h})
        &\leq 3\left[ \Tilde{\calF}_h^{\mu} (\hat w^h,A_{i+2}\setminus\overline{A_{i-1}}) + \Tilde{\calF}_h^{\mu} (\hat v^h,A_{i+2}\setminus\overline{A_{i-1}}) + \frac{C_N}{\varepsilon^2} \|\hat{v}^h - \hat{w}^h \|^2_{L^2(\Omega, \mu)} + \eta_N\right].
    \end{align*}
    Choosing $i(h)$ such that
    \begin{align*}
        \Tilde{\calF}_h^\mu (\hat w^{N,h}_{i(h)}, A'\cup B') \leq \frac{1}{N} \sum_{j=1}^N \Tilde{\calF}_h^\mu (\hat w^{N,h}_j, A'\cup B'),
    \end{align*}
    we then obtain
    \begin{align*}
        \Tilde{\calF}_h^\mu (\hat w^{N,h}_{i(h)}, A'\cup B') &\leq \Tilde{\calF}_h^\mu (\hat v^h, A) + \Tilde{\calF}_h^\mu (\hat w^h, B) + \frac{1}{N} \sum_{j=1}^N \Tilde{\calF}_h^\mu(\hat w^{N,h}_j, G_j^{N,h}),
    \end{align*}
    where the last term may be estimated by
    \begin{align*}
        \frac{1}{N}\sum_{j=1}^N \Tilde{\calF}_h^\mu(\hat w^{N,h}_j, G_j^{N,h})
        &\leq 3 \left[ \frac{\Tilde{\calF}_h^\mu(\hat w^h, A \backslash A') + \Tilde{\calF}_h^\mu(\hat v^h, A \backslash A')}{N}
        + \frac{C_N}{\varepsilon^2} \| \hat{v}^h - \hat{w}^h \|^2_{L^2(\Omega, \mu)} + \eta_N\right].
    \end{align*}
    Taking the limit superior as $h\to 0$ gives
    \begin{align*}
        \calF^\mu_{\sup}(v, A'\cup B') &\leq \limsup_{h\to 0} \Tilde{\calF}_h^\mu (\hat w^{N,h}_{i(h)}, A'\cup B') \\
        &\leq \calF^\mu_{\sup} (v, A) + \calF^\mu_{\sup} (v, B) + o(1)|_{N\to \infty}.
    \end{align*}
    By sending $N \to \infty$ and applying the inner regularity property, we conclude
    $$
        \calF^\mu_{\sup}(v, A\cup B) \leq \calF^\mu_{\sup} (v, A) + \calF^\mu_{\sup} (v, B).
    $$
    
    {\em (iii) Locality.} We first prove $\calF^\mu_{\sup}(v, A) \leq \calF^\mu_{\sup}(w, A)$. The argument is similar to the  previous points. For a fixed $\delta>0$ there exists $A_\delta \ssubset A$ such that $\int_{A\backslash \overline{A_\delta}} |\nabla v|^2 \dd\mu < \delta$. We choose two sequences $(\hat{v}^h)$ and $(\hat{w}^h)$ such that $\hat{v}^h\to v$, $\hat{w}^h\to w$ in $L^2(\Omega, \mu)$ satisfying
    \begin{align*}
        \limsup_{h\to 0} \Tilde{\calF}^{\mu}_h (\hat{v}^h, A\backslash \overline{A_\delta}) &= \calF^\mu_{\sup}(v, A) \leq C \int_{A\backslash \overline{A_\delta}} |\nabla v|^2 \dd\mu < C \delta, \\
        \limsup_{h\to 0} \Tilde{\calF}^{\mu}_h (\hat{w}^h, A) &= \calF^\mu_{\sup}(w, A).
    \end{align*}
    Set $\varepsilon := \text{dist} \left( A_\delta, A^c \right)$ and define sets $A_i:= \{ x\in A: ~ \text{dist}(x, A_\delta) < i \varepsilon/N \}$ for $i \in \{1, \dots, N \}$. Denote by $\varphi_i^N$ a cut-off function between $A_i$ and $A_{i+1}$ with $\| \nabla \varphi_i^N \|_{\sup} \leq 2 N/\varepsilon$. It has a piecewise constant approximation $\hat{\varphi}^{N,h}_i := \bbL \bbP \varphi_i^N$.
    Then, we define
    $$
        \hat{w}^{N,h}_i := \hat{\varphi}^{N,h}_i \hat{w}^h + (1 - \hat{\varphi}^{N,h}_i) \hat{v}^h,
    $$
    with $(\hat{w}^{N,h}_i)$ still converging to $w$ in $L^2(\Omega, \mu)$ as $h\to0$ for any $i=1,\ldots,N$. Similar to the proof of inner regularity, we obtain the existence of some $i(h)\in \{1,\ldots,N\}$ such that
    \begin{align*}
        \limsup_{h\to 0}\Tilde{\calF}^\mu_h (\hat{w}^{N,h}_{i(h)},A) 
        &\leq \calF^\mu_{\sup}(w, A) + O(\delta)|_{\delta\to 0} + o(1)|_{N\to\infty}.
    \end{align*}
    Passing to the limits $\delta\to 0$ and $N\to \infty$ then yields
    $$
        \calF^\mu_{\sup}(v, A) \leq \calF^\mu_{\sup}(w, A).
    $$
    The assertion follows as we can swap the roles of $v$ and $w$.
\end{proof}

We recall that 
$$
    \calF^{\pi, \varphi}(v, A) := \begin{cases}
            \calF^\pi(v, A) & \text{if } v - \varphi\in H^1_0(A), \\
            +\infty & \text{otherwise,}
        \end{cases}
$$
and the corresponding discrete counterpart of $\calF^{\pi, \varphi}$ is
$$
    \calF^{\pi, \varphi}_h(v^h, A) := \begin{cases}
            \calF^\pi_h(v^h, A) & \text{if } v^h = \bbP \varphi =: \varphi^h \text{ on }  \calT^h|_{A^c}, \\
            +\infty & \text{otherwise.}
        \end{cases}
$$

\begingroup
\def\thetheorem{\ref{prop_Gamma_convegence_bc}}
\begin{proposition}
Let $A\in\calO$ be arbitrary with Lipschitz boundary and $\varphi\in H^1(\Omega)$. For any sequence $(\calF^{\pi, \varphi}_h(\cdot, A))$ there exists a subsequence that $\Gamma$-converges in the $L^2(\Omega)$-topology to $\calF^{\pi, \varphi}(\cdot, A)$.
\end{proposition}
\addtocounter{theorem}{-1}
\endgroup
\begin{proof}
Let us first prove the $\Gamma$-$\liminf$ inequality. We consider a sequence $\bbL v^h \to v$ in $L^2(\Omega)$ such that $\sup_{h>0} \calF^{\pi, \varphi}_h(v^h, A) < \infty$. This implies that $v^h = \varphi^h$ on $\calT^h|_{A^c}$ and  $\calF^{\pi, \varphi}_h(v^h, A) = \calF^\pi_h(v^h, A)$. Consequently, we also have that $\sup_{h>0 }\calF^\pi_h(v^h, A) < \infty$. By the same argument as in Lemma~\ref{lemma_bounds_Dsup}, we deduce that $v\in H^1(A)$. Since $\Gamma$-$\lim \calF^\pi_h(\cdot, A) = \calF^\pi (\cdot, A)$ it remains to prove that $v - \varphi \in H_0^1(A)$.

Notice that $A^c\subset A^c_{\calT^h}$ for all $h>0$. Since $v^h=\varphi^h$ on $\calT^h|_{A^c}$, their piecewise reconstructions satisfy $\bbL v^h = \bbL\varphi^h$ on $A^c_{\calT^h}$, and hence also on $A^c$ for all $h>0$. Using the fact that $\bbL v^h\to v$ and $\bbL \varphi^h\to \varphi$ in $L^2(\Omega)$, we easily deduce that $v=\varphi$ in $L^2(\Omega\setminus A)$. The deduced regularity $v\in H^1(A)$ and assumed regularity $\varphi\in H^1(\Omega)$ then allows to conclude that $v-\varphi\in H_0^1(A)$.
 
Thus, the $\liminf$ inequality follows:
$$
    \liminf_{h\to 0} \calF^{\pi, \varphi}_h(v^h, A) \geq \calF^{\pi, \varphi} (v, A)\,. 
$$
    
Now we show the approximate $\limsup$ inequality. Let $v \in H^1(\Omega)$ such that $\supp (v - \varphi) \ssubset A$. There exists a recovery sequence $v^h \to v$ in $L^2(\Omega)$ such that $\displaystyle\lim_{h\to 0} \calF^\pi_h(v^h, A) = \calF^\pi(v, A) = \calF^{\pi, \varphi} (v, A)$.
    
Set $\varepsilon := \text{dist} \left( \supp(v-\varphi), A^c \right)$ and define sets $A_i:= \{ x\in A: ~ \text{dist}(x, \supp(v-\varphi)) < i \varepsilon/N \}$ for $i \in \{1, \dots, N \}$. Denote by $\varphi_i^N$ a cut-off function between $A_i$ and $A_{i+1}$ with $\| \nabla \varphi_i^N \|_{\sup} \leq 2 N/\varepsilon$. It has a piecewise constant approximation $\hat{\varphi}^{N,h}_i := \bbL \bbP \varphi_i^N$.
Then, we define
$$
    \hat{w}^{N,h}_i := \hat{\varphi}^{N,h}_i \hat{v}^h + (1 - \hat{\varphi}^{N,h}_i) \varphi,
$$
with $(\hat{w}^{N,h}_i)$ still converging to $v$ in $L^2(\Omega)$ as $h\to0$ for any $i=1,\ldots,N$. Similar to the proof of inner regularity, we obtain the existence of some $i(h)\in \{1,\ldots,N\}$ such that
$$
    \limsup_{h\to 0} \Tilde{\calF}^\pi_h(\hat{w}^{N,h}_{i(h)}, A) \leq \limsup_{h\to 0} \Tilde{\calF}^\pi_h(v^h, A) + o(1)|_{N\to\infty}.
$$
Passing $N\to \infty$ yields
$$
    \limsup_{h\to 0} \Tilde{\calF}^{\pi,\varphi}_h(\hat{w}^{N,h}_{i(h)}, A) \leq \calF^{\pi,\varphi} (v, A).
$$

\end{proof}

\bibliographystyle{plain}
\bibliography{ref}

\end{document}